\documentclass[11pt]{article}
\usepackage{inputenc}
\usepackage[colorlinks=true, bookmarksnumbered=true, bookmarksopen=true,
bookmarksopenlevel=3, pdfstartview=FitH, linkcolor=cyan, pdfmenubar=true,
pdftoolbar=true, bookmarks=true,citecolor=cyan, urlcolor=magenta,
filecolor=cyan,menucolor=black,plainpages=false,pdfpagelabels]{hyperref}
\usepackage[lmargin=2cm,tmargin=2cm,bmargin=2cm,rmargin=2cm]{geometry}
\usepackage{lineno,hyperref}
\usepackage{amsfonts}
\usepackage{amsmath,amssymb}
\usepackage{mathtools}
\usepackage{color}
\usepackage{amsmath,amssymb}
\usepackage{amsthm}
\usepackage{amsfonts}
\usepackage{tikz}
\usepackage{pstricks}
\usepackage{pstricks-add}
\usepackage{pst-all,pst-math,pst-plot}
\usepackage{textcomp}
\usepackage{graphicx,color}
\usepackage{graphicx}
\usepackage{fancyhdr,ifthen}

\usetikzlibrary{matrix}

\newtheorem{theorem}{Theorem}[section]
\newtheorem{lemma}[theorem]{Lemma}
\newtheorem{corollary}[theorem]{Corollary}
\newtheorem{proposition}[theorem]{Proposition}

\newtheorem{remark}[theorem]{Remark}

\numberwithin{equation}{section}

\newcommand{\R}{\mathbb{R}}

\allowdisplaybreaks

\begin{document}

\title{Remarks on a multi-dimensional transport equation with \\nonlocal velocity\\
\vspace{0.4cm}}
\author{
{Wanwan Zhang$^1$} {\thanks{{Corresponding author. }\newline
{E-mail addresses: zhangww@jxnu.edu.cn (W. W. Zhang).}
}}\\
\\
{\small $^{1}$ School of Mathematics and Statistics, Jiangxi Normal University, } \\
{\small {Nanchang, 330022, Jiangxi, P. R. China.}}
}
\date{}
\maketitle

\begin{abstract}
In this paper, we revisit a multi-dimensional nonlocal transport equation of the form
\begin{eqnarray*}
\partial_t\rho+g\mathcal{R}_a\rho\cdot \nabla\rho= 0,~\rho(\cdot,0)=\rho_{0},
\end{eqnarray*}
where the transform $\mathcal{R}_a$ is defined by
\begin{eqnarray*}
\mathcal{R}_a\rho(x,t)=\frac{\Gamma(\frac{n+1}{2})}{\pi^{\frac{n+1}{2}}}P.V.\int\limits_{\mathbb{R}^n}\Big(\frac{x-y}{|x-y|^{n+1}}-\frac{x-y}{(|x-y|^2+a^2)^{\frac{n+1}{2}}}\Big)\rho(y,t)dy.
\end{eqnarray*}
This model was first introduced in [W. Zhang, J. Differential Equations 467 (2026) 114275] as a natural generalization of the well-known Kiselev-Sarsam equation proposed in [A. Kiselve and N.A. Sarsam, Nonlinearity 38 (2025) 055018].
By using an improved pointwise inequality for the transform $\mathcal{R}_a$, we obtain a better and more natural version of a nonlinear integral inequality proved in Zhang \cite{[Zhang-JDE2026]} for the transform $\mathcal{R}_a$. In addition, we also establish a new weighted nonlinear inequality for the transform $\mathcal{R}_a$ implying the finite-time blow-up for a class of smooth radial initial data.
\

\vspace{0.2cm}


\vspace{0.2cm}

\end{abstract}

\tableofcontents
\
\setlength{\headheight}{14.5pt} \setlength{\headwidth}{\textwidth} %
\pagestyle{fancy} \fancyhf{} \renewcommand{\headrulewidth}{0pt}
\chead{\ifthenelse{\isodd{\value{page}}}{A multi-dimensional transport with nonlocal velocity}{Wanwan Zhang}} \rhead{\thepage}
\newpage
\section{Introduction and notations}
\hskip\parindent
In this paper, we study a multi-dimensional nonlocal active scalar equation of the form
\begin{align}\label{M-KS}
\begin{cases}
\partial_t\rho+g\mathcal{R}_a\rho\cdot \nabla\rho= 0, &\, (x,t)\in \mathbb{R}^n\times\mathbb{R}_+,\\
\rho(x,0)=\rho_{0}(x),&\, x\in \mathbb{R}^n,
\end{cases}
\end{align}
which was first introduced in \cite{[Zhang-JDE2026]}.
Here $a,g>0$ are fixed constants, $n\geq2$ is the space dimension, and
the transform $\mathcal{R}_a$ is defined by
\begin{eqnarray*}
\mathcal{R}_a\rho(x,t)=\frac{\Gamma(\frac{n+1}{2})}{\pi^{\frac{n+1}{2}}}P.V.\int\limits_{\mathbb{R}^n}\Big(\frac{x-y}{|x-y|^{n+1}}-\frac{x-y}{(|x-y|^2+a^2)^{\frac{n+1}{2}}}\Big)\rho(y,t)dy.
\end{eqnarray*}
The one-dimensional version of the equation \eqref{M-KS} is the well-known Kiselev-Sarsam equation given by
\begin{align}\label{KS}
\begin{cases}
\partial_t\rho+g\mathcal{H}_a\rho\partial_x\rho= 0, &\, (x,t)\in \mathbb{R}\times\mathbb{R}_+,\\
\rho(x,0)=\rho_{0}(x),&\, x\in \mathbb{R},
\end{cases}
\end{align}
where the transform $\mathcal{H}_a$ is defined by
\begin{eqnarray*}
\mathcal{H}_a\rho(x,t)
=\frac{1}{\pi}P.V.\int\limits_{\mathbb{R}}\frac{a^2\rho(y,t)}{(x-y)((x-y)^2+a^2)}dy.
\end{eqnarray*}
The one-dimensional nonlocal transport equation \eqref{KS} was introduced by Kiselev and Sarsam in \cite{[Kiselev-Sarsam]} as a one-dimensional model analogy for the two-dimensional incompressible porous media (IPM) equation given by
\begin{align*}
\begin{cases}
\partial_t\rho+u\cdot\nabla\rho= 0, &\, (x,t)\in \mathbb{R}^2\times\mathbb{R}_+,\\
u=-\nabla P-(0,g\rho),~\nabla\cdot u=0,\\
\rho(x,0)=\rho_{0}(x),&\, x\in\mathbb{R}^2,
\end{cases}
\end{align*}
which models the transport of a scalar density $\rho(x,t)$ by an incompressible fluid velocity field $u(x,t)$ under the effects of Darcy's law and gravity. Here $P=P(x,t)$ is the scalar pressure, $g>0$ is the constant of gravitational acceleration, and $\rho_{0}=\rho_{0}(x)$ is the initial density.
Here we briefly review some recent results for the IPM equation.
The local well-posedness has been established in various function spaces (see \cite{[Cordoba-Gancedo-Orive]} and \cite{[Constantin]}).
Bianchini, C\'{o}rdoba and Martn\'{i}nez-Zoroa \cite{[Bianchini-Cordoba-Martninez]} proved both ill-posedness for IPM and strong ill-posedness for perturbations of a spectrally stable steady state at the critical regularity.
Kiselev and Yao \cite{[Kiselev-Yao]} studied the small scale formation for the IPM equation provided global-in-time existence.
The question of finite-time singularity versus global regularity for IPM on the whole plane is still open.
There have been many results towards this problem.
C\'{o}rdoba and Mart\'{i}nez-Zoroa \cite{[Cordoba-Martninez-Zoroa]} showed the finite time singularities for the IPM equation with a smooth source.
Zlato\v{s} \cite{[Zlatos-Local-regularity],[Zlatos-singularity-formation]} proved the finite-time singularity formation of the Muskat equation on the half-plane for  smooth initial data.
Very recently, Dembski \cite{[Dembski]} proved the finite-time singularity of Lipschitz continuous solutions to the IPM equation vanishing on the boundary of the domains arbitrarily close to the half-plane.

We proceed to recall the known results on the Kiselev-Sarsam equation and related models.
Kiselev and Sarsam \cite{[Kiselev-Sarsam]} derived the one-dimensional model  \eqref{KS} from the IPM equation in the periodic half-plane and also remarked how the 1-D IPM  equation \eqref{KS} parallels the well-known C\'{o}rdoba-C\'{o}rdoba-Fontelos (CCF) equation
\begin{align}\label{CCF}
\begin{cases}
\partial_t\rho-\mathcal{H}\rho\partial_x\rho=0, &\, (x,t)\in \mathbb{R}\times\mathbb{R}_+,\\
\rho(x,0)=\rho_{0}(x),&\, x\in \mathbb{R},
\end{cases}
\end{align}
where $\mathcal{H}$ is the Hilbert transform defined by
\begin{eqnarray*}
\mathcal{H}\rho(x,t)
=\frac{1}{\pi}P.V.\int\limits_{\mathbb{R}}\frac{\rho(y,t)}{x-y}dy.
\end{eqnarray*}
The multi-dimensional generalization of the CCF equation \eqref{CCF} given by
\begin{align}\label{M-CCF}
\begin{cases}
\partial_t\rho-\mathcal{R}\rho\cdot\nabla\rho= 0, &\, (x,t)\in \mathbb{R}^n\times\mathbb{R}_+,\\
\rho(x,0)=\rho_{0}(x),&\, x\in \mathbb{R}^n,
\end{cases}
\end{align}
was first investigated in \cite{[Balodis-Cordoba]} by Balodis and C\'{o}rdoba.
Here $\mathcal{R}$ is the Riesz transform defined by
\begin{eqnarray*}
\mathcal{R}\rho(x,t)=\frac{\Gamma(\frac{n+1}{2})}{\pi^{\frac{n+1}{2}}}P.V.\int\limits_{\mathbb{R}^n}\frac{x-y}{|x-y|^{n+1}}\rho(y,t)dy.
\end{eqnarray*}

Since the equations \eqref{M-KS} and \eqref{KS} are intimately related to the equations \eqref{CCF} and \eqref{M-CCF}, we first review the existing results on the equations \eqref{CCF} and \eqref{M-CCF} for clear motivation of the present work.
C\'{o}rdoba-C\'{o}rdoba-Fontelos \cite{[Cordoba-Cordoba-Fontelos05]} first showed the finite-time singularity formation of solutions to \eqref{CCF} for a class of smooth even initial data.
Specifically, the blow-up proof in \cite{[Cordoba-Cordoba-Fontelos05]} was based on an ingenious inequality for the Hilbert transform $\mathcal{H}$: for any $-1<\delta<1$ and any even bounded smooth function $f$ defined on $\mathbb{R}$,
\begin{eqnarray}\label{CCF-inequality}
-\int\limits^\infty_0\frac{\mathcal{H} f(x) f'(x)}{x^{1+\delta}}dx
\geq
C_{\delta}\int\limits^\infty_0\frac{(f(x))^2}{x^{2+\delta}}dx,
\end{eqnarray}
where  $C_\delta>0$ is a constant depending only on $\delta$.  The proof of \eqref{CCF-inequality} in \cite{[Cordoba-Cordoba-Fontelos05]} was based on the Meillin transform and complex analysis.
Based on completely real variable arguments, Kiselev \cite{[Kiselev]} proved that more general inequality below
\begin{eqnarray}\label{nonlinear-inequality-H-Kiselev}
-\int\limits^1_0\frac{\mathcal{H} f(x) f'(x)(f(x))^{p-1}}{x^\delta}dx
\geq
C_{p,\delta}\int\limits^1_0\frac{(f(x))^{p+1}}{x^{1+\delta}}dx
\end{eqnarray}
holds true for any $p\geq1,\delta>0$ and any even bounded $C^1$ function $f$ with $f(0)=0$ and $f'\geq0$ on $(0,\infty)$.
Li and Rodrigo \cite{[Li-Rodrigo20]} established a pointwise inequality for the Hilbert transform $\mathcal{H}$ acting on even and non-decreasing (on $(0,\infty)$) function $f$ on $\mathbb{R}$:
\begin{eqnarray*}
-\mathcal{H}f(x)\geq\frac{2}{\pi }\cdot\frac{1}{x}\int\limits^x_0(f(x)-f(y))dy,
\end{eqnarray*}
and then gave a short proof of the inequalities \eqref{CCF-inequality} and \eqref{nonlinear-inequality-H-Kiselev} by a simple application of Hardy's inequality.
The authors \cite{[Li-Rodrigo20]} also established the integral inequality below for the Hilbert transform $\mathcal{H}$ : for any even Schwartz function $f$,
\begin{eqnarray}\label{exponential-Hilbert}
-\int\limits^\infty_0f'(x)\mathcal{H}f(x)\frac{e^{-x}}{x}dx
\geq \frac{1}{2\pi}
\int\limits\limits^\infty_0\frac{(f(0)-f(x))^2}{x^2}dx
-1000\|f\|^2_{L^\infty}.
\end{eqnarray}
We also refer the readers to \cite{[Castro-Cordoba],[Cordoba-Cordoba-Fontelos06],[Silvestre-Vicol]}
for different and elegant blow-up proofs for the CCF equation \eqref{CCF}.

For the multi-dimensional CCF equation \eqref{M-CCF},
Balodis and C\'{o}rdoba \cite{[Balodis-Cordoba]}
proved the finite-time blow-up of smooth solutions to \eqref{M-CCF} by applying a multi-dimensional version of the integral inequality \eqref{CCF-inequality}: for any $C_n<\delta<1$ and a suitable smooth function $f$ with constant sign vanishing at the origin,
\begin{eqnarray*}
-\int\limits_{\R^n}\frac{(\mathcal{R} f(x)-\mathcal{R} f(0))\cdot\nabla f(x)}{|x|^{n+\delta}}dx
\geq
C_n\int\limits_{\R^n}\frac{(f(x))^2}{|x|^{n+1+\delta}}dx,
\end{eqnarray*}
where $0<C_n<1$ is a fixed constant.
When $n=2$, such blow-up result was also proved for a similar equation in \cite{[Dong-Li]} independently.
Later, the transport equation with  fractional velocity given by
\begin{align}\label{M-CCF-generalized}
\begin{cases}
\partial_t\rho+\Lambda^{-2+2\alpha}\nabla\rho\cdot \nabla\rho= 0, &\, (x,t)\in \mathbb{R}^n\times\mathbb{R}_+,\\
\rho(x,0)=\rho_{0}(x),&\, x\in \mathbb{R}^n,
\end{cases}
\end{align}
was also studied. Here the space dimension $n\geq2$ and $0<\alpha<1$.
The equation \eqref{M-CCF-generalized} with $\alpha=\frac12$ is reduced to the multi-dimensional CCF equation \eqref{M-CCF}.
Chae \cite{[Chae]} established the local well-posedness for \eqref{M-CCF-generalized} in the Sobolev space.
Dong \cite{[Dong]} obtained the weighted nonlinear inequality for all $\alpha\in (0,1)$: for any $\delta\in(-2\alpha,2-2\alpha)$ and radial Schwartz function $f$,
\begin{eqnarray}\label{nonlinear-inequality}
\int\limits_{\R^n}\frac{\Lambda^{-2+2\alpha}\nabla f(x)\cdot\nabla f(x)}{|x|^{n+\delta}}dx
\geq
C_{n,\alpha,\delta}\int\limits_{\R^n}\frac{(f(0)-f(x))^2}{|x|^{n+2\alpha+\delta}}dx,
\end{eqnarray}
which was applied to prove the blow-up of smooth solutions to \eqref{M-CCF-generalized}  for a class of smooth and radially symmetric initial data.
Motivated by \cite{[Silvestre-Vicol]}, Jiu and Zhang \cite{[Jiu-Zhang]} also proved the finite-time singularity of solutions to \eqref{M-CCF-generalized} for smooth initial data $\theta_0$ with $\displaystyle\sup_{x\in\R^n}\theta_0(x)>0$ via the De Giorgi iteration technique.
Under the additional non-increasing assumption of radial $f$,  Li and Rodrigo \cite{[Li-Rodrigo20]} proved the inequality \eqref{nonlinear-inequality} by deriving the pointwise inequality below: for any $x\neq0$,
\begin{align*}
-\Lambda^{-2+2\alpha}\nabla f(x)\cdot\frac{x}{|x|}
\geq \frac{C_{\alpha,n}}{|x|^{n-1+2\alpha}}\int\limits_0^{|x|} (-f'(\varrho))\varrho^n
d\varrho,
\end{align*}
along with the use of the Hardy's inequality.
The multi-dimensional extension of the one-dimensional inequality \eqref{exponential-Hilbert} to the Riesz transform $\mathcal{R}$ was established in \cite{[Zhang-JAM27]}:
for any radial Schwartz function $f:\R^n\rightarrow\R$, we have
\begin{eqnarray}\label{Riesz-transform-integral-inequality}
-\int\limits_{\R^n}\frac{\mathcal{R} f(x)\cdot\nabla f(x)}{|x|^{n}}e^{-|x|}dx
\geq \frac{1}{2nB(\frac12,\frac{n}{2})}
\int\limits_{\R^n}\frac{(f(0)-f(x))^2}{|x|^{n+1}}dx
-C_{n}\|f\|^2_{L^\infty}.
\end{eqnarray}
Finally, the finite-time blow-up problem of the fractionally dissipative equations of \eqref{CCF} and \eqref{M-CCF} in the supercritical scheme was also investigated in the literature.
It was proved that certain smooth solution to the equation
\begin{align}\label{CCF-generalized-dissipation}
\begin{cases}
\partial_t\rho-\mathcal{H}\rho\partial_x\rho+\Lambda^\gamma\rho= 0~ {\rm or}~\partial_t\rho-\mathcal{R}\rho\cdot\nabla\rho+\Lambda^\gamma\rho=0,\\
\rho(x,0)=\rho_{0}(x),
\end{cases}
\end{align}
develop finite-time blow-up for $\gamma\in(0,\frac12)$. We refer the readers to \cite{[Kiselev],[Li-Rodrigo08],[Li-Rodrigo09],[Li-Rodrigo20],[Silvestre-Vicol],[Zhang]} and the
references therein for more details.
In the case of $\gamma\in[\frac12,1)$, whether solutions of  \eqref{CCF-generalized-dissipation} with smooth initial data may blow up in finite time remains still open.

Now we proceed to review the known results on the Kiselev-Sarsam equation \eqref{KS} and its multi-dimensional generalization \eqref{M-KS}.
Kiselev and Sarsam \cite{[Kiselev-Sarsam]} proved the local well-posedness for the equation \eqref{KS} posed on the circle and adapted the arguments for the Hilbert transform in \cite{[Kiselev]} to show that, for any $a,\delta>0$ and any  nonnegative even smooth function $f$  defined on $\mathbb{T}$  with $f(0)=0$ and $f'\geq0$ on $[0,\pi)$,
\begin{eqnarray}\label{IPM-inequality}
-\int\limits_0^{\frac\pi2}\frac{\mathcal{H}_af(x)f'(x)}{x^\delta}dx\geq C_{a,\delta}\int\limits_0^{\frac\pi2}\frac{(f(x))^2}{x^{1+\delta}}dx,
\end{eqnarray}
where $C_{a,\delta}$ is a universal constant depending only on $\delta$ and $a$.  As an application of the inequality \eqref{IPM-inequality}, Kiselev and Sarsam proved the finite-time singularity of solutions to \eqref{KS} for a class of  smooth even initial data in the setting of the periodic circle. Recently, Liu and Zhang \cite{[Liu-Zhang]} established several weighted integral inequalities for the transform $\mathcal{H}_a$ in the setting of the real line. Based on these integral inequalities, the authors proved the finite-time blow-up of solutions to \eqref{KS}. Specifically, the authors obtained a pointwise inequality for the transform $\mathcal{H}_a$ acting on even and non-decreasing (on $(0,\infty)$) function $f$ on $\mathbb{R}$:
\begin{eqnarray}\label{pointwise-inequality-Ha}
-\mathcal{H}_a f(x)\geq\frac{2}{\pi x}\Big(1-\frac{x^2}{x^2+a^2}\Big)\int\limits^x_0(f(x)-f(y))dy,
\end{eqnarray}
which can be implied to show the nonlinear weighted inequality for the transform $\mathcal{H}_a$: for any $\delta\in(-1,1)$ and even and non-decreasing (on $(0,\infty)$) function $f$ on $\mathbb{R}$,
\begin{eqnarray*}
-\int\limits^\infty_0\frac{\mathcal{H}_a f(x)f'(x)}{x^{1+\delta}}dx
\geq \frac{a^2}{\pi}(3+\delta-2\sqrt{2+\delta})\int\limits^\infty_0\frac{(f(x)-f(0))^2}{x^{2+\delta}(x^2+a^2)}dx.
\end{eqnarray*}
The inequality \eqref{exponential-Hilbert} for the Hilbert transform $\mathcal{H}$ was also extended to the transform $\mathcal{H}_a$ in \cite{[Liu-Zhang]}: for any even Schwartz function,
\begin{eqnarray}\label{exponentially-Ha}
-\int\limits^\infty_0\frac{\mathcal{H}_af(x)f'(x)}{x}e^{-x}dx
\geq \frac{1}{2\pi}
\int\limits\limits^\infty_0
\frac{(f(0)-f(x))^2}{x^2}\cdot\frac{a^2(3x^2+a^2)}{(x^2+a^2)^2}dx
-\frac{2}{\pi}\Big(1911+\frac{3}{a^2}\Big)\|f\|^2_{L^\infty}.
\end{eqnarray}
 Very recently, by extending the one-dimensional transform $\mathcal{H}_a$ to the multi-dimensional transform $\mathcal{R}_a$ in a natural way, Zhang \cite{[Zhang-JDE2026]} proposed the multi-dimensional Kiselev-Sasarm equation \eqref{M-KS} and established a pointwise inequality for the transform $\mathcal{R}_a$ acting on the radially non-decreasing function $f$:
\begin{eqnarray}\label{pointwise-inequality-JDE}
-\mathcal{R}_a f(x)\cdot\frac{x}{|x|}
\geq\frac{nB(\frac12,\frac{n+1}{2})}{2^{n+1}\pi}\frac{1}{|x|^n}\Big[1-\frac{2^{n+1}|x|^{n+1}}{(4|x|^2+a^2)^{\frac{n+1}{2}}}\Big]
\int\limits_0^{|x|} (f(|x|)-f(\varrho))\varrho^{n-1}
d\varrho,
\end{eqnarray}
which can be viewed as the multi-dimensional version of the inequality \eqref{pointwise-inequality-Ha} for the transform $\mathcal{H}_a$.
Using the pointwise inequality \eqref{pointwise-inequality-JDE}, the author \cite{[Zhang-JDE2026]} obtained a weighted nonlinear inequality for the transform $\mathcal{R}_a$,
\begin{eqnarray}\label{weighted-Ra-monotone}
-\int\limits_{\mathbb{R}^n}\frac{\mathcal{R}_a f(x)\cdot\nabla f(x)}{|x|^{n+\delta}}dx
\geq \frac{(\sqrt{n+1+\delta}-\sqrt{n})^2}{2^{n+2}\pi (B(\frac12,\frac{n+1}{2}))^{-1}}
\int\limits_{\mathbb{R}^n}\frac{(f(x)-f(0))^2}{|x|^{n+1+\delta}}\Big[1-\frac{2^{n+1}|x|^{n+1}}{(4|x|^2+a^2)^{\frac{n+1}{2}}}\Big]dx,
\end{eqnarray}
which was applied to show that smooth solution to \eqref{M-KS} must blow up in finite time for the compactly surpported and radially non-decreasing initial data $\rho_0$ with $\rho_0(0)<0$.

It is worth mentioning that, compared to the pointwise inequality \eqref{pointwise-inequality-Ha} for the one-dimensional transform $\mathcal{H}_a$, the inequality \eqref{pointwise-inequality-JDE} for the multi-dimensional transform $\mathcal{R}_a$ is not particularly satisfactory.
The purpose of this paper is mainly two fold. First, we first establish a more natural pointwise inequality for the transform $\mathcal{R}_a$ (see Proposition \ref{lower-bound-Ra}) and then obtain an improved nonlinear weighted inequality (see Proposition \ref{n-nonlinear-inequality-monotone}). Second, we extend the one-dimensional nonlinear inequality \eqref{exponentially-Ha} for the transform $\mathcal{H}_a$ to the multi-dimensional case and obtained a new nonlinear weighted inequality for the transform $\mathcal{R}_a$  implying the finite-time blow-up of smooth solution to \eqref{M-KS} for a class of radial initial data (see Proposition \ref{nonlinear-weighted-inequality-exponential}).

At the end of this section, some notations are introduced as follows.
Throughout this paper, we will use $C$ to denote a positive constant, whose value may change from line to line, and write $C_{n,a}$ or $C(n,a)$ to emphasize the dependence of a constant on $n$ and $a$.
For $p\in[1,\infty]$, we denote $L^p(\mathbb{R}^n)$ the standard Lebesgue space and its  norm by  $\|\cdot\|_{L^p(\mathbb{R}^n)}$.
For $s>0$, we use the notation $H^s(\mathbb{R}^n)$ to denote the nonhomogeneous Sobolev space of $s$ order,
whose endowed norm is denoted by
\begin{align*}
\|f\|_{H^s(\mathbb{R}^n)}:=\|f\|_{L^2(\mathbb{R}^n)}+\|(-\Delta)^{\frac s2}f\|_{L^2(\mathbb{R}^n)},
\end{align*}
where the fractional Laplacian $(-\Delta)^{\frac s2}$ is defined through the Fourier transform as (see, e.g., \cite{[Stein1970]})
\begin{eqnarray*}
\widehat{(-\Delta)^{\frac s2}f}(\xi)=(2\pi|\xi|)^s\widehat{f}(\xi).
\end{eqnarray*}
For a sake of the convenience, the $L^{p}(\mathbb{R}^n)$-norm of a function $f$ is sometimes abbreviated as $\|f\|_{L^p}$ and  the $H^s(\mathbb{R}^n)$-norm as $\|f\|_{H^s}$.
All norms of a function $f(x,t)$ depending on space and time variables will refer to the spatial norms.
The functions $\Gamma(\cdot)$ and $B(\cdot,\cdot)$ stand for the standard Gamma and Beta function, respectively.
Let $\mathbb{S}^{n-1}$ be the unit sphere in $\mathbb{R}^n$, i.e., $\mathbb{S}^{n-1}=\{x\in\mathbb{R}^n:|x|=1\}$
and $\omega_{n-1}$ be its surface area. We recall that $\omega_{n-1}=\frac{2\pi^{\frac{n}{2}}}{\Gamma(\frac{n}{2})}$.

Finally, we recall some definitions and properties about the Gauss hypergeometric function and Jacobi polynomials that we will use along the paper.
For $|z|<1$ and $\gamma\notin\{0,-1,-2,...\}$, the Gauss hypergeometric function ${}_2F_1(\alpha,\beta;\gamma;z)$ is defined by
\begin{align*}
{}_2F_1(\alpha,\beta;\gamma;z)=\displaystyle\sum^\infty_{m=0}\frac{(\alpha)_m(\beta)_m}{(\gamma)_m}\frac{z^m}{m!},
\end{align*}
which extends analytically beyond the interval $(-1,1)$ and converges at $z=1$ under the condition ${\rm Re}(\gamma-\alpha-\beta)>0$.
Here the Pochhammer's symbol's $(x)_m$ is defined by
\begin{eqnarray*}
(x)_m
=
\begin{cases}
1,
&  \mbox{if $m=0,$ }\\
x(x+1)\cdot\cdot\cdot(x+m-1),
& \mbox{if $m\geq1.$ } \\
\end{cases}
\end{eqnarray*}
We also recall the Pfaff's identity for the Gauss hypergeometric function (see \cite{[Andrews-Askey-Roy]})
\begin{align}\label{Pfaffs-identity}
{}_2F_1(\alpha,\beta,\gamma;z)=(1-z)^{-\alpha}{}_2F_1\Big(\alpha,\gamma-\beta,\gamma;\frac{z}{z-1}\Big)
\end{align}
and the finite Chu-Vandermonde identity
\begin{align}\label{Chu-Vandermonde}
{}_2F_1\Big(-j,\beta;\gamma;1\Big)=\frac{(\gamma-\beta)_j}{(\gamma)_j}
\end{align}
(see e.g., Corollary 2.2.3 in \cite{[Andrews-Askey-Roy]})

For every integer $k\geq1$, we proceed to recall the standard $k$-th Jacobi polynomials $P^{(\alpha,\beta)}_k(x)$ with indices $\alpha,\beta>-1$ defined by
\begin{align*}
P^{(\alpha,\beta)}_k(x)=\frac{(-1)^k}{2^kk!}(1-x)^{-\alpha}(1+x)^{-\beta}\frac{d^k}{dx^k}\Big((1-x)^{\alpha+k}(1+x)^{\beta+k}\Big),~x\in(-1,1).
\end{align*}
We recall the following Jacobi identity
\begin{align}\label{Jacobi-identity}
P_k^{(\alpha,\beta)}(x)=\frac{1+x}{2}P_k^{(\alpha,\beta+1)}(x)+\frac{1-x}{2}P_k^{(\alpha+1,\beta)}(x)
\end{align}
and the basic relation between the hypergeometric function and the Jacobi polynomials
\begin{align}\label{hypergeometric-Jacobi}
P^{(\alpha,\beta)}_k(x)=\frac{(\alpha+1)_k}{k!}{}_2F_1\Big(-k,k+\alpha+\beta+1;\alpha+1;\frac{1-x}{2}\Big).
\end{align}
The remaining part of this paper is organized as follows. In Section 2, we first recall the local well-posedness for the model \eqref{M-KS}, and then present some useful facts utilized later in this paper.
Section 3 is devoted to some weighted nonlinear integral inequalities for the transform $\mathcal{R}_a$.
The finite-time blow-up of smooth solutions to \eqref{M-KS} is discussed in Section 4.
The proofs of some lemmas are given in Section 5.
\section{Preliminaries}
\subsection{Basic facts}
We first recall the local well-posedness result for the equation \eqref{M-KS} and its Beala-Kato-Majda type blow-up criterion, which was established in \cite{[Zhang-JDE2026]}.
\begin{lemma}\label{local}
Let $n\geq2$ and $a,g>0$. For each $\rho_0\in H^s(\mathbb{R}^n)$ with $s>\frac n2+1$, there exists a $T=T(\|\rho_0\|_{H^s})>0$ such that \eqref{M-KS} admits a unique solution $\rho$ in $C([0,T); H^s(\mathbb{R}^n))\cap {\rm Lip}((0,T);H^{s-1}(\mathbb{R}^n))$.
Moreover, if $0<T_\ast<\infty$ is the first blow-up time such that $\rho$ cannot be continued in $C([0,T_\ast); H^s(\mathbb{R}^n))$, then
\begin{eqnarray*}
\displaystyle\limsup_{t\rightarrow T_\ast}\|\rho(t)\|_{H^s}=\infty~{\rm if~and~only~if}~\lim_{t\rightarrow T_\ast}\int\limits_0^{t}\|\nabla\rho(\tau)\|_{L^\infty}d\tau=\infty.
\end{eqnarray*}
\end{lemma}
We also recall that the radial symmetry of the initial data is preserved by the solution to \eqref{M-KS}.
\begin{lemma}\label{radial-property-preserved}
Fix $n\geq2$ and $a,g>0$. If $\rho$ is a smooth solution to \eqref{M-KS} with a radial and nondecreasing initial data $\rho_0$, then  $\rho(x,t)$ is also radial and nondecreasing in its life span.
\end{lemma}

We also need a string of simple inequalities utilized later. Let $n\geq2$ be an integer and $a>0$.
We denote the weight (see \eqref{boundary-0})
\begin{align}\label{weight-definition}
w_a(\varrho):=1-\frac{\varrho^{n+1}(\varrho^2-na^2)}{(\varrho^2+a^2)^{\frac{n+3}{2}}},~\varrho\geq0.
\end{align}
Then straightforward computations give that
\begin{align*}
w'_a(\varrho)
=\frac{3(n+1)a^2\varrho^n(a\sqrt{\frac{n}{3}}+\varrho)}{(\varrho^2+a^2)^{\frac{n+5}{2}}}\Big(a\sqrt{\frac{n}{3}}-\varrho\Big),~~\varrho\geq0,
\end{align*}
which implies that, $w_a(\varrho)>0$ for any $\varrho\in[0,\infty)$, and
\begin{align}\label{lower-bound-weight}
w_a(\varrho)\geq1,~{\rm for~ every~ } \varrho\in\Big[0,a\sqrt{\frac{n}{3}}\Big].
\end{align}
Moreover,
\begin{align}\label{small-order}
w_a(\varrho)
&=1-\Big(1-n\frac{a^2}{\varrho^2}\Big)\Big(1+\frac{a^2}{\varrho^2}\Big)^{-\frac{n+3}{2}}\nonumber\\
&=1-\Big(1-n\frac{a^2}{\varrho^2}\Big)\Big[1-\frac{n+3}{2}\frac{a^2}{\varrho^2}+O\Big(\frac{a^4}{\varrho^4}\Big)\Big]\nonumber\\
&=\frac{3(n+1)}{2}\frac{a^2}{\varrho^2} +O\Big(\frac{a^4}{\varrho^4}\Big),~{\rm as}~\varrho\rightarrow\infty.
\end{align}
\begin{lemma}\label{Young-inequality-Weight}
Let $n\geq2$ and $a>0$. Given $f:[0,\infty)\mapsto\R$ a smooth and bounded function. We have that, for any $\varepsilon>0$,
\begin{eqnarray*}
\int\limits_0^\infty\frac{(f(\varrho)-f(0))^2}{\varrho e^{\frac\varrho2}}d\varrho
\leq\varepsilon\int\limits_0^\infty\frac{(f(\varrho)-f(0))^2}{\varrho^{2}}w_a(\varrho)d\varrho
+\Big[\frac{8}{\min\Big\{1,a\sqrt{\frac{n}{3}}\Big\}}+\frac{1}{\varepsilon}\Big]\|f\|^2_{L^\infty}
\end{eqnarray*}
and
\begin{eqnarray*}
\int\limits_0^\infty\frac{(f(\varrho)-f(0))^2}{\varrho^{2}}(1-e^{-\frac{\varrho}{2}})d\varrho
\leq
\varepsilon\int\limits_0^\infty\frac{(f(\varrho)-f(0))^2}{\varrho^{2}}w_a(\varrho)d\varrho
+\Big[\frac{4}{\min\Big\{1,a\sqrt{\frac{n}{3}}\Big\}}+\frac{1}{4\varepsilon}\Big]\|f\|^2_{L^\infty},
\end{eqnarray*}
where the weight $w_a(\varrho)$ is defined as \eqref{weight-definition}.
\end{lemma}
\textbf{Proof.} We denote
\begin{align*}
\widetilde{a}:=\min\Big\{1,a\sqrt{\frac{n}{3}}\Big\}.
\end{align*}
By the elementary inequalities
\begin{eqnarray*}
e^{\frac{\varrho}{2}}>1~{\rm and}~e^{\frac{\varrho}{2}}>\frac{\varrho}{2},~ {\rm for}~\varrho>0,
\end{eqnarray*}
along with the Cauchy-Schwarz, Young's inequality and \eqref{lower-bound-weight}, we first derive that
\begin{align*}
&\int\limits_0^\infty\frac{(f(\varrho)-f(0))^2}{\varrho e^{\frac\varrho2}}d\varrho
\leq2\|f\|_{L^\infty}\int\limits_0^{\widetilde{a}}\frac{|f(\varrho)-f(0)|}{\varrho}d\varrho
+8\|f\|^2_{L^\infty}\int\limits_{{\widetilde{a}}}^\infty\frac{d\varrho}{\varrho^2}\\
&~~~~~~=2\|f\|_{L^\infty}\int\limits_0^{\widetilde{a}}\frac{|f(\varrho)-f(0)|}{\varrho}\sqrt{w_a(\varrho)}\frac{d\varrho}{\sqrt{w_a(\varrho)}}
+\frac{8}{\widetilde{a}}\|f\|^2_{L^\infty}\\
&~~~~~~\leq2\|f\|_{L^\infty}\Big(\int\limits_0^{\widetilde{a}}\frac{(f(\varrho)-f(0))^2}{\varrho^2}w_a(\varrho)d\varrho\Big)^{\frac12}
\Big(\int\limits_0^{\widetilde{a}}\frac{d\varrho}{w_a(\varrho)}\Big)^{\frac12}
+\frac{8}{\widetilde{a}}\|f\|^2_{L^\infty}\\
&~~~~~~\leq2\|f\|_{L^\infty}\Big(\int\limits_0^\infty\frac{(f(\varrho)-f(0))^2}{\varrho^2}w_a(\varrho)d\varrho\Big)^{\frac12}
+\frac{8}{\widetilde{a}}\|f\|^2_{L^\infty}\\
&~~~~~~\leq\varepsilon\int\limits_0^\infty\frac{(f(\varrho)-f(0))^2}{\varrho^{2}}w_a(\varrho)d\varrho
+\Big(\frac{8}{\widetilde{a}}+\frac{1}{\epsilon}\Big)\|f\|^2_{L^\infty},
\end{align*}
where $\varepsilon>0$ is any positive constant.
Similarly, also by the elementary inequalities
\begin{eqnarray*}
1-e^{-\frac{\varrho}{2}}<\frac{\varrho}{2}~{\rm and}~0<1-e^{-\frac{\varrho}{2}}<1,~ {\rm for}~\varrho>0,
\end{eqnarray*}
we can obtain that
\begin{align*}
&\int\limits_0^\infty\frac{(f(\varrho)-f(0))^2}{\varrho^{2}}(1-e^{-\frac{\varrho}{2}})d\varrho
\leq\|f\|_{L^\infty}\int\limits_0^{\widetilde{a}}\frac{|f(\varrho)-f(0)|}{\varrho}d\varrho
+4\|f\|^2_{L^\infty}\int\limits_{\widetilde{a}}^\infty\frac{d\varrho}{\varrho^{2}}\\
&~~~~~~~~~~~~~\leq\|f\|_{L^\infty}\Big[\int\limits_0^{\widetilde{a}}\frac{(f(\varrho)-f(0))^2}{\varrho^{2}}w_a(\varrho)d\varrho\Big]^{\frac12}
\Big[\int\limits_0^{\widetilde{a}}\frac{d\varrho}{w_a(\varrho)}\Big]^{\frac12}
+\frac{4}{\widetilde{a}}\|f\|^2_{L^\infty}\\
&~~~~~~~~~~~~~\leq\|f\|_{L^\infty}\Big(\int\limits_0^\infty\frac{(f(\varrho)-f(0))^2}{\varrho^{2}}w_a(\varrho)d\varrho\Big)^{\frac12}
+\frac{4}{\widetilde{a}}\|f\|^2_{L^\infty}\\
&~~~~~~~~~~~~~\leq\varepsilon\int\limits_0^\infty\frac{(f(\varrho)-f(0))^2}{\varrho^{2}}w_a(\varrho)d\varrho+\Big(\frac{4}{\widetilde{a}}+\frac{1}{4\varepsilon}\Big)\|f\|^2_{L^\infty}.
\end{align*}
We then finish the proof of Lemma \ref{Young-inequality-Weight}. \hfill\hfill$\square$\vskip12pt

Finally, we recall a well-known Bernstein type estimate for the Jacobi polynomials, which was applied directly in \cite{[Betancor-De Leon-Contreras]} (see (5) there).  We will give its proof in the Appendix \ref{appendix} for completeness.
\begin{lemma}\label{Bernstein-Lemma}
Let $\alpha\geq0$ and $\beta\geq0$. There exists a constant $C_{\alpha,\beta}>0$, depending only on $\alpha$ and $\beta$, such that for every integer $k\geq1$ and any $x\in(-1,1)$,
\begin{align*}
(1-x)^{\frac{\alpha}{2}+\frac14}(1+x)^{\frac{\beta}{2}+\frac14}|P^{(\alpha,\beta)}_k(x)|\leq \frac{C_{\alpha,\beta}}{\sqrt{k}}.
\end{align*}
\end{lemma}
\subsection{An auxiliary function}
In this subsection, we mainly study an auxiliary function defined by
\begin{eqnarray*}
g_\mu(\lambda,m,\beta)
:=\int\limits_0^\pi\frac{\sin^m\theta }{(1-2\lambda\cos\theta+\lambda^2+\mu^2)^{\frac{m}{2}+\beta}}d\theta,
\end{eqnarray*}
where $\lambda\in(-1,1)$, $m\in\mathbb{N}$,  $\beta>0$ and $\mu>0$.
We can readily see that $g_\mu$ is even, infinitely differentiable (and real analytic) in $\lambda\in(-1,1)$ and its Taylor's series at $0$ has the form:
\begin{eqnarray*}
g_\mu(\lambda,m,\beta)=\sum^\infty_{k=0}\frac{\frac{\partial^{2k}g_\mu}{\partial \lambda^{2k}}(0,m,\beta)}{(2k)!}\lambda^{2k}.
\end{eqnarray*}
Then we have
\begin{lemma}\label{Taylor-series-coefficients}
Let the integer $m\geq2$, $\beta>0$ and $\mu>0$. Then, for any integer $k\geq1$,
\begin{eqnarray}\label{claim}
\begin{split}
&\frac{\partial^{2k}g_\mu }{\partial \lambda^{2k}}(0,m,\beta)
=B\Big(\frac12,\frac{m+1}{2}\Big)(2k-1)!!\displaystyle\prod^{k-1}_{j=0}\frac{(2\beta+2j)(m+2\beta+2j)}{m+2j+2}\\
&\ \ \ \ \ \ \ \ \ \ \ \  \ \ \ \ \ \ \ \ \ \ \ \ \ \
\times
\frac{1+\displaystyle\sum^{k}_{j=1}(-1)^j\binom{k}{j}\Big(\displaystyle\prod^{j-1}_{l=0}\frac{m+2(k-l)}{2\beta+2l}\Big)\mu^{2j}}{(1+\mu^2)^{\frac{m}{2}+\beta+2k}}.
\end{split}
\end{eqnarray}

\end{lemma}
The proof of this lemma will be postponed in the Appendix \ref{appendix}.

\begin{remark}\label{remark-4.1}
We also need another auxiliary function defined by
\begin{eqnarray*}
g_0(\lambda,m,\beta):=\int\limits_0^\pi\frac{\sin^m\theta}{(1-2\lambda\cos\theta+\lambda^2)^{\frac{m}{2}+\beta}}d\theta,
\end{eqnarray*}
where $\lambda\in(-1,1)$, $m\in\mathbb{N}$ and $\beta>0$.
Clearly, $g_0$ is also even, infinitely differentiable (and real analytic) in $\lambda\in(-1,1)$ and its Taylor's series at $0$ has the form:
\begin{eqnarray*}
g_0(\lambda,m,\beta)=\sum^\infty_{k=0}\frac{\frac{\partial^{2k}g_0}{\partial \lambda^{2k}}(0,m,\beta)}{(2k)!}\lambda^{2k}.
\end{eqnarray*}
From the proof of Lemma \ref{Taylor-series-coefficients}, we can see that, for $k\geq1$,
\begin{eqnarray}\label{claim-g1}
\begin{split}
&\frac{\partial^{2k}g_0}{\partial \lambda^{2k}}(0,m,\beta)
=B\Big(\frac12,\frac{m+1}{2}\Big)(2k-1)!!\displaystyle\prod^{k-1}_{j=0}\frac{(2\beta+2j)(m+2\beta+2j)}{m+2j+2},
\end{split}
\end{eqnarray}
which corresponds to the special case $\mu=0$ of \eqref{claim} and was derived in the previous work \cite{[Zhang-JAM27]}.
\end{remark}

The functions $g_\mu(\cdot,m,\beta)$ and $g_0(\cdot,m,\beta)$ with $m=n\geq2$ and $\beta=\frac12$ will play a vital role in this paper (see, e.g., \eqref{velocity-expression}).
For the convenience of our later application, we denote
\begin{eqnarray}\label{G-definition}
G(\lambda)
:=\int\limits_0^\pi\frac{\sin^n\theta }{(1-2\lambda\cos\theta+\lambda^2)^{\frac{n+1}{2}}}d\theta
\end{eqnarray}
and
\begin{eqnarray}\label{G-miu-definition}
\widetilde{G}(\lambda,\mu)
:=\int\limits_0^\pi\frac{\sin^n\theta }{(1-2\lambda\cos\theta+\lambda^2+\mu^2)^{\frac{n+1}{2}}}d\theta,
\end{eqnarray}
where $\lambda\in(-1,1)$,  $n\geq2$ and $\mu>0$. Clearly,
\begin{eqnarray}\label{0-value}
G(0)=B\Big(\frac12,\frac{n+1}{2}\Big),~\widetilde{G}(0,\mu)=\frac{B(\frac12,\frac{n+1}{2})}{(1+\mu^2)^{\frac{n+1}{2}}} ~{\rm and}~ \frac{\partial\widetilde{G}}{\partial\mu}(0,\mu)=-\frac{(n+1) B(\frac12,\frac{n+1}{2})\mu}{(1+\mu^2)^{\frac{n+3}{2}}}
\end{eqnarray}
and it is easy to check that, for any $\lambda>0$ and $\lambda\neq1$ (notice that $G(\lambda)$ is not defined at 1),
\begin{eqnarray}\label{transform-identity}
G\Big(\frac{1}{\lambda}\Big)=\lambda^{n+1}G(\lambda).
\end{eqnarray}
In addition, $G(\cdot)$ is smooth in $[0,1)\cup(1,\infty)$ with a singularity at $1$. The singularity is of order $\ln|\lambda-1|$. Specifically,
\begin{align}\label{singularity-order}
\lim_{\lambda\rightarrow1}\frac{G(\lambda)}{\ln|\lambda-1|}=-1,~~~\lim_{\lambda\rightarrow1}(\lambda-1)G'(\lambda)=-1,
\end{align}
whose proofs will be given later in the Appendix \ref{appendix}.
We also have that, for any $\lambda>0$ and $\mu>0$,
\begin{eqnarray}\label{transform-identity-G}
\widetilde{G}\Big(\frac{1}{\lambda},\mu\Big)=\lambda^{n+1}\widetilde{G}(\lambda,\lambda\mu).
\end{eqnarray}
As a corollary of Lemma \ref{Taylor-series-coefficients} and Remark \ref{remark-4.1}, we immediately have
\begin{corollary}\label{Maclaurin-series}
For any $\lambda\in(-1,1)$ and all $\mu>0$, it holds that
\begin{align}\label{Taylor}
G(\lambda)-\widetilde{G}(\lambda,\mu)
=\sum^\infty_{k=0}a_{2k}b_k(\mu)\lambda^{2k},
\end{align}
where
\begin{eqnarray*}
a_{2k}
=
\begin{cases}
B\Big(\frac12,\frac{n+1}{2}\Big),
&  \mbox{if $k=0,$ }\\
B\Big(\frac12,\frac{n+1}{2}\Big)\frac{(2k-1)!!(\frac{n}{2}+\frac{1}{2})_k}{(2k)!!(\frac{n}{2}+1)_k},
& \mbox{if $k\geq1,$ } \\
\end{cases}
\end{eqnarray*}
and
\begin{eqnarray}\label{bkmu-1}
b_{k}(\mu)
=
\begin{cases}
1-\frac{1}{(1+\mu^2)^{\frac{n+1}{2}}},
&  \mbox{if $k=0,$ }\\
1-\frac{1+\displaystyle\sum^{k}_{j=1}(-1)^{j}\binom{k}{j}\Big(\displaystyle\prod^{j-1}_{l=0}\frac{n+2(k-l)}{1+2l}\Big)\mu^{2j}}{(1+\mu^2)^{\frac{n+1}{2}+2k}},
& \mbox{if $k\geq1.$ } \\
\end{cases}
\end{eqnarray}
Consequently, for any $\lambda\in(-1,1)$ and all $\mu>0$,
\begin{align}\label{1-derivative-Taylor}
G'(\lambda)-\frac{\partial\widetilde{G}}{\partial\lambda}(\lambda,\mu)
=\sum^\infty_{k=1}2ka_{2k}b_k(\mu)\lambda^{2k-1}
\end{align}
and
\begin{align}\label{2-derivative-Taylor}
G''(\lambda)-\frac{\partial^2\widetilde{G}}{\partial\lambda^2}(\lambda,\mu)
=\sum^\infty_{k=1}2k(2k-1)a_{2k}b_k(\mu)\lambda^{2k-2}.
\end{align}
\end{corollary}
We also need the Taylor series at $0$ for the functions $\frac{\partial \widetilde{G}}{\partial \mu}(\cdot,\mu)$ and $\frac{\partial^2 \widetilde{G}}{\partial\lambda\partial \mu}(\cdot,\mu)$.
\begin{corollary}\label{1-mu-2-lambda-mu}
For any $\lambda\in(0,1)$ and all $\mu>0$, it holds that
\begin{align}\label{1-mu}
\frac{\partial \widetilde{G}}{\partial \mu}(\lambda,\mu)
=-(n+1)\mu\sum^\infty_{k=0}\widetilde{a}_{2k}\widetilde{b}_k(\mu)\lambda^{2k},
\end{align}
where
\begin{eqnarray*}
\widetilde{a}_{2k}
=
\begin{cases}
B\Big(\frac12,\frac{n+1}{2}\Big),
&  \mbox{if $k=0,$ }\\
B\Big(\frac12,\frac{n+1}{2}\Big)\frac{(2k+1)!!(\frac{n}{2}+\frac{3}{2})_k}{(2k)!!(\frac{n}{2}+1)_k},
& \mbox{if $k\geq1,$ } \\
\end{cases}
\end{eqnarray*}
and
\begin{eqnarray}\label{bkmu-2}
\widetilde{b}_{k}(\mu)
=
\begin{cases}
\frac{1}{(1+\mu^2)^{\frac{n+3}{2}}},
&  \mbox{if $k=0,$ }\\
\frac{1+\displaystyle\sum^{k}_{j=1}(-1)^{j}\binom{k}{j}\Big(\displaystyle\prod^{j-1}_{l=0}\frac{n+2(k-l)}{3+2l}\Big)\mu^{2j}}{(1+\mu^2)^{\frac{n+3}{2}+2k}},
& \mbox{if $k\geq1.$ } \\
\end{cases}
\end{eqnarray}
Consequently, for any $\lambda\in(-1,1)$ and all $\mu>0$,
\begin{align}\label{2-lambda-mu}
\frac{\partial^2 \widetilde{G}}{\partial\lambda\partial \mu}(\lambda,\mu)
=-(n+1)\mu\sum^\infty_{k=1}2k\widetilde{a}_{2k}\widetilde{b}_k(\mu)\lambda^{2k-1},
\end{align}
\end{corollary}
\textbf{Proof.} By \eqref{G-miu-definition}, we find that
\begin{align}\label{derivative-mu}
\frac{\partial \widetilde{G}}{\partial\mu}(\lambda,\mu)
=-(n+1)\mu\int\limits_0^\pi\frac{\sin^n\theta }{(1-2\lambda\cos\theta+\lambda^2+\mu^2)^{\frac{n+3}{2}}}d\theta.
\end{align}
Letting $m=n$ and $\beta=\frac32$ in Lemma \ref{Taylor-series-coefficients}, we can obtain that
\begin{align*}
\int\limits_0^\pi\frac{\sin^n\theta }{(1-2\lambda\cos\theta+\lambda^2+\mu^2)^{\frac{n+3}{2}}}d\theta
=\sum^\infty_{k=0}\widetilde{a}_{2k}\widetilde{b}_k(\mu)\lambda^{2k},
\end{align*}
which along with \eqref{derivative-mu} gives \eqref{1-mu}. Differentiation term by term on \eqref{1-mu} directly gives \eqref{2-lambda-mu}.
\hfill\hfill$\square$\vskip12pt
It should be remarked that, for any $k\geq0$,
\begin{align}\label{widetilde-a-2k-relation}
\widetilde{a}_{2k}=\frac{(2k+1)(2k+n+1)}{n+1}a_{2k}
\end{align}
and
\begin{align}\label{asympotic}
\lim_{k\rightarrow\infty}ka_{2k}
&=B\Big(\frac12,\frac{n+1}{2}\Big)\lim_{k\rightarrow\infty}\Big[k\cdot\frac{(\frac12)_k}{(1)_k}\cdot\frac{(\frac{n}{2}+\frac{1}{2})_k}{(\frac{n}{2}+1)_k}\Big]\nonumber\\
&=B\Big(\frac12,\frac{n+1}{2}\Big)
\lim_{k\rightarrow\infty}\Big[k\cdot\frac{\Gamma(1)}{\Gamma(\frac12)}k^{\frac12-1}\cdot\frac{\Gamma(\frac{n}{2}+1)}{\Gamma(\frac{n}{2}+\frac{1}{2})}k^{(\frac{n}{2}+\frac{1}{2})-(\frac{n}{2}+1)}\Big]\\
&=B\Big(\frac12,\frac{n+1}{2}\Big)
\frac{\Gamma(\frac{n}{2}+1)}{\Gamma(\frac12)\Gamma(\frac{n}{2}+\frac{1}{2})}=1\nonumber.
\end{align}
The key inequality in the following lemma is of independent interest, which  implies a bit surprising fact that all the Taylor coefficients of $G(\cdot)-\widetilde{G}(\cdot,\mu)$ are positive.

\begin{lemma}\label{Key-inequality}
Let $N>-1$ and $K\geq1$ be integers. Then, for every $0\neq\mu\in\mathbb{R}$,
\begin{align}\label{Key}
(1+\mu^2)^{\frac{N+1}{2}+2K}
>\Big|1+\displaystyle\sum^{K}_{j=1}(-1)^j\binom{K}{j}\Big(\displaystyle\prod^{j-1}_{l=0}\frac{N+2(K-l)}{1+2l}\Big)\mu^{2j}\Big|.
\end{align}
\end{lemma}
\textbf{Proof.}
Set
\begin{align}\label{polynomial-definition}
\mathcal{P}_{N,K}(x):=1+\displaystyle\sum^{K}_{j=1}(-1)^j\binom{K}{j}\Big(\displaystyle\prod^{j-1}_{l=0}\frac{N+2(K-l)}{1+2l}\Big)x^{j}.
\end{align}
The desired inequality \eqref{Key} is precisely the one below: for any $0\neq\mu\in\mathbb{R}$,
\begin{align*}
\frac{|\mathcal{P}_{N,K}(\mu^2)|}{(1+\mu^2)^{\frac{N+1}{2}+2K}}<1.
\end{align*}
We identify the quotient on the left as a normalized Fourier transform. For this purpose, we first claim that, for every $\mu\in\mathbb{R}$,
\begin{align}\label{derivative-identity}
\frac{\frac{d^{2K}}{d\mu^{2K}}(1+\mu^2)^{-\frac{N+1}{2}}}{\frac{d^{2K}}{d\mu^{2K}}\Big|_{\mu=0}(1+\mu^2)^{-\frac{N+1}{2}}}
=\frac{\mathcal{P}_{N,K}(\mu^2)}{(1+\mu^2)^{\frac{N+1}{2}+2K}}.
\end{align}
Indeed, for a smooth function $\phi$, the special case of the Fa\`{a} di Bruno formula (see \cite{[Faa]}, and also \cite{[Spindler]} for a modern formulation and a simple proof) corresponding to the map $\mu\mapsto \mu^2$ gives that
\begin{align*}
\frac{d^{2K}}{d\mu^{2K}}\phi(\mu^2)=\displaystyle\sum^{K}_{k=0}\frac{(2K)!4^k}{(2k)!(K-k)!}\mu^{2k}\phi^{(K+k)}(\mu^2).
\end{align*}
Taking $\phi(s)=(1+s)^{-\frac{N+1}{2}}$, we have that, for any $0\leq k\leq K$,
\begin{align*}
\phi^{(K+k)}(s)=(-1)^{K+k}\Big(\frac{N+1}{2}\Big)_{K+k}(1+s)^{-\frac{N+1}{2}-K-k}.
\end{align*}
Hence
\begin{align*}
(-1)^K\frac{d^{2K}}{d\mu^{2K}}(1+\mu^2)^{-\frac{N+1}{2}}
=(2K)!\displaystyle\sum^{K}_{k=0}(-1)^k\frac{4^k(\frac{N+1}{2})_{K+k}}{(2k)!(K-k)!}\frac{\mu^{2k}}{(1+\mu^2)^{\frac{N+1}{2}+K+k}}.
\end{align*}
At $t=0$ this formula reads as
\begin{align*}
(-1)^K\frac{d^{2K}}{d\mu^{2K}}\Big|_{\mu=0}(1+\mu^2)^{-\frac{N+1}{2}}
=\frac{(2K)!(\frac{N+1}{2})_K}{K!}.
\end{align*}
Dividing by this value yields that
\begin{align*}
\frac{\frac{d^{2K}}{d\mu^{2K}}(1+\mu^2)^{-\frac{N+1}{2}}}{\frac{d^{2K}}{d\mu^{2K}}\Big|_{\mu=0}(1+\mu^2)^{-\frac{N+1}{2}}}
=\frac{K!}{(1+\mu^2)^{\frac{N+1}{2}+2K}}\displaystyle\sum^{K}_{k=0}\frac{(-1)^k}{(K-k)!}\frac{4^k(\frac{N+1}{2}+K)_{k}}{(2k)!}\mu^{2k}(1+\mu^2)^{K-k}.
\end{align*}
By the Binomial Theorem, we further compute the right formula as
\begin{align*}
\frac{\frac{d^{2K}}{d\mu^{2K}}(1+\mu^2)^{-\frac{N+1}{2}}}{\frac{d^{2K}}{dt^{2K}}\Big|_{t=0}(1+t^2)^{-\frac{N+1}{2}}}
&=\frac{K!}{(1+\mu^2)^{\frac{N+1}{2}+2K}}\displaystyle\sum^{K}_{k=0}\frac{(-1)^k}{(K-k)!}\frac{4^k(\frac{N+1}{2}+K)_{k}}{(2k)!}\mu^{2k}\displaystyle\sum^{K-k}_{m=0}\binom{K-k}{m}\mu^{2m}\\
&=\frac{1}{(1+\mu^2)^{\frac{N+1}{2}+2K}}\displaystyle\sum^{K}_{k=0}\displaystyle\sum^{K-k}_{m=0}(-1)^k\frac{(\frac{N+1}{2}+K)_k}{(\frac{1}{2})_k}\binom{m+k}{k}\binom{K}{m+k}\mu^{2(m+k)}\\
&=\frac{1}{(1+\mu^2)^{\frac{N+1}{2}+2K}}\displaystyle\sum^{K}_{j=0}\binom{K}{j}\displaystyle\sum^{j}_{k=0}(-1)^k\binom{j}{k}\frac{(\frac{N+1}{2}+K)_k}{(\frac{1}{2})_k}\mu^{2j}\\
&=\frac{1}{(1+\mu^2)^{\frac{N+1}{2}+2K}}\displaystyle\sum^{K}_{j=0}\binom{K}{j}\displaystyle\sum^{j}_{k=0}\frac{(-j)_k(\frac{N+1}{2}+K)_k}{(\frac{1}{2})_kk!}\mu^{2j}\\
&=\frac{1}{(1+\mu^2)^{\frac{N+1}{2}+2K}}\displaystyle\sum^{K}_{j=0}\binom{K}{j}{}_2F_1\Big(-j,\frac{N+1}{2}+K;\frac12;1\Big)\mu^{2j}\\
&=\frac{1}{(1+\mu^2)^{\frac{N+1}{2}+2K}}\displaystyle\sum^{K}_{j=0}\binom{K}{j}\frac{(-\frac{N}{2}-K)_j}{(\frac12)_j}\mu^{2j}\\
&=\frac{1}{(1+\mu^2)^{\frac{N+1}{2}+2K}}\displaystyle\sum^{K}_{j=0}(-1)^j\binom{K}{j}\Big(\displaystyle\prod^{j-1}_{l=0}\frac{N+2(K-l)}{1+2l}\Big)\mu^{2j}\\
&=\frac{\mathcal{P}_{N,K}(\mu^2)}{(1+\mu^2)^{\frac{N+1}{2}+2K}},
\end{align*}
where we also used the identity $(2k)!=4^kk!(\frac{1}{2})_k$ and the finite Chu-Vandermonde identity \eqref{Chu-Vandermonde}.
Thus, the derivative identity \eqref{derivative-identity} follows.

Secondly, we identify the derivative on the left of \eqref{derivative-identity} as a normalized Fourier transform.
Using the Gamma integral
\begin{align*}
(1+\mu^2)^{-\frac{N+1}{2}}=\frac{1}{\Gamma(\frac{N+1}{2})}\int\limits_0^\infty\tau^{\frac{N-1}{2}}e^{-\tau}e^{-\tau \mu^2}d\tau
\end{align*}
and the Fourier representation of a Gaussian
\begin{align*}
e^{-\tau \mu^2}=\frac{1}{2\sqrt{\pi\tau}}\int\limits_{\mathbb{R}}e^{-\frac{\xi^2}{4\tau}}e^{i\mu\xi}d\xi,
\end{align*}
we obtain that
\begin{align}\label{representation-identity}
(1+\mu^2)^{-\frac{N+1}{2}}=\int\limits_{\mathbb{R}}e^{i\mu\xi}w_N(\xi) d\xi,
\end{align}
where
\begin{align*}
w_N(\xi)=\frac{1}{2\sqrt{\pi}\Gamma(\frac{N+1}{2})}\int\limits_0^\infty\tau^{\frac{N}{2}-1}e^{-\tau-\frac{\xi^2}{4\tau}}
d\tau.
\end{align*}
The function $w_N$ is even and strictly positive.
Moreover,
\begin{align*}
\int\limits_{\mathbb{R}}\xi^{2K}w_N(\xi)d\xi
=\frac{2^{2K}}{\sqrt{\pi}}\cdot\frac{\Gamma\Big(K+\frac12\Big)\Gamma\Big(\frac{N+1}{2}+K\Big)}{\Gamma(\frac{N+1}{2})}<\infty.
\end{align*}
This follows directly by interchanging the two positive integrals and evaluating the Gaussian moment.
Consequently, differentiating \eqref{representation-identity} $2K$ times under the integral yields that
\begin{align*}
\frac{d^{2K}}{d\mu^{2K}}(1+\mu^2)^{-\frac{N+1}{2}}
&=(-1)^K\int\limits_{\mathbb{R}}\xi^{2K}e^{i\mu\xi}w_N(\xi) d\xi\\
&=(-1)^K\int\limits_{\mathbb{R}}\xi^{2K}\cos(\mu\xi)w_N(\xi) d\xi.
\end{align*}
At $\mu=0$, this representation reads as
\begin{align*}
\frac{d^{2K}}{d\mu^{2K}}\Big|_{\mu=0}(1+\mu^2)^{-\frac{N+1}{2}}
=(-1)^K\int\limits_{\mathbb{R}}\xi^{2K}w_N(\xi) d\xi.
\end{align*}
Combining these identities with \eqref{derivative-identity}, we find that, for any $0\neq \mu\in\mathbb{R}$,
\begin{align*}
\frac{|\mathcal{P}_{N,K}(\mu^2)|}{(1+\mu^2)^{\frac{N+1}{2}+2K}}
=\frac{\Big|\int\limits_{\mathbb{R}}\xi^{2K}\cos(\mu\xi)w_N(\xi) d\xi\Big|}{\int\limits_{\mathbb{R}}\xi^{2K}w_N(\xi) d\xi}
\leq\frac{\int\limits_{\mathbb{R}}\xi^{2K}|\cos(\mu\xi)|w_N(\xi) d\xi}{\int\limits_{\mathbb{R}}\xi^{2K}w_N(\xi) d\xi}<1.
\end{align*}
The proof of Lemma \ref{Key-inequality} is finished.
\hfill\hfill$\square$\vskip12pt
\begin{remark}
The special case $N=1$ of the key inequality \eqref{Key} may admit an elementary proof.
Indeed, for any $1\leq j\leq K$, we note that
\begin{align*}
\binom{K}{j}\displaystyle\prod^{j-1}_{l=0}\frac{1+2(K-l)}{1+2l}=\binom{2K+1}{2j}.
\end{align*}
It follows that, for any integer $K\geq1$ and any $\mu>0$,
\begin{align*}
&1+\displaystyle\sum^{K}_{j=1}(-1)^j\binom{K}{j}\Big(\displaystyle\prod^{j-1}_{l=0}\frac{1+2(K-l)}{1+2l}\Big)\mu^{2j}
=\displaystyle\sum^{K}_{j=0}(-1)^{j}\binom{2K+1}{2j}\mu^{2j}\\
&~~~~~~~~=\frac{1}{2}\Big[(1+i\mu)^{2K+1}+(1-i\mu)^{2K+1}\Big]\\
&~~~~~~~~=\frac{1}{2}\Big[(\sqrt{1+\mu^2}e^{i\arctan\mu})^{2K+1}+(\sqrt{1+\mu^2}e^{-i\arctan\mu})^{2K+1}\Big]\\
&~~~~~~~~=(1+\mu^2)^{K+\frac12}\cos((2K+1)\arctan\mu),
\end{align*}
which yields that
\begin{align*}
\Big|1+\displaystyle\sum^{K}_{j=1}(-1)^j\binom{K}{j}\Big(\displaystyle\prod^{j-1}_{l=0}\frac{1+2(K-l)}{1+2l}\Big)\mu^{2j}\Big|
&=(1+\mu^2)^{K+\frac12}|\cos((2K+1)\arctan\mu)|\\
&\leq(1+\mu^2)^{K+\frac12}<(1+\mu^2)^{2K+1}.
\end{align*}
\end{remark}
\section{Nonlinear weighted inequalities for the transform $\mathcal{R}_a$}
\hskip\parindent In this section, we present some new nonlinear weighted inequalities for the transform $\mathcal{R}_a$. By abuse of notation, we will not distinguish $f(x)$ and $f(|x|)$ for a radially symmetric function $f$.
\subsection{Improved weighted inequalities for radially nondecreasing functions}
In \cite{[Zhang-JDE2026]}, Zhang first established the pointwise inequality \eqref{pointwise-inequality-JDE} for the transform $\mathcal{R}_a$ acting on the radially nondecreasing function, and then proved the nonlinear integral inequality \eqref{weighted-Ra-monotone} for the transform $\mathcal{R}_a$. In this subsection, we will improve these inequalities in \cite{[Zhang-JDE2026]} for the transform $\mathcal{R}_a$.
\begin{proposition}\label{lower-bound-Ra}
Fix $n\geq2$ and $a>0$. Let $f:\mathbb{R}^n\rightarrow\mathbb{R}$ be a radial, nondecreasing and continuously differentiable function with $\nabla f\in L^1(\mathbb{R}^n)\cap L^\infty(\mathbb{R}^n)$. Then,  for any $x\neq0$, we have
\begin{eqnarray}\label{lower-bound-Ra-inequality}
-\mathcal{R}_a f(x)\cdot\frac{x}{|x|}
\geq\frac{nB(\frac12,\frac{n+1}{2})}{\pi}\frac{1}{|x|^n}\Big[1-\frac{|x|^{n+1}}{(|x|^2+a^2)^{\frac{n+1}{2}}}\Big]\int\limits_0^{|x|} (f(|x|)-f(\varrho))\varrho^{n-1}
d\varrho.
\end{eqnarray}
\end{proposition}
\begin{remark}
The pointwise inequality \eqref{lower-bound-Ra-inequality} for the transform $\mathcal{R}_a$  is better and more natural than the inequality \eqref{pointwise-inequality-JDE} established in \cite{[Zhang-JDE2026]}. This can be seen in two ways. First, we note that
\begin{eqnarray*}
1-\frac{|x|^{n+1}}{(|x|^2+a^2)^{\frac{n+1}{2}}}
>\frac{1}{2^{n+1}}\Big[1-\frac{2^{n+1}|x|^{n+1}}{(4|x|^2+a^2)^{\frac{n+1}{2}}}\Big].
\end{eqnarray*}
Second, when $n=1$, the inequality \eqref{lower-bound-Ra-inequality} formally coincides with the one-dimensional inequality \eqref{pointwise-inequality-Ha} for the transform $\mathcal{H}_a$ obtained in \cite{[Liu-Zhang]}.
\end{remark}
\textbf{Proof of Proposition \ref{lower-bound-Ra}.} Following the standard computations in the proof of Lemma 4.4 in  \cite{[Zhang-JDE2026]}, we have that, for any $x\neq0$,
\begin{align}\label{velocity-expression}
&-\mathcal{R}_af(x)\cdot\frac{x}{|x|}
=\frac{\Gamma(\frac{n-1}{2})}{2\pi^{\frac{n+1}{2}}}\int\limits_0^\infty f'(\varrho)\varrho^{n-1}
\int\limits_{\mathbb{S}^{n-1}}\Big(\frac{z_1}{||x|e_1-\varrho z|^{n-1}}-\frac{z_1}{(||x|e_1-\varrho z|^2+a^2)^{\frac{n-1}{2}}}\Big)d\sigma(z)d\varrho\nonumber\\
&~~~~~~~~=\frac{1}{\pi}\frac{1}{|x|^{n}}\int\limits_0^\infty f'(\varrho)\varrho^n
\int\limits_0^\pi\Big[\frac{\sin^n\mu }{(1-2\frac{\varrho}{|x|}\cos\mu+\frac{\varrho^2}{|x|^2})^{\frac{n+1}{2}}}-\frac{\sin^n\mu }{(1-2\frac{\varrho}{|x|}\cos\mu+\frac{\varrho^2}{|x|^2}+\frac{a^2}{|x|^2})^{\frac{n+1}{2}}}\Big]d\mu d\varrho\nonumber\\
&~~~~~~~~=\frac{1}{\pi}\frac{1}{|x|^n}\int\limits_0^\infty f'(\varrho)\varrho^{n}
\Big[G\Big(\frac{\varrho}{|x|}\Big)-\widetilde{G}\Big(\frac{\varrho}{|x|},\frac{a}{|x|}\Big)\Big]d\varrho,
\end{align}
where the functions $G(\cdot)$ and $\widetilde{G}(\cdot,\mu)$ are defined as \eqref{G-definition} and \eqref{G-miu-definition}, respectively.
By Lemma \ref{Key-inequality}, we have that $b_{k}(\mu)>0$ for any integer $k\geq1$ and any $\mu>0$. As a consequence,
by \eqref{Taylor}, we can obtain that,
for any $\lambda\in(0,1)$ and all $\mu>0$,
\begin{align*}
G(\lambda)-\widetilde{G}(\lambda,\mu)
&=a_0b_0(\mu)+\sum^\infty_{k=1}a_{2k}b_k(\mu)\lambda^{2k}\\
&>a_0b_0(\mu)=B\Big(\frac12,\frac{n+1}{2}\Big)\Big[1-\frac{1}{(1+\mu^2)^{\frac{n+1}{2}}}\Big].
\end{align*}
Thus, we finally have that
\begin{align*}
-\mathcal{R}_af(x)\cdot\frac{x}{|x|}
&\geq\frac{1}{\pi}\frac{1}{|x|^n}\int\limits_0^{|x|} f'(\varrho)\varrho^{n}
\Big[G\Big(\frac{\varrho}{|x|}\Big)-\widetilde{G}\Big(\frac{\varrho}{|x|},\frac{a}{|x|}\Big)\Big]d\varrho\\
&\geq\frac{1}{\pi}\frac{1}{|x|^n}\int\limits_0^{|x|} f'(\varrho)\varrho^{n}
B\Big(\frac12,\frac{n+1}{2}\Big)\Big[1-\frac{1}{(1+(\frac{a}{|x|})^2)^{\frac{n+1}{2}}}\Big]d\varrho\\
&=\frac{B(\frac12,\frac{n+1}{2})}{\pi}\frac{1}{|x|^n}\Big[1-\frac{|x|^{n+1}}{(|x|^2+a^2)^{\frac{n+1}{2}}}\Big]\int\limits_0^{|x|} f'(\varrho)\varrho^{n}d\varrho\\
&=\frac{nB(\frac12,\frac{n+1}{2})}{\pi}\frac{1}{|x|^n}\Big[1-\frac{|x|^{n+1}}{(|x|^2+a^2)^{\frac{n+1}{2}}}\Big]\int\limits_0^{|x|} (f(|x|)-f(\varrho))\varrho^{n-1}
d\varrho,
\end{align*}
which is the desired lower bound. \hfill\hfill$\square$\vskip12pt
\begin{remark}
For a radial and non-increasing Schwartz function $f:\mathbb{R}^n\rightarrow\mathbb{R}$, we have the pointwise  inequality
\begin{eqnarray*}
\mathcal{R}_a f(x)\cdot\frac{x}{|x|}
\geq\frac{nB(\frac12,\frac{n+1}{2})}{\pi}\frac{1}{|x|^n}\Big[1-\frac{|x|^{n+1}}{(|x|^2+a^2)^{\frac{n+1}{2}}}\Big]
\int\limits_0^{|x|} (f(\varrho)-f(|x|))\varrho^{n-1}
d\varrho.
\end{eqnarray*}
\end{remark}
By using the improved pointwise inequality \eqref{lower-bound-Ra-inequality}, we derive the following nonlinear integral inequality for the transform $\mathcal{R}_a$, which improves the inequality \eqref{weighted-Ra-monotone} established in \cite{[Zhang-JDE2026]}.
\begin{proposition}\label{n-nonlinear-inequality-monotone}
Fix $n\geq2$ and $a>0$. Let $-1<\delta<1$ and $f:\mathbb{R}^n\rightarrow\mathbb{R}$ be a radial, nondecreasing and continuously differentiable function with $\nabla f\in L^1(\mathbb{R}^n)\cap L^\infty(\mathbb{R}^n)$.
Then
\begin{eqnarray*}
-\int\limits_{\mathbb{R}^n}\frac{\mathcal{R}_a f(x)\cdot\nabla f(x)}{|x|^{n+\delta}}dx
\geq C_{n,\delta}
\int\limits_{\mathbb{R}^n}\frac{(f(0)-f(x))^2}{|x|^{n+1+\delta}}\Big[1-\frac{|x|^{n+1}}{(|x|^2+a^2)^{\frac{n+1}{2}}}\Big]dx,
\end{eqnarray*}
where
\begin{align*}
C_{n,\delta}=\frac{(\sqrt{n+1+\delta}-\sqrt{n})^2}{2\pi}B\Big(\frac12,\frac{n+1}{2}\Big).
\end{align*}
\end{proposition}
\textbf{Proof.}
By Proposition \ref{lower-bound-Ra} and the monotonicity of $f$, we have that
\begin{eqnarray*}
\begin{split}
&-\mathcal{R}_af(x)\cdot\nabla f(x)
=-\mathcal{R}_af(x)\cdot\frac{x}{|x|} f'(|x|)\\
&~~~~~~~~~~\geq\frac{nB(\frac12,\frac{n+1}{2})}{\pi}\frac{f'(|x|)}{|x|^n}\Big[1-\frac{|x|^{n+1}}{(|x|^2+a^2)^{\frac{n+1}{2}}}\Big]\int\limits_0^{|x|} (f(|x|)-f(\varrho))\varrho^{n-1}
d\varrho,
\end{split}
\end{eqnarray*}
which implies that
\begin{align*}
&-\int\limits_{\mathbb{R}^n}\frac{\mathcal{R}_a f(x)\cdot\nabla f(x)}{|x|^{n+\delta}}dx
\geq\frac{B(\frac12,\frac{n+1}{2})}{n^{-1}\pi}\int\limits_{\mathbb{R}^n}\frac{f'(|x|)}{|x|^{2n+\delta}}\Big[1-\frac{|x|^{n+1}}{(|x|^2+a^2)^{\frac{n+1}{2}}}\Big]\int\limits_0^{|x|} (f(|x|)-f(\varrho))\varrho^{n-1}
d\varrho dx\\
&~~~~~~~~~~~~~~~~~~~~~~~~=\frac{B(\frac12,\frac{n+1}{2})}{n^{-1}\pi}\omega_{n-1}\underbrace{\int\limits_0^\infty\frac{f'(r)}{r^{n+1+\delta}}\Big[1-\frac{r^{n+1}}{(r^2+a^2)^{\frac{n+1}{2}}}\Big]\int\limits_0^{r} (f(r)-f(\varrho))\varrho^{n-1}
d\varrho dr}_{I}.
\end{align*}
Furthermore, by Fubini's theorem and integration by parts, we derive that
\begin{align*}
I&=\frac12\int\limits_0^\infty\varrho^{n-1}\int\limits_\varrho^\infty\frac{1}{r^{n+1+\delta}}\Big[1-\frac{r^{n+1}}{(r^2+a^2)^{\frac{n+1}{2}}}\Big]\frac{\partial}{\partial r}\Big((f(r)-f(\varrho))^2\Big) drd\varrho\\
&=\frac{n+1+\delta}{2}\iint\limits_{r\geq\varrho>0}\frac{(f(r)-f(\varrho))^2}{r^{n+2+\delta}}\Big[1-\frac{r^{n+1}}{(r^2+a^2)^{\frac{n+1}{2}}}\Big]\varrho^{n-1}d\varrho dr\\
&\ \ \ \
+\frac{n+1}{2}\iint\limits_{r\geq\varrho>0}\frac{a^2(f(r)-f(\varrho))^2}{r^{1+\delta}(r^2+a^2)^{\frac{n+3}{2}}}\varrho^{n-1}d\varrho dr\\
&\geq\frac{n+1+\delta}{2}\underbrace{\iint\limits_{r\geq\varrho>0}\frac{(f(r)-f(\varrho))^2}{r^{n+2+\delta}}\Big[1-\frac{r^{n+1}}{(r^2+a^2)^{\frac{n+1}{2}}}\Big]\varrho^{n-1}d\varrho dr}_{I_1}.
\end{align*}
For any $0<\alpha<1$, we note that
\begin{align*}
(f(r)-f(\varrho))^2\geq(1-\alpha)(f(r)-f(\varrho))^2+\Big(1-\frac{1}{\alpha}\Big)(f(\varrho)-f(0))^2.
\end{align*}
Then the integral $I_1$ can be estimated as
\begin{align*}
I_1&\geq(1-\alpha)\iint\limits_{r\geq\varrho>0}\frac{(f(r)-f(0))^2}{r^{n+2+\delta}}\Big[1-\frac{r^{n+1}}{(r^2+a^2)^{\frac{n+1}{2}}}\Big]\varrho^{n-1}d\varrho dr\\
&\ \ \ \ \
+(1-\frac{1}{\alpha})\iint\limits_{r\geq\varrho>0}\frac{(f(\varrho)-f(0))^2}{r^{n+2+\delta}}\Big[1-\frac{r^{n+1}}{(r^2+a^2)^{\frac{n+1}{2}}}\Big]\varrho^{n-1}d\varrho dr\\
&\geq(1-\alpha)\iint\limits_{r\geq\varrho>0}\frac{(f(r)-f(0))^2}{r^{n+2+\delta}}\Big[1-\frac{r^{n+1}}{(r^2+a^2)^{\frac{n+1}{2}}}\Big]\varrho^{n-1}d\varrho dr\\
&\ \ \ \ \
+\Big(1-\frac{1}{\alpha}\Big)\iint\limits_{r\geq\varrho>0}\frac{(f(\varrho)-f(0))^2}{r^{n+2+\delta}}\Big[1-\frac{\varrho^{n+1}}{(\varrho^2+a^2)^{\frac{n+1}{2}}}\Big]\varrho^{n-1}d\varrho dr\\
&=\Big(\frac{1-\alpha}{n}+\frac{1-\frac{1}{\alpha}}{n+1+\delta}\Big)\int\limits^\infty_0\frac{(f(r)-f(0))^2}{r^{2+\delta}}\Big[1-\frac{r^{n+1}}{(r^2+a^2)^{\frac{n+1}{2}}}\Big]dr\\
&=\Big[\frac{2n+1+\delta}{n(n+1+\delta)}-\Big(\frac{\alpha}{n}+\frac{\frac{1}{\alpha}}{n+1+\delta}\Big)\Big]\int\limits^\infty_0\frac{(f(r)-f(0))^2}{r^{2+\delta}}\Big[1-\frac{r^{n+1}}{(r^2+a^2)^{\frac{n+1}{2}}}\Big]dr,
\end{align*}
where we used the monotonicity of the weight $1-\frac{r^{n+1}}{(r^2+a^2)^{\frac{n+1}{2}}}$ in $r$ and the fact that $1-\frac{1}{\alpha}<0$.
Finally, by choosing $\alpha=\sqrt{\frac{n}{n+1+\delta}}\in(0,1)$, we obtain that
\begin{align*}
I_1&\geq\Big[\frac{2n+1+\delta}{n(n+1+\delta)}-\frac{2}{\sqrt{n(n+1+\delta)}}\Big]\int\limits^\infty_0\frac{(f(r)-f(0))^2}{r^{2+\delta}}\Big[1-\frac{r^{n+1}}{(r^2+a^2)^{\frac{n+1}{2}}}\Big]dr\\
&=\frac{(\sqrt{n+1+\delta}-\sqrt{n})^2}{n(n+1+\delta)}\int\limits^\infty_0\frac{(f(r)-f(0))^2}{r^{2+\delta}}\Big[1-\frac{r^{n+1}}{(r^2+a^2)^{\frac{n+1}{2}}}\Big]dr.
\end{align*}
Altogether, we can obtain that
\begin{align*}
-\int\limits_{\mathbb{R}^n}\frac{\mathcal{R}_a f(x)\cdot\nabla f(x)}{|x|^{n+\delta}}dx
\geq\frac{(\sqrt{n+1+\delta}-\sqrt{n})^2}{2\pi (B(\frac12,\frac{n+1}{2}))^{-1}}\int\limits_{\mathbb{R}^n}\frac{(f(x)-f(0))^2}{|x|^{n+1+\delta}}\Big[1-\frac{|x|^{n+1}}{(|x|^2+a^2)^{\frac{n+1}{2}}}\Big]dx,
\end{align*}
which concludes the proof of Proposition \ref{n-nonlinear-inequality-monotone}.  \hfill\hfill$\square$\vskip12pt
\subsection{Weighted inequalities for radial and non-monotone functions}
We next prove two new nonlinear weighted inequalities for the transform $\mathcal{R}_a$ acting on radial Schwartz function that is not necessarily monotone. The analogous
one-dimensional  inequalities for the transform $\mathcal{H}_a$ were established in \cite{[Liu-Zhang]}.
Our first new nonlinear weighted inequality for the transform $\mathcal{R}_a$ can be stated as
\begin{proposition}\label{n-nonlinear-inequality}
Let $n\geq2$ and $f:\mathbb{R}^n\rightarrow\mathbb{R}$ be a radial Schwartz function. Then
\begin{align*}
-\int\limits_{\mathbb{R}^n}\frac{\mathcal{R}_a f(x)\cdot\nabla f(x)}{|x|^{n}}dx
\geq\frac{B(\frac12,\frac{n+1}{2})}{2\pi}
\int\limits_{\mathbb{R}^n}\frac{(f(0)-f(x))^2}{|x|^{n+1}}\Big[1-\frac{|x|^{n+1}(|x|^2-na^2)}{(|x|^2+a^2)^{\frac{n+3}{2}}}\Big]dx.
\end{align*}
\end{proposition}
\textbf{Proof.} By \eqref{velocity-expression} and integration by parts, we have that
\begin{align}\label{nonlinear-term-representation}
-\mathcal{R}_af(x)\cdot\nabla f(x)
&=\frac{1}{\pi}\frac{f'(|x|)}{|x|^n}\int\limits_0^\infty f'(\varrho)\varrho^{n}
\Big[G\Big(\frac{\varrho}{|x|}\Big)-\widetilde{G}\Big(\frac{\varrho}{|x|},\frac{a}{|x|}\Big)\Big]d\varrho\nonumber\\
&=\frac{1}{\pi}\frac{f'(|x|)}{|x|^n}\int\limits_0^\infty \varrho^{n}\Big[G\Big(\frac{\varrho}{|x|}\Big)-\widetilde{G}\Big(\frac{\varrho}{|x|},\frac{a}{|x|}\Big)\Big]\frac{d}{d\varrho}(f(\varrho)-f(|x|))d\varrho\nonumber\\
&=-\frac{1}{\pi}\frac{f'(|x|)}{|x|^n}\int\limits_0^\infty \frac{\partial}{\partial\varrho}\Big[\varrho^{n}\Big(G\Big(\frac{\varrho}{|x|}\Big)-\widetilde{G}\Big(\frac{\varrho}{|x|},\frac{a}{|x|}\Big)\Big)\Big](f(\varrho)-f(|x|))d\varrho,
\end{align}
where in the last equality we used the boundary conditions
\begin{eqnarray*}
\lim_{\varrho\rightarrow0^+}
\varrho^{n}\Big[G\Big(\frac{\varrho}{|x|}\Big)-\widetilde{G}\Big(\frac{\varrho}{|x|},\frac{a}{|x|}\Big)\Big](f(\varrho)-f(|x|))=0
\end{eqnarray*}
and, by \eqref{transform-identity} and \eqref{transform-identity-G},
\begin{align*}
&\lim_{\varrho\rightarrow+\infty}\varrho^{n}\Big[G\Big(\frac{\varrho}{|x|}\Big)-\widetilde{G}\Big(\frac{\varrho}{|x|},\frac{a}{|x|}\Big)\Big](f(\varrho)-f(|x|))\\
&=|x|^{n+1}\lim_{\varrho\rightarrow+\infty}\Big[G\Big(\frac{|x|}{\varrho}\Big)-\widetilde{G}\Big(\frac{|x|}{\varrho},\frac{a}{\varrho}\Big)\Big]\frac{f(\varrho)-f(|x|)}{\varrho}=0.
\end{align*}
Here we also used the singularity order of the function $G(\cdot)$ at $1$ to show that the boundary term at the singularity $|x|$ vanishes, that is,
\begin{align*}
\lim_{\varrho\rightarrow|x|}\varrho^{n}G\Big(\frac{\varrho}{|x|}\Big)(f(\varrho)-f(|x|))
&=|x|^{n+1}\lim_{\varrho\rightarrow|x|}\Big(\frac{\varrho}{|x|}-1\Big)G\Big(\frac{\varrho}{|x|}\Big)\frac{f(\varrho)-f(|x|)}{\varrho-|x|}\\
&=-|x|^{n+1}f'(|x|)\lim_{\varrho\rightarrow|x|}\Big(\frac{\varrho}{|x|}-1\Big)\ln\Big|\frac{\varrho}{|x|}-1\Big|\\
&=0,
\end{align*}
where we used the first limit in \eqref{singularity-order}.

It then follows from Fubini's theorem and integration by parts again that
\begin{eqnarray}\label{integral-representation-1}
\begin{split}
-\int\limits_{\mathbb{R}^n}\frac{\mathcal{R}_af(x)\cdot\nabla f(x)}{|x|^{n}}dx
&=-\frac{\omega_{n-1}}{\pi}\int\limits_0^\infty \frac{f'(r)}{r^{n+1}}
\int\limits_0^\infty \frac{\partial}{\partial\varrho}\Big[\varrho^{n}\Big(G\Big(\frac{\varrho}{r}\Big)-\widetilde{G}\Big(\frac{\varrho}{r},\frac{a}{r}\Big)\Big)\Big](f(\varrho)-f(r))d\varrho dr\\
&=\frac{\omega_{n-1}}{2\pi}\int\limits_0^\infty\int\limits_0^\infty
 \frac{1}{r^{n+1}}\frac{\partial}{\partial\varrho}\Big[\varrho^{n}\Big(G\Big(\frac{\varrho}{r}\Big)-\widetilde{G}\Big(\frac{\varrho}{r},\frac{a}{r}\Big)\Big)\Big]\frac{\partial}{\partial r}((f(\varrho)-f(r))^2)drd\varrho\\
&=\frac{\omega_{n-1}}{2\pi}\int\limits_0^\infty\frac{1}{r^{n+1}}\frac{\partial}{\partial\varrho}\Big[\varrho^{n}\Big(G\Big(\frac{\varrho}{r}\Big)-\widetilde{G}\Big(\frac{\varrho}{r},\frac{a}{r}\Big)\Big)\Big]
(f(\varrho)-f(r))^2\Big)\Big|^{\infty}_{r\rightarrow0^+}d\varrho\\
&\
-\frac{\omega_{n-1}}{2\pi}\int\limits_0^\infty\int\limits_0^\infty
\frac{\partial}{\partial r}\Big[\frac{1}{r^{n+1}}\frac{\partial}{\partial\varrho}\Big(\varrho^{n}\Big(G\Big(\frac{\varrho}{r}\Big)-\widetilde{G}\Big(\frac{\varrho}{r},\frac{a}{r}\Big)\Big)\Big)\Big](f(\varrho)-f(r))^2d\varrho dr.
\end{split}
\end{eqnarray}
Here the singularity order of $G(\cdot)$ and $G'(\cdot)$ at $1$ also ensures that the boundary term at the singularity $\varrho$ is zero, that is,
\begin{align*}
\lim_{r\rightarrow\varrho}\frac{1}{r^{n+1}}\frac{\partial}{\partial\varrho}\Big[\varrho^{n}G\Big(\frac{\varrho}{r}\Big)\Big]
(f(\varrho)-f(r))^2
&=\frac{(f'(\varrho))^2}{\varrho^{n+1}}\lim_{r\rightarrow\varrho}(\varrho-r)^2\Big[n\varrho^{n-1}G\Big(\frac{\varrho}{r}\Big)+\frac{\varrho^n}{r}G'\Big(\frac{\varrho}{r}\Big)\Big]\\
&=\frac{(f'(\varrho))^2}{\varrho^2}\Big[n\lim_{r\rightarrow\varrho}(\varrho-r)^2G\Big(\frac{\varrho}{r}\Big)+\lim_{r\rightarrow\varrho}(\varrho-r)^2G'\Big(\frac{\varrho}{r}\Big)\Big]\\
&=0,
\end{align*}
where we used the limits in \eqref{singularity-order}.
We note that
\begin{eqnarray}\label{boundary-infinity}
\begin{split}
&\lim_{r\rightarrow\infty}\frac{1}{r^{n+1}}\frac{\partial}{\partial\varrho}\Big[\varrho^{n}\Big(G\Big(\frac{\varrho}{r}\Big)-\widetilde{G }\Big(\frac{\varrho}{r},\frac{a}{r}\Big)\Big)\Big]
(f(\varrho)-f(r))^2\\
&=n\varrho^{n-1}\lim_{r\rightarrow\infty}\frac{(f(\rho)-f(r))^2}{r^{n+1}}\Big[G\Big(\frac{\varrho}{r}\Big)-\widetilde{G}\Big(\frac{\varrho}{r},\frac{a}{r}\Big)\Big]\\
&\ \ \ \
+\varrho^n\lim_{r\rightarrow\infty}\frac{(f(\varrho)-f(r))^2}{r^{n+2}}
\Big[G'\Big(\frac{\varrho}{r}\Big)-\frac{\partial \widetilde{G}}{\partial\lambda}\Big(\frac{\varrho}{r},\frac{a}{r}\Big)\Big]\\
&=0,
\end{split}
\end{eqnarray}
and, by \eqref{transform-identity}, \eqref{transform-identity-G} and \eqref{0-value}, we also have that
\begin{align}\label{boundary-0}
&\lim_{r\rightarrow0^+}\frac{1}{r^{n+1}}\frac{\partial}{\partial\varrho}\Big[\varrho^{n}\Big(G\Big(\frac{\varrho}{r}\Big)-\widetilde{G}\Big(\frac{\varrho}{r},\frac{a}{r}\Big)\Big)\Big]
=\lim_{r\rightarrow0^+}\frac{\partial}{\partial\varrho}\Big[\frac{1}{\varrho}\Big(G\Big(\frac{r}{\varrho}\Big)-\widetilde{G}\Big(\frac{r}{\varrho},\frac{a}{\varrho}\Big)\Big)\Big]\nonumber\\
&=\lim_{r\rightarrow0^+}\Big[-\frac{1}{\varrho^2}\Big(G\Big(\frac{r}{\varrho}\Big)-\widetilde{G}\Big(\frac{r}{\varrho},\frac{a}{\varrho}\Big)\Big)
-\frac{r}{\varrho^3}\Big(G'\Big(\frac{r}{\varrho}\Big)
-\frac{\partial \widetilde{G}}{\partial\lambda}\Big(\frac{r}{\varrho},\frac{a}{\varrho}\Big)\Big)
+\frac{a}{\varrho^3}\frac{\partial \widetilde{G}}{\partial\mu}\Big(\frac{r}{\varrho},\frac{a}{\varrho}\Big)\Big]\nonumber\\
&=-\frac{1}{\varrho^2}\Big[G(0)-\widetilde{G}\Big(0,\frac{a}{\varrho}\Big)\Big]
+\frac{a}{\varrho^3}\frac{\partial \widetilde{G}}{\partial\mu}\Big(0,\frac{a}{\varrho}\Big)\nonumber\\
&=-\frac{1}{\varrho^2}\Big[B\Big(\frac12,\frac{n+1}{2}\Big)-\frac{B(\frac12,\frac{n+1}{2})}{(1+(\frac{a}{\varrho})^2)^{\frac{n+1}{2}}}\Big]
-\frac{a}{\varrho^3}\frac{(n+1)B(\frac12,\frac{n+1}{2})}{(1+(\frac{a}{\varrho})^2)^{\frac{n+3}{2}}}\frac{a}{\varrho}\nonumber\\
&=-B\Big(\frac{1}{2},\frac{n+1}{2}\Big)\frac{1}{\varrho^2}\Big[1-\frac{\varrho^{n+1}(\varrho^2-na^2)}{(\varrho^2+a^2)^{\frac{n+3}{2}}}\Big].
\end{align}
Thus,
\begin{align}\label{main-term-1}
&\frac{\omega_{n-1}}{2\pi}\int\limits_0^\infty\Big(\frac{1}{r^{n+1}}\frac{\partial}{\partial\varrho}\Big(\varrho^{n}\Big(G\Big(\frac{\varrho}{r}\Big)-\widetilde{G}\Big(\frac{\varrho}{r},\frac{a}{r}\Big)\Big)\Big)
(f(\varrho)-f(r))^2\Big)\Big|^{\infty}_{r\rightarrow0^+}d\varrho\nonumber\\
&=\frac{\omega_{n-1}}{2\pi} B\Big(\frac12,\frac{n+1}{2}\Big)\int\limits_0^\infty\frac{(f(\varrho)-f(0))^2}{\rho^2}\Big[1-\frac{\varrho^{n+1}(\varrho^2-na^2)}{(\varrho^2+a^2)^{\frac{n+3}{2}}}\Big] d\varrho\nonumber\\
&=\frac{B(\frac12,\frac{n+1}{2})}{2\pi}\int\limits_{\mathbb{R}^n}\frac{(f(0)-f(x))^2}{|x|^{n+1}}\Big[1-\frac{|x|^{n+1}(|x|^2-na^2)}{(|x|^2+a^2)^{\frac{n+3}{2}}}\Big]dx.
\end{align}
It is sufficient  to show that, for all $0<\varrho,r<\infty$ and $\varrho\neq r$,
\begin{eqnarray}\label{negative}
\frac{\partial}{\partial r}\Big[\frac{1}{r^{n+1}}\frac{\partial}{\partial\varrho}\Big(\varrho^{n}\Big(G\Big(\frac{\varrho}{r}\Big)-\widetilde{G}\Big(\frac{\varrho}{r},\frac{a}{r}\Big)\Big)\Big)\Big]
<0.
\end{eqnarray}
Indeed, direct calculations give that, for $0<\varrho<r$,
\begin{eqnarray}\label{1-varrho}
\frac{\partial}{\partial\varrho}\Big(\varrho^{n}\Big(G\Big(\frac{\varrho}{r}\Big)-\widetilde{G}\Big(\frac{\varrho}{r},\frac{a}{r}\Big)\Big)\Big)
=n\varrho^{n-1}\Big[G\Big(\frac{\varrho}{r}\Big)
-\widetilde{G}\Big(\frac{\varrho}{r},\frac{a}{r}\Big)\Big]
+\frac{\varrho^n}{r}\Big[G'\Big(\frac{\varrho}{r}\Big)
-\frac{\partial\widetilde{G}}{\partial\lambda}\Big(\frac{\varrho}{r},\frac{a}{r}\Big)\Big]
\end{eqnarray}
and then
\begin{align*}
&\frac{\partial}{\partial r}\Big[\frac{1}{r^{n+1}}\frac{\partial}{\partial\varrho}\Big(\varrho^{n}\Big(G\Big(\frac{\varrho}{r}\Big)-\widetilde{G}\Big(\frac{\varrho}{r},\frac{a}{r}\Big)\Big)\Big)\Big]\\
&=n\rho^{n-1}\frac{\partial}{\partial r}\Big[\frac{1}{r^{n+1}}\Big(G\Big(\frac{\rho}{r}\Big)-\widetilde{G}\Big(\frac{\varrho}{r},\frac{a}{r}\Big)\Big)\Big]
+\rho^{n}\frac{\partial}{\partial r}\Big[\frac{1}{r^{n+2}}\Big(G'\Big(\frac{\varrho}{r}\Big)-\frac{\partial \widetilde{G}}{\partial\lambda}\Big(\frac{\varrho}{r},\frac{a}{r}\Big)\Big)\Big]\\
&=-n(n+1)\frac{\rho^{n-1}}{r^{n+2}}\Big[G\Big(\frac{\varrho}{r}\Big)-\widetilde{G}\Big(\frac{\varrho}{r},\frac{a}{r}\Big)\Big]
-(2n+2)\frac{\varrho^{n}}{r^{n+3}}\Big[G'\Big(\frac{\varrho}{r}\Big)-\frac{\partial \widetilde{G}}{\partial\lambda}\Big(\frac{\varrho}{r},\frac{a}{r}\Big)\Big]\\
& \ \ \ \
-\frac{\varrho^{n+1}}{r^{n+4}}\Big[G''\Big(\frac{\varrho}{r}\Big)-\frac{\partial^2\widetilde{G}}{\partial^2\lambda}\Big(\frac{\varrho}{r},\frac{a}{r}\Big)\Big]
+na\frac{\varrho^{n-1}}{r^{n+3}}\frac{\partial \widetilde{G}}{\partial \mu}\Big(\frac{\varrho}{r},\frac{a}{r}\Big)
+\frac{a\varrho^n}{r^{n+4}}\frac{\partial^2\widetilde{G}}{\partial\lambda\partial \mu}\Big(\frac{\varrho}{r},\frac{a}{r}\Big),
\end{align*}
which together with \eqref{Taylor}, \eqref{1-derivative-Taylor}-\eqref{1-mu}, \eqref{2-lambda-mu} and \eqref{widetilde-a-2k-relation} implies that, for $0<\varrho<r$,
\begin{eqnarray}\label{negative-1}
\begin{split}
&\frac{\partial}{\partial r}\Big[\frac{1}{r^{n+1}}\frac{\partial}{\partial\varrho}\Big(\varrho^{n}\Big(G\Big(\frac{\varrho}{r}\Big)-\widetilde{G}\Big(\frac{\varrho}{r},\frac{a}{r}\Big)\Big)\Big)\Big]\\
&=-n(n+1)\frac{\rho^{n-1}}{r^{n+2}}\sum^\infty_{k=0}a_{2k}b_k\Big(\frac{a}{r}\Big)\Big(\frac{\varrho}{r}\Big)^{2k}
-(2n+2)\frac{\varrho^{n}}{r^{n+3}}\sum^\infty_{k=1}2ka_{2k}b_k\Big(\frac{a}{r}\Big)\Big(\frac{\varrho}{r}\Big)^{2k-1}\\
&\ \ \ \
-\frac{\varrho^{n+1}}{r^{n+4}}\sum^\infty_{k=1}2k(2k-1)a_{2k}b_k\Big(\frac{a}{r}\Big)\Big(\frac{\varrho}{r}\Big)^{2k-2}
-n(n+1)a\frac{\varrho^{n-1}}{r^{n+3}}\frac{a}{r}\sum^\infty_{k=0}\widetilde{a}_{2k}\widetilde{b}_k\Big(\frac{a}{r}\Big)\Big(\frac{\varrho}{r}\Big)^{2k}\\
&\ \ \ \
-(n+1)\frac{a\varrho^n}{r^{n+4}}\frac{a}{r}\sum^\infty_{k=1}2k\widetilde{a}_{2k}\widetilde{b}_k\Big(\frac{a}{r}\Big)\Big(\frac{\varrho}{r}\Big)^{2k-1}\\
&=-\frac{\rho^{n-1}}{r^{n+2}}\Big[\sum^\infty_{k=0}(2k+n)(2k+n+1)a_{2k}b_k\Big(\frac{a}{r}\Big)\Big(\frac{\varrho}{r}\Big)^{2k}
+(n+1)\Big(\frac{a}{r}\Big)^2\sum^\infty_{k=1}(2k+n)\widetilde{a}_{2k}\widetilde{b}_k\Big(\frac{a}{r}\Big)\Big(\frac{\varrho}{r}\Big)^{2k}\Big]\\
&=-\frac{\rho^{n-1}}{r^{n+2}}\Big[n(n+1)a_0b_0\Big(\frac{a}{r}\Big)
+\sum^\infty_{k=1}(2k+n)(2k+n+1)a_{2k}\Big[b_k\Big(\frac{a}{r}\Big)
+(2k+1)\Big(\frac{a}{r}\Big)^2\widetilde{b}_k\Big(\frac{a}{r}\Big)\Big]\Big(\frac{\varrho}{r}\Big)^{2k}\Big].
\end{split}
\end{eqnarray}
Furthermore, for $0<r<\varrho$, by \eqref{transform-identity} and \eqref{transform-identity-G}, we derive that
\begin{eqnarray}\label{2-varrho}
\begin{split}
&\frac{1}{r^{n+1}}\frac{\partial}{\partial\varrho}\Big[\varrho^{n}\Big(G\Big(\frac{\varrho}{r}\Big)-\widetilde{G}\Big(\frac{\varrho}{r},\frac{a}{r}\Big)\Big)\Big]
=\frac{\partial}{\partial\varrho}\Big[\frac{1}{\varrho}\Big(G\Big(\frac{r}{\varrho}\Big)-\widetilde{G}\Big(\frac{r}{\varrho},\frac{a}{\varrho}\Big)\Big)\Big]\\
&=-\frac{1}{\varrho^2}\Big[G\Big(\frac{r}{\varrho}\Big)-\widetilde{G}\Big(\frac{r}{\varrho},\frac{a}{\varrho}\Big)\Big]
-\frac{r}{\varrho^3}\Big[G'\Big(\frac{r}{\varrho}\Big)-\frac{\partial \widetilde{G}}{\partial \lambda}\Big(\frac{r}{\varrho},\frac{a}{\varrho}\Big)\Big]+\frac{a}{\varrho^3}\frac{\partial \widetilde{G}}{\partial \mu}\Big(\frac{r}{\varrho},\frac{a}{\varrho}\Big)
\end{split}
\end{eqnarray}
and, by \eqref{1-derivative-Taylor}, \eqref{2-derivative-Taylor} and \eqref{2-lambda-mu}, we further have that, for $0<r<\varrho$,
\begin{align}\label{negative-2}
&\frac{\partial}{\partial r}\Big[\frac{1}{r^{n+1}}\frac{\partial}{\partial\varrho}\Big(\varrho^{n}\Big(G\Big(\frac{\varrho}{r}\Big)-\widetilde{G}\Big(\frac{\varrho}{r},\frac{a}{r}\Big)\Big)\Big)\Big]\nonumber\\
&=-\frac{2}{\varrho^3}\Big[G'\Big(\frac{r}{\varrho}\Big)
-\frac{\partial \widetilde{G}}{\partial\lambda}\Big(\frac{r}{\varrho},\frac{a}{\varrho}\Big)\Big]
-\frac{r}{\varrho^4}\Big[G''\Big(\frac{r}{\varrho}\Big)-\frac{\partial^2 \widetilde{G}}{\partial\lambda^2}\Big(\frac{r}{\varrho},\frac{a}{\varrho}\Big)\Big]
+\frac{a}{\varrho^4}\frac{\partial^2 \widetilde{G}}{\partial\lambda\partial\mu}\Big(\frac{r}{\varrho},\frac{a}{\varrho}\Big)\nonumber\\
&=-\frac{2}{\varrho^3}\sum^\infty_{k=1}2ka_{2k}b_k\Big(\frac{a}{\varrho}\Big)\Big(\frac{r}{\varrho}\Big)^{2k-1}
-\frac{r}{\varrho^4}\sum^\infty_{k=1}2k(2k-1)a_{2k}b_k\Big(\frac{a}{\varrho}\Big)\Big(\frac{r}{\varrho}\Big)^{2k-2}\nonumber\\
&\ \ \ \ \ \ \
-(n+1)\frac{a}{\varrho^4}\frac{a}{\varrho}\sum^\infty_{k=1}2k\widetilde{a}_{2k}\widetilde{b}_k\Big(\frac{a}{\varrho}\Big)\Big(\frac{r}{\varrho}\Big)^{2k-1}\nonumber\\
&=-\frac{1}{\varrho^3}\sum^\infty_{k=1}2k(2k+1)a_{2k}b_k\Big(\frac{a}{\varrho}\Big)\Big(\frac{r}{\varrho}\Big)^{2k-1}
-(n+1)\frac{a}{\varrho^4}\frac{a}{\varrho}\sum^\infty_{k=1}2k\widetilde{a}_{2k}\widetilde{b}_k\Big(\frac{a}{\varrho}\Big)\Big(\frac{r}{\varrho}\Big)^{2k-1}\nonumber\\
&=-\frac{1}{\varrho^3}\sum^\infty_{k=1}2k
\Big[(2k+1)a_{2k}b_k\Big(\frac{a}{\varrho}\Big)
+(n+1)\widetilde{a}_{2k}\Big(\frac{a}{\varrho}\Big)^2\widetilde{b}_k\Big(\frac{a}{\varrho}\Big)\Big]\Big(\frac{r}{\varrho}\Big)^{2k-1}\nonumber\\
&=-\frac{1}{\varrho^3}\sum^\infty_{k=1}2k(2k+1)a_{2k}
\Big[b_k\Big(\frac{a}{\varrho}\Big)
+(2k+n+1)\Big(\frac{a}{\varrho}\Big)^2\widetilde{b}_k\Big(\frac{a}{\varrho}\Big)\Big]\Big(\frac{r}{\varrho}\Big)^{2k-1}.
\end{align}
To show \eqref{negative}, by \eqref{negative-1} and \eqref{negative-2}, it is sufficient to prove the following lemma.
\begin{lemma}\label{Key-inequality-positivity}
Let $n\geq2$  be an integer. Then, for any $k\geq1$ and  every $\mu>0$,
\begin{align}\label{positivity-1}
b_k(\mu)+(2k+1)\mu^2\widetilde{b}_k(\mu)>0
\end{align}
and
\begin{align}\label{positivity-2}
b_k(\mu)+(2k+n+1)\mu^2\widetilde{b}_k(\mu)>0
\end{align}
where $b_k(\mu)$ and $\widetilde{b}_k(\mu)$ are defined as \eqref{bkmu-1} and \eqref{bkmu-2}, respectively.
\end{lemma}
By Lemma \ref{Key-inequality-positivity}, combining \eqref{negative-1} and \eqref{negative-2} gives \eqref{negative}, which implies that
\begin{eqnarray*}
\int\limits_0^\infty\int\limits_0^\infty
\frac{\partial}{\partial r}\Big[\frac{1}{r^{n+1}}\frac{\partial}{\partial\varrho}\Big(\varrho^{n}\Big(G\Big(\frac{\varrho}{r}\Big)-\widetilde{G}\Big(\frac{\varrho}{r},\frac{a}{r}\Big)\Big)\Big)\Big](f(\varrho)-f(r))^2d\varrho dr\leq0.
\end{eqnarray*}
This negative integral along with
\eqref{integral-representation-1} and \eqref{main-term-1} completes the proof of Proposition \ref{n-nonlinear-inequality}.
\hfill\hfill$\square$\vskip12pt
Now we give the proof of Lemma \ref{Key-inequality-positivity}.

\textbf{Proof of Lemma \ref{Key-inequality-positivity}.} We first prove the inequality \eqref{positivity-1}.
Before that, we will prove the identity
\begin{align}\label{identity-2k-1}
(1+x)\mathcal{P}_{n,k}(x)-(2k+1)x\mathcal{Q}_{n,k}(x)=\mathcal{P}_{n+2,k}(x),
\end{align}
where the polynomial $\mathcal{P}_{N,K}(x)$ is defined as \eqref{polynomial-definition}, and $\mathcal{Q}_{n,k}(x)$ is defined by
\begin{align*}
\mathcal{Q}_{n,k}(x):=1+\displaystyle\sum^{k}_{j=1}(-1)^j\binom{k}{j}\Big(\displaystyle\prod^{j-1}_{l=0}\frac{n+2(k-l)}{3+2l}\Big)x^{j}.
\end{align*}
Indeed, for $1\leq j\leq k$, we denote
\begin{align*}
q_j:=\displaystyle\prod^{j-1}_{l=0}\frac{n+2(k-l)}{1+2l}.
\end{align*}
Clearly,
\begin{align}\label{relation}
q_j=\frac{n+2k-2j+2}{2j-1}q_{j-1},~\displaystyle\prod^{j-1}_{l=0}\frac{1+2l}{3+2l}=\frac{1}{2j+1}.
\end{align}
It follows that
\begin{align*}
&(1+x)\mathcal{P}_{n,k}(x)-(2k+1)x\mathcal{Q}_{n,k}(x)\\
&=(1+x)\Big[1+\displaystyle\sum^{k}_{j=1}(-1)^j\binom{k}{j}q_jx^{j}\Big]-(2k+1)x\Big[1+\displaystyle\sum^{k}_{j=1}(-1)^j\frac{q_j}{2j+1}\binom{k}{j}x^{j}\Big]\\
&=1-2kx+\displaystyle\sum^{k}_{j=1}(-1)^j\binom{k}{j}q_jx^{j}-\displaystyle\sum^{k}_{j=1}(-1)^j\frac{2k-2j}{2j+1}\binom{k}{j}q_jx^{j+1}\\
&=1-k(n+2+2k)x+\displaystyle\sum^{k}_{j=2}(-1)^j\Big[\binom{k}{j}q_j+\frac{2k-2j+2}{2j-1}\binom{k}{j-1}q_{j-1}\Big]x^{j}\\
&=1-k(n+2+2k)x+\displaystyle\sum^{k}_{j=2}(-1)^j\binom{k}{j}q_j\frac{n+2k+2}{n+2k-2j+2}x^{j}\\
&=1+\displaystyle\sum^{k}_{j=1}(-1)^j\binom{k}{j}\Big(\displaystyle\prod^{j-1}_{l=0}\frac{(n+2)+2(k-l)}{1+2l}\Big)x^{j}\\
&=\mathcal{P}_{n+2,k}(x),
\end{align*}
which is exactly \eqref{identity-2k-1}.
Then, by virtue of \eqref{bkmu-1}, \eqref{bkmu-2} and \eqref{identity-2k-1}, we have that, for any $k\geq1$,
\begin{align*}
b_k(\mu)+(2k+1)\mu^2\widetilde{b}_k(\mu)
&=1-\frac{\mathcal{P}_{n,k}(\mu^2)}{(1+\mu^2)^{\frac{n+1}{2}+2k}}+(2k+1)\mu^2\cdot\frac{\mathcal{Q}_{n,k}(\mu^2)}{(1+\mu^2)^{\frac{n+3}{2}+2k}}\\
&=1-\frac{(1+\mu^2)\mathcal{P}_{n,k}(\mu^2)-(2k+1)\mu^2\mathcal{Q}_{n,k}(\mu^2)}{(1+\mu^2)^{\frac{n+3}{2}+2k}}\\
&=1-\frac{\mathcal{P}_{n+2,k}(\mu^2)}{(1+\mu^2)^{\frac{n+3}{2}+2k}}>0,
\end{align*}
where the last inequality follows from the special case of $N=n+2$ and $K=k$ in Lemma \ref{Key-inequality}.

Next we show the inequality \eqref{positivity-2}. Similarly, we can derive that
\begin{align}\label{identity-2k-n}
(1+x)\mathcal{P}_{n,k}(x)-(2k+n+1)x\mathcal{Q}_{n,k}(x).
=\mathcal{P}_{n-2,k+1}(x).
\end{align}
Indeed, again by \eqref{relation}, we have that
\begin{align*}
&(1+x)\mathcal{P}_{n,k}(x)-(2k+n+1)x\mathcal{Q}_{n,k}(x)\\
&=(1+x)\Big[1+\displaystyle\sum^{k}_{j=1}(-1)^j\binom{k}{j}q_jx^{j}\Big]-(2k+n+1)x\Big[1+\displaystyle\sum^{k}_{j=1}(-1)^j\binom{k}{j}\frac{q_j}{2j+1}x^{j}\Big]\\
&=1-(2k+n)x+\displaystyle\sum^{k}_{j=1}(-1)^j\binom{k}{j}q_jx^{j}-\displaystyle\sum^{k}_{j=1}(-1)^j\frac{2k-2j+n}{2j+1}\binom{k}{j}q_jx^{j+1}\\
&=1-(k+1)q_1x+\displaystyle\sum^{k}_{j=2}(-1)^j\Big[\binom{k}{j}q_j+\frac{2k-2j+2+n}{2j-1}\binom{k}{j-1}q_{j-1}\Big]x^{j}+\frac{(-1)^{k+1}nq_k}{2k+1}x^{k+1}\\
&=1-(k+1)q_1x+\displaystyle\sum^{k}_{j=2}(-1)^j\binom{k}{j}q_j\frac{k+1}{k-j+1}x^{j}+(-1)^{k+1}q_{k+1}x^{k+1}\\
&=1+\displaystyle\sum^{k+1}_{j=1}(-1)^j\binom{k+1}{j}q_jx^{j}\\
&=1+\displaystyle\sum^{k+1}_{j=1}(-1)^j\binom{k+1}{j}\Big(\displaystyle\prod^{j-1}_{l=0}\frac{(n-2)+2((k+1)-l)}{1+2l}\Big)x^{j}\\
&=\mathcal{P}_{n-2,k+1}(x),
\end{align*}
which is the desired identity.
Then, by virtue of \eqref{bkmu-1}, \eqref{bkmu-2} and \eqref{identity-2k-n}, we have that, for any $k\geq1$,
\begin{align*}
b_k(\mu)+(2k+n+1)\mu^2\widetilde{b}_k(\mu)
&=1-\frac{\mathcal{P}_{n,k}(\mu^2)}{(1+\mu^2)^{\frac{n+1}{2}+2k}}+(2k+n+1)\mu^2\cdot\frac{\mathcal{Q}_{n,k}(\mu^2)}{(1+\mu^2)^{\frac{n+3}{2}+2k}}\\
&=1-\frac{(1+\mu^2)\mathcal{P}_{n,k}(\mu^2)-(2k+n+1)\mu^2\mathcal{Q}_{n,k}(\mu^2)}{(1+\mu^2)^{\frac{n+3}{2}+2k}}\\
&=1-\frac{\mathcal{P}_{n-2,k+1}(\mu^2)}{(1+\mu^2)^{\frac{(n-2)+1}{2}+2(k+1)}}>0,
\end{align*}
where the last inequality follows from the special case of $N=n-2>-1$ and $K=k+1$ in Lemma \ref{Key-inequality}. \hfill\hfill$\square$\vskip12pt

Since the nonlinear inequality in Proposition \ref{n-nonlinear-inequality} is not directly useful for implying blow-ups of \eqref{M-KS} for general radial initial data, we establish the next weighed inequality to show the finite-time blow-up.
\begin{proposition}\label{nonlinear-weighted-inequality-exponential}
Let $n\geq2$ and $a>0$.
Let $f:\mathbb{R}^n\rightarrow\mathbb{R}$ be a radial Schwartz function. Then
\begin{eqnarray*}
\begin{split}
-\int\limits_{\mathbb{R}^n}\frac{\mathcal{R}_a f(x)\cdot\nabla f(x)}{|x|^ne^{|x|}}dx
&\geq\frac{B(\frac12,\frac{n+1}{2})}{4\pi}
\int\limits_{\mathbb{R}^n}\frac{(f(0)-f(x))^2}{|x|^{n+1}}\Big[1-\frac{|x|^{n+1}(|x|^2-na^2)}{(|x|^2+a^2)^{\frac{n+3}{2}}}\Big]dx
-C_{n,a}\|f\|^2_{L^\infty},
\end{split}
\end{eqnarray*}
where the constant $C_{n,a}$ depends only on $n$ and $a$.
\end{proposition}
\textbf{Proof.}  We only need to modify the proof of Proposition \ref{n-nonlinear-inequality}.
Similar to \eqref{integral-representation-1}, by \eqref{nonlinear-term-representation}, \eqref{boundary-infinity}, \eqref{boundary-0} and \eqref{negative},
we can derive that
\begin{align}\label{representation-2}
&-\int\limits_{\mathbb{R}^n}\frac{\mathcal{R}_af(x)\cdot\nabla f(x)}{|x|^{n}}e^{-|x|}dx\nonumber\\
&=-\frac{\omega_{n-1}}{\pi}\int\limits_0^\infty \frac{f'(r)e^{-r}}{r^{n+1}}
\int\limits_0^\infty \frac{\partial}{\partial\rho}\Big[\varrho^{n}\Big(G\Big(\frac{\varrho}{r}\Big)-\widetilde{G}\Big(\frac{\varrho}{r},\frac{a}{r}\Big)\Big)\Big](f(\varrho)-f(r))d\varrho dr\nonumber\\
&=\frac{\omega_{n-1}}{2\pi}\int\limits_0^\infty\Big(\frac{e^{-r}}{r^{n+1}}\frac{\partial}{\partial\varrho}\Big(\varrho^{n}\Big(G\Big(\frac{\varrho}{r}\Big)-\widetilde{G}\Big(\frac{\varrho}{r},\frac{a}{r}\Big)\Big)\Big)
(f(\varrho)-f(r))^2\Big)\Big|^{\infty}_{r\rightarrow0^+}d\varrho\nonumber\\
&\ \ \
-\frac{\omega_{n-1}}{2\pi}\int\limits_0^\infty\int\limits_0^\infty
\frac{\partial}{\partial r}\Big[\frac{e^{-r}}{r^{n+1}}\frac{\partial}{\partial\varrho}\Big(\varrho^{n}\Big(G\Big(\frac{\varrho}{r}\Big)-\widetilde{G}\Big(\frac{\varrho}{r},\frac{a}{r}\Big)\Big)\Big)\Big](f(\varrho)-f(r))^2d\varrho dr\nonumber\\
&=\frac{\omega_{n-1}}{2\pi} B\Big(\frac12,\frac{n+1}{2}\Big)\int\limits_0^\infty\frac{(f(\varrho)-f(0))^2}{\rho^2}\Big[1-\frac{\varrho^{n+1}(\varrho^2-na^2)}{(\varrho^2+a^2)^{\frac{n+3}{2}}} \Big] d\varrho\nonumber\\
&\ \ \
-\frac{\omega_{n-1}}{2\pi}\int\limits_0^\infty\int\limits_0^\infty
e^{-r}\frac{\partial}{\partial r}\Big[\frac{1}{r^{n+1}}\frac{\partial}{\partial\varrho}\Big(\varrho^{n}\Big(G\Big(\frac{\varrho}{r}\Big)-\widetilde{G}\Big(\frac{\varrho}{r},\frac{a}{r}\Big)\Big)\Big)\Big](f(\varrho)-f(r))^2d\varrho dr\nonumber\\
&\ \ \
+\frac{\omega_{n-1}}{2\pi}\int\limits_0^\infty\int\limits_0^\infty
\frac{e^{-r}}{r^{n+1}}\frac{\partial}{\partial\varrho}\Big[\varrho^{n}\Big(G\Big(\frac{\varrho}{r}\Big)-\widetilde{G}\Big(\frac{\varrho}{r},\frac{a}{r}\Big)\Big)\Big](f(\varrho)-f(r))^2d\varrho dr\nonumber\\
&\geq
\frac{\omega_{n-1}}{2\pi} B\Big(\frac12,\frac{n+1}{2}\Big)\int\limits_0^\infty\frac{(f(\varrho)-f(0))^2}{\rho^2}\Big[1-\frac{\varrho^{n+1}(\varrho^2-na^2)}{(\varrho^2+a^2)^{\frac{n+3}{2}}} \Big] d\varrho\nonumber\\
&\ \ \
+\frac{\omega_{n-1}}{2\pi}\underbrace{\int\limits_0^\infty\int\limits_0^\infty
\frac{e^{-r}}{r^{n+1}}\frac{\partial}{\partial\varrho}\Big[\varrho^{n}\Big(G\Big(\frac{\varrho}{r}\Big)-\widetilde{G}\Big(\frac{\varrho}{r},\frac{a}{r}\Big)\Big)\Big](f(\varrho)-f(r))^2d\varrho dr}_{\mathcal{I}}.
\end{align}
By virtue of \eqref{1-varrho} and \eqref{2-varrho}, we rewrite the integral $\mathcal{I}$ as
\begin{align*}
\mathcal{I}
&=\iint\limits_{0<\varrho<r}\frac{\varrho^{n-1}(f(\varrho)-f(r))^2}{r^{n+1}e^{r}}
\Big[n\Big(G\Big(\frac{\varrho}{r}\Big)-\widetilde{G}\Big(\frac{\varrho}{r},\frac{a}{r}\Big)\Big)
+\frac{\varrho}{r}\Big(G'\Big(\frac{\varrho}{r}\Big)
-\frac{\partial \widetilde{G}}{\partial\lambda}\Big(\frac{\varrho}{r},\frac{a}{r}\Big)\Big)\Big]d\varrho dr\\
&\
-\iint\limits_{0<r<\varrho}\frac{(f(\varrho)-f(r))^2}{\varrho^2e^r}
\Big[\Big(G\Big(\frac{r}{\varrho}\Big)-\widetilde{G}\Big(\frac{r}{\varrho},\frac{a}{\varrho}\Big)\Big)
+\frac{r}{\varrho}\Big(G'\Big(\frac{r}{\varrho}\Big)
-\frac{\partial \widetilde{G}}{\partial\lambda}\Big(\frac{r}{\varrho},\frac{a}{\varrho}\Big)\Big)\Big]d\varrho dr\\
&\ \ \ \ \ \ \ \ \
+a\iint\limits_{0<r<\varrho}\frac{(f(\varrho)-f(r))^2}{\varrho^3e^r}
\frac{\partial \widetilde{G}}{\partial\mu}\Big(\frac{r}{\varrho},\frac{a}{\varrho}\Big)d\varrho dr,
\end{align*}
which follows from $G(\lambda)-\widetilde{G}(\lambda,\mu)>0$ for any $\lambda\in(0,1)$ and every $\mu>0$ and a change of variables that
\begin{align*}
\mathcal{I}
&\geq\iint\limits_{0<\varrho<r}\frac{\varrho^{n}(f(\varrho)-f(r))^2}{r^{n+2}e^{r}}
\Big[G'\Big(\frac{\varrho}{r}\Big)
-\frac{\partial \widetilde{G}}{\partial\lambda}\Big(\frac{\varrho}{r},\frac{a}{r}\Big)\Big]d\varrho dr
+a\iint\limits_{0<r<\varrho}\frac{(f(\varrho)-f(r))^2}{\varrho^3e^r}
\frac{\partial \widetilde{G}}{\partial\mu}\Big(\frac{r}{\varrho},\frac{a}{\varrho}\Big)d\varrho dr\\
&\
-\iint\limits_{0<r<\varrho}\frac{(f(\varrho)-f(r))^2}{\varrho^2e^r}
\Big[\Big(G\Big(\frac{r}{\varrho}\Big)-\widetilde{G}\Big(\frac{r}{\varrho},\frac{a}{\varrho}\Big)\Big)
+\frac{r}{\varrho}\Big(G'\Big(\frac{r}{\varrho}\Big)
-\frac{\partial \widetilde{G}}{\partial\lambda}\Big(\frac{r}{\varrho},\frac{a}{\varrho}\Big)\Big)\Big]d\varrho dr \\
&=-\underbrace{\iint\limits_{0<\varrho<r}\Big(e^{-\varrho}-e^{-r}\frac{\varrho^{n-1}}{r^{n-1}}\Big)\frac{(f(\varrho)-f(r))^2}{r^2}
\frac{\varrho}{r}\Big(G'\Big(\frac{\varrho}{r}\Big)-\frac{\partial \widetilde{G}}{\partial\lambda}\Big(\frac{\varrho}{r},\frac{a}{r}\Big)\Big)d\varrho dr}_{\mathcal{J}}\\
&\ \ \ \ \
-\underbrace{\iint\limits_{0<r<\varrho}\frac{(f(\varrho)-f(r))^2}{\varrho^2e^r}
\Big[G\Big(\frac{r}{\varrho}\Big)-\widetilde{G}\Big(\frac{r}{\varrho},\frac{a}{\varrho}\Big)
\Big]d\varrho dr}_{\mathcal{K}}
+\underbrace{a\iint\limits_{0<r<\varrho}\frac{(f(\varrho)-f(r))^2}{\varrho^3e^r}
\frac{\partial \widetilde{G}}{\partial\mu}\Big(\frac{r}{\varrho},\frac{a}{\varrho}\Big)d\varrho dr}_{\mathcal{L}}.
\end{align*}
Next we estimate the integrals $\mathcal{J}$, $\mathcal{K}$ and $\mathcal{L}$ separately.
By Lemma \ref{Key-inequality}, we can obtain a surprising fact that the coefficient $b_k(\mu)$ is uniformly bounded in $k\geq0$ and $\mu>0$, that is, for any $k\geq0$ and $\mu>0$,
\begin{align}\label{uniformly-bounded}
0<b_k(\mu)<2.
\end{align}
Furthermore, by the Mean Value Theorem, we have that, for $0<\varrho<r$,
\begin{align*}
e^{-\varrho}-e^{-r}\Big(\frac{\varrho}{r}\Big)^{n-1}
&=\rho^{n-1}\Big(\frac{e^{-\varrho}}{\varrho^{n-1}}-\frac{e^{-r}}{r^{n-1}}\Big)\\
&\leq \frac{\varrho^{n-1}}{e^\varrho}
\Big(\varrho^{1-n}+\frac{n-1}{\varrho^{n}}\Big)(r-\varrho)\\
&=\frac{r-\varrho}{e^\varrho}
+\frac{n-1}{e^\varrho}\Big(\frac{r}{\varrho}-1\Big),
\end{align*}
which along with \eqref{1-derivative-Taylor} and \eqref{uniformly-bounded} implies that
\begin{align*}
\mathcal{J}
&=\sum^\infty_{k=1}2ka_{2k}
\iint\limits_{0<\varrho<r}\Big[e^{-\varrho}-e^{-r}\Big(\frac{\varrho}{r}\Big)^{n-1}\Big]
\frac{(f(\varrho)-f(r))^2}{r^2}b_k\Big(\frac{a}{\varrho}\Big)\Big(\frac{\varrho}{r}\Big)^{2k}d\varrho dr\\
&\leq\sum^\infty_{k=1}4ka_{2k}
\iint\limits_{0<\varrho<r}\Big[e^{-\varrho}-e^{-r}\Big(\frac{\varrho}{r}\Big)^{n-1}\Big]
\frac{(f(\varrho)-f(r))^2}{r^2}\Big(\frac{\varrho}{r}\Big)^{2k}d\varrho dr\\
&\leq\sum^\infty_{k=1}4ka_{2k}
\underbrace{\iint\limits_{0<\varrho<r}
(r-\varrho)\frac{(f(\varrho)-f(r))^2}{r^2e^\varrho}\Big(\frac{\varrho}{r}\Big)^{2k}d\varrho dr}_{\mathcal{J}_{1}}\\
&\
+(n-1)\sum^\infty_{k=1}4ka_{2k}
\underbrace{\iint\limits_{0<\varrho<r}
\Big(\frac{r}{\varrho}-1\Big)\frac{(f(\varrho)-f(r))^2}{r^2e^\varrho}\Big(\frac{\varrho}{r}\Big)^{2k}d\varrho dr}_{\mathcal{J}_{2}}.
\end{align*}
The integral $\mathcal{J}_1$ can be estimated directly as, for any $k\geq1$,
\begin{align}\label{J-1-estimate}
\mathcal{J}_1
&\leq4\|f\|^2_{L^\infty}\iint\limits_{0<\varrho<r}
\frac{r-\varrho}{r^2e^\varrho}\Big(\frac{\varrho}{r}\Big)^{2k}d\varrho
dr\nonumber\\
&=4\|f\|^2_{L^\infty}\int\limits^\infty_0e^{-\varrho}\varrho^{2k}\Big(\int\limits^\infty_\varrho\frac{1-\frac{\varrho}{r}}{r^{2k+1}}dr\Big)
d\varrho\nonumber\\
&=\frac{2\|f\|^2_{L^\infty}}{k(2k+1)}<\frac{1}{k^2}\|f\|^2_{L^\infty}.
\end{align}
For the integral $\mathcal{J}_2$, we split it into two parts:
\begin{align*}
\mathcal{J}_2
&=\Big(\underbrace{\iint\limits_{0<\varrho\leq\frac{r}{2}}
}_{\mathcal{J}_{21}}
+\underbrace{\iint\limits_{\frac{r}{2}<\varrho<r}
}_{\mathcal{J}_{22}}\Big)\Big(\frac{r}{\varrho}-1\Big)\frac{(f(\varrho)-f(r))^2}{r^2e^\varrho}\Big(\frac{\varrho}{r}\Big)^{2k}d\varrho dr.
\end{align*}
Using the region restriction $0<\frac{\varrho}{r}\leq\frac12$ and the elementary inequality
\begin{eqnarray}\label{elementary-inequality}
(\alpha-\beta)^2\leq2(\alpha-\gamma)^2+2(\beta-\gamma)^2~{\rm  for}~ \alpha,\beta,\gamma\in\mathbb{R},
\end{eqnarray}
the part $\mathcal{J}_{21}$ can be bounded as, for any $k\geq1$,
\begin{align*}
\mathcal{J}_{21}
&\leq\Big(\frac{1}{2}\Big)^{2k-1}
\iint\limits_{0<\varrho\leq\frac{r}{2}}
\Big(\frac{r}{\varrho}-1\Big)\frac{(f(\varrho)-f(r))^2}{r^2e^\varrho}\cdot\frac{\varrho}{r}d\varrho dr\\
&=\frac{1}{2^{2k-1}}
\iint\limits_{0<\rho\leq\frac{r}{2}}
\Big(1-\frac{\varrho}{r}\Big)\frac{(f(\varrho)-f(r))^2}{r^2e^\varrho}d\varrho dr
\leq\frac{1}{2^{2k-1}}
\iint\limits_{0<\varrho\leq\frac{r}{2}}
\frac{(f(\varrho)-f(r))^2}{r^2e^\varrho}d\varrho dr\\
&\leq\frac{1}{2^{2k-1}}\int\limits^\infty_0\frac{(f(\varrho)-f(0))^2}{\varrho e^\varrho}d\varrho
+\frac{1}{2^{2k-2}}\int\limits^\infty_0\frac{(f(0)-f(r))^2}{r^{2}}(1-e^{-\frac{r}{2}}) dr.
\end{align*}
By \eqref{elementary-inequality} again, we further estimate the part $\mathcal{J}_{22}$ as
\begin{align*}
\mathcal{J}_{22}
&\leq2\iint\limits_{\frac{r}{2}<\varrho<r}
\Big[\frac{r}{\varrho}-1\Big]\frac{(f(\varrho)-f(0))^2}{r^2e^{\frac{\varrho}{2}}}\Big(\frac{\varrho}{r}\Big)^{2k}d\varrho dr
+2\iint\limits_{\frac{r}{2}<\varrho<r}
\Big[\frac{r}{\varrho}-1\Big]\frac{(f(0)-f(r))^2}{r^2e^{\frac{r}{2}}}\Big(\frac{\varrho}{r}\Big)^{2k}d\varrho dr\\
&=2\int\limits_0^\infty \frac{\varrho^{2k}(f(\varrho)-f(0))^2}{e^{\frac{\varrho}{2}}}\Big[\int\limits^{2\varrho}_\varrho\frac{\frac{r}{\varrho}-1}{r^{2k+2}}dr\Big]d\varrho
+2\int\limits_0^\infty \frac{(f(r)-f(0))^2}{r^{2k+2}e^{\frac{r}{2}}}\Big[\int\limits^r_{\frac{r}{2}}\varrho^{2k}\Big(\frac{r}{\varrho}-1\Big)d\varrho\Big]dr\\
&=4\Big[\frac{1-2^{-2k}}{2k}-\frac{1-2^{-2k-1}}{2k+1}\Big]\int\limits^\infty_0\frac{(f(\varrho)-f(0))^2}{\varrho e^{\frac{\varrho}{2}}}d\varrho\leq\frac{2}{k(2k+1)}\int\limits^\infty_0\frac{(f(\varrho)-f(0))^2}{\rho e^{\frac{\varrho}{2}}}d\varrho.
\end{align*}
Thus, using $2^{2k-2}\geq k^2$ for $k\geq1$, we can obtain that, for any $k\geq1$,
 \begin{eqnarray*}
\mathcal{J}_2
\leq
\frac{3}{2k^2}\int\limits^\infty_0\frac{(f(\varrho)-f(0))^2}{\varrho e^{\frac{\varrho}{2}}}d\varrho
+\frac{1}{k^2}\int\limits^\infty_0\frac{(f(0)-f(r))^2}{r^{2}}(1-e^{-\frac{r}{2}}) dr.
\end{eqnarray*}
which along with \eqref{J-1-estimate} implies that
\begin{eqnarray}\label{J}
\begin{split}
\mathcal{J}
&\leq
6(n-1)\sum^\infty_{k=1}\frac{a_{2k}}{k}\int\limits^\infty_0\frac{(f(\varrho)-f(0))^2}{\varrho e^{\frac{\varrho}{2}}}d\varrho
+
4\sum^\infty_{k=1}\frac{a_{2k}}{k}\|f\|^2_{L^\infty}\\
&\ \ \
+4(n-1)\sum^\infty_{k=1}\frac{a_{2k}}{k}\int\limits^\infty_0\frac{(f(0)-f(r))^2}{r^{2}}(1-e^{-\frac{r}{2}}) dr.
\end{split}
\end{eqnarray}
By \eqref{Taylor} and \eqref{uniformly-bounded} again, we have that
\begin{align*}
0<\mathcal{K}
&=\sum^\infty_{k=0}a_{2k}\iint\limits_{0<r<\varrho}\frac{(f(\varrho)-f(r))^2}{\varrho^{2}e^{r}} b_k\Big(\frac{a}{\varrho}\Big)\Big(\frac{r}{\varrho}\Big)^{2k}d\varrho dr\nonumber\\
&\leq\sum^\infty_{k=0}2a_{2k}\iint\limits_{0<r<\varrho}\frac{(f(\varrho)-f(r))^2}{\varrho^{2}e^{r}}
\Big(\frac{r}{\varrho}\Big)^{2k}d\varrho dr.
\end{align*}
To continue to bound the integral $\mathcal{K}$, we split it into two parts
\begin{align*}
\iint\limits_{0<r<\varrho}\frac{(f(\varrho)-f(r))^2}{\varrho^{2}e^{r}}
\Big(\frac{r}{\varrho}\Big)^{2k}d\varrho
&=\Big(\underbrace{\iint\limits_{0<r\leq\frac{\varrho}{2}}}_{\mathcal{K}_1}
+\underbrace{\iint\limits_{\frac{\varrho}{2}<r<\varrho}}_{\mathcal{K}_2}\Big)\frac{(f(\varrho)-f(r))^2}{\varrho^{2}e^{r}} \Big(\frac{r}{\varrho}\Big)^{2k}d\varrho dr.
\end{align*}
Similarly, the part $\mathcal{K}_1$ can be estimated as
\begin{eqnarray}\label{estimate-K1}
\begin{split}
\mathcal{K}_1
&\leq\Big(\frac12\Big)^{2k}\iint\limits_{0<r<\frac{\varrho}{2}}\frac{(f(\varrho)-f(r))^2}{\varrho^{2}e^r}d\varrho dr\\
&\leq\frac{1}{2^{2k-1}}\iint\limits_{0<r<\frac{\varrho}{2}}\frac{(f(\varrho)-f(0))^2}{\varrho^{2}e^r}d\varrho dr+\frac{1}{2^{2k-1}}\iint\limits_{0<r<\frac{\varrho}{2}}\frac{(f(r)-f(0))^2}{\varrho^2e^r}d\varrho dr\\
&=\frac{1}{2^{2k-1}}\int\limits_0^\infty\frac{(f(\varrho)-f(0))^2}{\varrho^{2}}(1-e^{-\frac{\varrho}{2}})d\varrho
+\frac{1}{2^{2k}}\int\limits_0^\infty \frac{(f(r)-f(0))^2}{re^r} dr.
\end{split}
\end{eqnarray}
By \eqref{elementary-inequality} again, another part $\mathcal{K}_2$ can be bounded as
\begin{align*}
\mathcal{K}_2
&\leq2\iint\limits_{\frac{\varrho}{2}<r<\varrho}\frac{(f(\varrho)-f(0))^2}{\varrho^{2}e^{r}} \Big(\frac{r}{\varrho}\Big)^{2k}
d\varrho dr+2\iint\limits_{\frac{\varrho}{2}<r<\varrho}\frac{(f(r)-f(0))^2}{\varrho^{2}e^r} \Big(\frac{r}{\varrho}\Big)^{2k}
d\varrho dr\\
&=2\int\limits_0^\infty\frac{(f(\varrho)-f(0))^2}{\varrho^{2+2k}}\int\limits^\varrho_{\frac{\varrho}{2}}\frac{r^{2k}}{e^r}dr d\varrho+2\int\limits_0^\infty \frac{r^{2k}(f(r)-f(0))^2}{e^{r}}\int\limits_r^{2r}\frac{d \varrho}{\varrho^{2+2k}}dr\\
&\leq\frac{2}{2k+1}\int\limits_0^\infty\frac{(f(\varrho)-f(0))^2}{\varrho e^{\frac\varrho2}}d\varrho
+\frac{2}{2k+1}\int\limits_0^\infty \frac{(f(r)-f(0))^2}{re^{r}}dr\\
&\leq\frac{4}{2k+1}\int\limits_0^\infty\frac{(f(\varrho)-f(0))^2}{\varrho e^{\frac\varrho2}}d\varrho.
\end{align*}
Then we obtain that, for any $k\geq0$,
\begin{align}\label{integral-estimate}
\iint\limits_{0<r<\varrho}\frac{(f(\varrho)-f(r))^2}{\varrho^{2}e^{r}}
\Big(\frac{r}{\varrho}\Big)^{2k}d\varrho dr
\leq\frac{2\int\limits_0^\infty\frac{(f(\varrho)-f(0))^2}{\varrho^{2}}(1-e^{-\frac{\varrho}{2}})d\varrho
+5\int\limits_0^\infty\frac{(f(\varrho)-f(0))^2}{\varrho e^{\frac\varrho2}}d\varrho}{2k+1},
\end{align}
where we used $2^{2k}\geq 2k+1$ for any $k\geq0$.

Thus, we obtain that
\begin{align}\label{K}
\mathcal{K}
\leq\sum^\infty_{k=0}\frac{4a_{2k}}{2k+1}\int\limits_0^\infty\frac{(f(\varrho)-f(0))^2}{\varrho^{2}}(1-e^{-\frac{\varrho}{2}})d\rho
+\sum^\infty_{k=0}\frac{10a_{2k}}{2k+1}\int\limits_0^\infty\frac{(f(\varrho)-f(0))^2}{\varrho e^{\frac\varrho2}}d\varrho
\end{align}
Finally, by \eqref{1-mu} and \eqref{widetilde-a-2k-relation}, the term $\mathcal{L}$ can be expressed by
\begin{align*}
\mathcal{L}
&=-(n+1)\sum^\infty_{k=0}\widetilde{a}_{2k}\iint\limits_{0<r<\varrho}\frac{(f(\varrho)-f(r))^2}{\varrho^2e^r}
\Big(\frac{a}{\varrho}\Big)^2\widetilde{b}_k\Big(\frac{a}{\varrho}\Big)\Big(\frac{r}{\varrho}\Big)^{2k}d\varrho dr\\
&=-\sum^\infty_{k=0}(2k+1)(2k+n+1)a_{2k}\iint\limits_{0<r<\varrho}\frac{(f(\varrho)-f(r))^2}{\varrho^2e^r}
\Big(\frac{a}{\varrho}\Big)^2\widetilde{b}_k\Big(\frac{a}{\varrho}\Big)\Big(\frac{r}{\varrho}\Big)^{2k}d\varrho dr.
\end{align*}
By \eqref{asympotic}, we may need a uniform decay in $k$ of the quantity $\mu^2|\widetilde{b}_{k}(\mu)|$.
For this purpose,
by the definition of \eqref{bkmu-2} and the identity \eqref{identity-2k-1}, we note that,
for $k\geq1$,
\begin{align*}
\mu^2\widetilde{b}_{k}(\mu)
=\frac{\mu^2\mathcal{Q}_{n,k}(\mu^2)}{(1+\mu^2)^{\frac{n+3}{2}+2k}}
&=\frac{1}{2k+1}\frac{(1+\mu^2)\mathcal{P}_{n,k}(\mu^2)-\mathcal{P}_{n+2,k}(\mu^2)}{(1+\mu^2)^{\frac{n+3}{2}+2k}}\\
&=\frac{1}{2k+1}\Big[\frac{\mathcal{P}_{n,k}(\mu^2)}{(1+\mu^2)^{\frac{n+1}{2}+2k}}-\frac{\mathcal{P}_{n+2,k}(\mu^2)}{(1+\mu^2)^{\frac{(n+2)+1}{2}+2k}}\Big].
\end{align*}
Then we present a decay rate lemma, whose proof will be given later.
\begin{lemma}\label{Decay-rate}
Let $n\geq2$ be a fixed integer.
It holds that, for any integer $k\geq1$,
\begin{align*}
\sup_{\mu>0}\Big|\frac{\mathcal{P}_{n,k}(\mu^2)}{(1+\mu^2)^{\frac{n+1}{2}+2k}}-\frac{\mathcal{P}_{n+2,k}(\mu^2)}{(1+\mu^2)^{\frac{(n+2)+1}{2}+2k}}\Big|\leq\frac{C(n)}{\sqrt{2k+1}},
\end{align*}
where  the polynomial $\mathcal{P}_{N,K}(x)$ is defined in \eqref{polynomial-definition},
and the constant $C(n)>0$ depends only on the integer $n\geq2$.
\end{lemma}
By Lemma \ref{Decay-rate}, we obtain that, for any $k\geq1$,
\begin{align*}
\sup_{\mu>0}|\mu^2\widetilde{b}_{k}(\mu)|
\leq\frac{C(n)}{(2k+1)\sqrt{2k+1}},
\end{align*}
which along with the fact
\begin{align*}
0\leq\mu^2|\widetilde{b}_0(\mu)|=\frac{\mu^2}{(1+\mu^2)^{\frac{n+3}{2}}}\leq1,~{\rm for~any~} \mu>0,
\end{align*}
and the integral estimate \eqref{integral-estimate} yields that
\begin{align*}
|\mathcal{L}|
&\leq\sum^\infty_{k=0}(2k+1)(2k+n+1)a_{2k}\iint\limits_{0<r<\varrho}\frac{(f(\varrho)-f(r))^2}{\varrho^2e^r}
\Big(\frac{a}{\varrho}\Big)^2\Big|\widetilde{b}_k\Big(\frac{a}{\varrho}\Big)\Big|\Big(\frac{r}{\varrho}\Big)^{2k}d\varrho dr\\
&\leq C_n\sum^\infty_{k=0}\frac{2k+n+1}{\sqrt{2k+1}}a_{2k}\iint\limits_{0<r<\varrho}\frac{(f(\varrho)-f(r))^2}{\varrho^2e^r}
\Big(\frac{r}{\varrho}\Big)^{2k}d\varrho dr\\
&\leq C_n\sum^\infty_{k=0}\frac{(2k+n+1)a_{2k}}{(2k+1)\sqrt{2k+1}}\Big[2\int\limits_0^\infty\frac{(f(\varrho)-f(0))^2}{\varrho^{2}(1-e^{-\frac{\varrho}{2}})^{-1}}d\varrho
+5\int\limits_0^\infty\frac{(f(\varrho)-f(0))^2}{\varrho e^{\frac\varrho2}}d\varrho\Big]\\
&\leq (n+1)C_n\sum^\infty_{k=0}\frac{a_{2k}}{\sqrt{2k+1}}\Big[2\int\limits_0^\infty\frac{(f(\varrho)-f(0))^2}{\varrho^{2}(1-e^{-\frac{\varrho}{2}})^{-1}}d\varrho
+5\int\limits_0^\infty\frac{(f(\varrho)-f(0))^2}{\varrho e^{\frac\varrho2}}d\varrho\Big].
\end{align*}
It follows from \eqref{K} and \eqref{J} that
\begin{eqnarray*}
\begin{split}
 \mathcal{I}
&\geq
-\Big[6(n-1)\sum^\infty_{k=1}\frac{a_{2k}}{k}
+C_n\sum^\infty_{k=0}\frac{a_{2k}}{\sqrt{2k+1}}
\Big]\int\limits^\infty_0\frac{(f(\varrho)-f(0))^2}{\varrho e^{\frac{\varrho}{2}}}d\varrho-4\|f\|^2_{L^\infty}\sum^\infty_{k=1}\frac{a_{2k}}{k}\\
&\
-\Big[4(n-1)\sum^\infty_{k=1}\frac{a_{2k}}{k}
+C_n\sum^\infty_{k=0}\frac{a_{2k}}{\sqrt{2k+1}}\Big]\int\limits_0^\infty\frac{(f(\varrho)-f(0))^2}{\varrho^{2}}(1-e^{-\frac{\varrho}{2}})d\varrho.
\end{split}
\end{eqnarray*}
Taking
\begin{eqnarray*}
\varepsilon=\frac{B(\frac12,\frac{n+1}{2})}{4\Big[6(n-1)\displaystyle\sum^\infty_{k=1}\frac{a_{2k}}{k}
+C_n\displaystyle\sum^\infty_{k=0}\frac{a_{2k}}{\sqrt{2k+1}}
\Big]}
\end{eqnarray*}
in the first inequality of Lemma \ref{Young-inequality-Weight} and  another
\begin{eqnarray*}
\varepsilon=\frac{ B(\frac12,\frac{n+1}{2})}{4\Big[4(n-1)\displaystyle\sum^\infty_{k=1}\frac{a_{2k}}{k}
+C_n\displaystyle\sum^\infty_{k=0}\frac{a_{2k}}{\sqrt{2k+1}}\Big]}
\end{eqnarray*}
in the second inequality of Lemma \ref{Young-inequality-Weight}, it follows from \eqref{representation-2} that
\begin{align*}
-\int\limits_{\mathbb{R}^n}\frac{\mathcal{R}_af(x)\cdot\nabla f(x)}{|x|^{n}}e^{-|x|}dx
&\geq\frac{B(\frac12,\frac{n+1}{2})}{4\pi\omega^{-1}_{n-1}} \int\limits_0^\infty\frac{(f(\varrho)-f(0))^2}{\rho^2}\Big[1-\frac{\varrho^{n+1}(\varrho^2-na^2)}{(\varrho^2+a^2)^{\frac{n+3}{2}}} \Big] d\varrho
-C_{n,a}\|f\|^2_{L^\infty}\\
&=\frac{B(\frac12,\frac{n+1}{2})}{4\pi}
\int\limits_{\mathbb{R}^n}\frac{(f(0)-f(x))^2}{|x|^{n+1}}\Big[1-\frac{|x|^{n+1}(|x|^2-na^2)}{(|x|^2+a^2)^{\frac{n+3}{2}}}\Big]dx
-C_{n,a}\|f\|^2_{L^\infty}.
\end{align*}
We then complete the proof of Proposition \ref{nonlinear-weighted-inequality-exponential}. \hfill\hfill$\square$\vskip12pt

\begin{remark}
It is readily seen that the constant $C_{n,a}$ in the Proposition \ref{nonlinear-weighted-inequality-exponential} can be computed explicitly as
\begin{align*}
C_{n,a}
&=\frac{\omega_{n-1}}{2\pi}
\Big[\frac{4\Big[6(n-1)\displaystyle\sum^\infty_{k=1}\frac{a_{2k}}{k}
+C_n\sum^\infty_{k=0}\frac{a_{2k}}{\sqrt{2k+1}}
\Big]^2+\Big[4(n-1)\sum^\infty_{k=1}\frac{a_{2k}}{k}
+C_n\sum^\infty_{k=0}\frac{a_{2k}}{\sqrt{2k+1}}\Big]^2}{B(\frac12,\frac{n+1}{2})}\\
&\ \ \ \ \ \ \ \ \ \ \ \ \
+\frac{64(n-1)\displaystyle\sum^\infty_{k=1}\frac{a_{2k}}{k}
+12C_n\sum^\infty_{k=0}\frac{a_{2k}}{\sqrt{2k+1}}}{\min\Big\{1,a\sqrt{\frac{n}{3}}\Big\}}+4\sum^\infty_{k=1}\frac{a_{2k}}{k}
\Big],
\end{align*}
which depends only on the dimension $n$ when $a>\sqrt{\frac{3}{n}}$. Consequently, when $a\rightarrow\infty$, the integral inequality for the transform $\mathcal{R}_a$ in Proposition \ref{nonlinear-weighted-inequality-exponential}  formally reduces to the integral inequality \eqref{Riesz-transform-integral-inequality} for the Riesz transform $\mathcal{R}$.
\end{remark}
Now we give the proof of Lemma \ref{Decay-rate}.

\textbf{Proof of Lemma \ref{Decay-rate}.}
Note that
\begin{align*}
\displaystyle\prod^{j-1}_{l=0}\frac{N+2(k-l)}{1+2l}=(-1)^j\frac{(-\frac{N}{2}-k)_j}{(\frac12)_j}~{\rm and}~\binom{k}{j}
=(-1)^j\frac{(-k)_j}{j!}.
\end{align*}
Hence
\begin{align*}
\mathcal{P}_{N,k}(\mu^2)
=1+\displaystyle\sum^{k}_{j=1}\frac{(-k)_j(-\frac{N}{2}-k)_j}{(\frac12)_j}\frac{(-\mu^2)^j}{j!}
={}_2F_1\Big(-k,-\frac{N}{2}-k;\frac12;-\mu^2\Big).
\end{align*}
Here we used the fact that $(-k)_j=0$ for any $j\geq k+1$.
By the Pfaff's transformation \eqref{Pfaffs-identity} and \eqref{hypergeometric-Jacobi}, we can obtain that
\begin{align*}
\frac{\mathcal{P}_{N,k}(\mu^2)}{(1+\mu^2)^{\frac{N+1}{2}+2k}}
&=\frac{(1+\mu^2)^k{}_2F_1\Big(-k,\frac{N+1}{2}+k;\frac12;\frac{\mu^2}{1+\mu^2}\Big)}{(1+\mu^2)^{\frac{N+1}{2}+2k}}\\
&=\frac{(2k)!!}{(2k-1)!!}(1+\mu^2)^{-k-\frac{N}{2}-\frac{1}{2}}P^{(-\frac12,\frac{N}{2})}_k\Big(\frac{1-\mu^2}{1+\mu^2}\Big)\\
&=\frac{(2k)!!}{(2k-1)!!}q^{k+\frac{N}{2}+\frac12}P^{(-\frac12,\frac{N}{2})}_k(2q-1),
\end{align*}
where we set $q:=\frac{1}{1+\mu^2}\in(0,1)$. It follows from \eqref{Jacobi-identity} that
\begin{align}\label{Decay-representation}
\mathcal{A}_{n,k}(\mu)
&:=\frac{\mathcal{P}_{n,k}(\mu^2)}{(1+\mu^2)^{\frac{n+1}{2}+2k}}-\frac{\mathcal{P}_{n+2,k}(\mu^2)}{(1+\mu^2)^{\frac{(n+2)+1}{2}+2k}}\nonumber\\
&=\frac{(2k)!!}{(2k-1)!!}q^{k+\frac{n}{2}+\frac12}\Big[P^{(-\frac12,\frac{n}{2})}_k(2q-1)-qP^{(-\frac12,\frac{n}{2}+1)}_k(2q-1)\Big]\nonumber\\
&=\frac{(2k)!!}{(2k-1)!!}(1-q)q^{k+\frac{n}{2}+\frac12}P^{(\frac12,\frac{n}{2})}_k(2q-1).
\end{align}
The weighted Bernstein estimates for Jacobi polynomials in Lemma \ref{Bernstein-Lemma} gives that
\begin{align*}
\Big|P^{(\frac12,\frac{n}{2})}_k(2q-1)\Big|\leq \frac{C(n)}{\sqrt{k}}(1-q)^{-\frac12}q^{-\frac{n}{4}-\frac14}.
\end{align*}
Note that, for any $k\geq1$,
\begin{align*}
(2k)!!
=2\sqrt{k}\displaystyle\prod^{k-1}_{l=1}\sqrt{2l(2l+2)}
\leq2\sqrt{k}\displaystyle\prod^{k-1}_{l=1}(2l+1)
=2\sqrt{k}(2k-1)!!,
\end{align*}
we deduce that, for any $\mu>0$,
\begin{align*}
|\mathcal{A}_{n,k}(\mu)|
&\leq2\sqrt{k}(1-q)q^{k+\frac{n}{2}+\frac12}\cdot \frac{C(n)}{\sqrt{k}}(1-q)^{-\frac12}q^{-\frac{n}{4}-\frac14}\\
&=C(n)q^{k+\frac{n}{4}+\frac14}\sqrt{1-q}\leq C(n)q^{k}\sqrt{1-q}\\
&\leq C(n)\Big(\frac{2k}{2k+1}\Big)^k\frac{1}{\sqrt{2k+1}}
\leq\frac{C}{\sqrt{2k+1}},
\end{align*}
where we have used the fact
\begin{align*}
\displaystyle\max_{0<q<1}q^k\sqrt{1-q}=\Big(\frac{2k}{2k+1}\Big)^k\frac{1}{\sqrt{2k+1}}.
\end{align*}
The proof of Lemma \ref{Decay-rate} is then completed.
\hfill\hfill$\square$\vskip12pt
\begin{remark}
It should be remarked that the decay rate in Lemma \ref{Decay-rate} is optimal. The lower-bound construction is as follows. For any $k\geq1$, we define
\begin{align*}
\nu_k:=k+\frac{n}{4}+\frac34,~~m_k:=\Big\lfloor\frac{\sqrt{k}}{2\pi}\Big\rfloor,~~\theta_k:=\frac{2\pi m_k+\frac{\pi}{2}}{\nu_k}
\end{align*}
and
\begin{align*}
q_k:=\cos^2\frac{\theta_k}{2},~~~\mu_k:=\tan\frac{\theta_k}{2}.
\end{align*}
Then
\begin{align*}
q_k=\frac{1}{1+\mu^2_k},~~~~2q_k-1=\cos\theta_k.
\end{align*}
For fixed $n\geq2$, the Darboux asymptotic formula for  Jacobi polynomials (see Theorem 8.21.13 in \cite{[Szego]}) yields that
\begin{align*}
P^{(\frac12,\frac{n}{2})}_k(\cos\theta)
=\frac{1}{\sqrt{\pi k}\sin\frac{\theta}{2}\cos^{\frac{n}{2}+\frac12}\frac{\theta}{2}}\Big[\cos\Big(\nu_k\theta-\frac{\pi}{2}\Big)+O\Big(\frac{1}{k\sin\theta}\Big)\Big].
\end{align*}
By the definition of $\theta_k$, we have that
\begin{align*}
\cos\Big(\nu_k\theta_k-\frac{\pi}{2}\Big)=\cos(2\pi m_k)\equiv1,~~\lim_{k\rightarrow\infty}\sqrt{k}\theta_k=1,
\end{align*}
which implies that
\begin{align*}
\lim_{k\rightarrow\infty}P^{(\frac12,\frac{n}{2})}_k(\cos\theta_k)
=\frac{1}{\sqrt{\pi }}\lim_{k\rightarrow\infty}\frac{1+O\Big(\frac{1}{k\sin\theta_k}\Big)}{\sqrt{ k}\sin\frac{\theta_k}{2}\cos^{\frac{n}{2}+\frac12}\frac{\theta_k}{2}}=\frac{2}{\sqrt{\pi}}.
\end{align*}
Furthermore,
\begin{align*}
\lim_{k\rightarrow\infty}k(1-q_k)=\lim_{k\rightarrow\infty}k\sin^2\frac{\theta_k}{2}=\frac{1}{4}
\end{align*}
and
\begin{align*}
\lim_{k\rightarrow\infty}q_k^{k+\frac{n}{2}+\frac12}
=e^{\lim\limits_{k\rightarrow\infty}(2k+n+1)\ln\cos\frac{\theta_k}{2}}
=e^{-2\lim\limits_{k\rightarrow\infty}(2k+n+1)\sin^2\frac{\theta_k}{4}}
=e^{-\frac14}.
\end{align*}
Inserting these limits into \eqref{Decay-representation} implies that
\begin{align*}
\lim_{k\rightarrow\infty}\sqrt{k}\mathcal{A}_{n,k}(\mu_k)
&=\lim_{k\rightarrow\infty}\frac{1}{\sqrt{k}}\frac{(2k)!!}{(2k-1)!!}\lim_{k\rightarrow\infty}k(1-q_k)\lim_{k\rightarrow\infty}q_k^{k+\frac{n}{2}+\frac12}\lim_{k\rightarrow\infty}P^{(\frac12,\frac{n}{2})}_k(\cos\theta_k)\\
&=\sqrt{\pi}\cdot\frac{1}{4}\cdot e^{-\frac14}\cdot\frac{2}{\sqrt{\pi}}=\frac12e^{-\frac14},
\end{align*}
which shows that the decay rate in Lemma \ref{Decay-rate} is indeed optimal.
\end{remark}
\section{Finite-time blow-up for smooth radial data}
\hskip\parindent In this section, we consider the finite time blow-up of the solution to \eqref{M-KS} for smooth radial data. Proposition \ref{n-nonlinear-inequality-monotone} can be used to show that the family of smooth, compactly supported, radial and non-decreasing initial data with negative value at the origin to \eqref{M-KS} undergo finite time blow-up. The first blow-up result can be stated as
\begin{theorem}\label{singularity-formation-monotone}
Let $n\geq2$ and $a,g>0$.
Suppose $\rho_0\in C^\infty_c(\mathbb{R}^n)$ is radial and non-decreasing with $\rho_0(0)<0$.
Then the solution $\rho$ to \eqref{M-KS} with the initial data $\rho_0$ develops the gradient blow-up in finite time.
\end{theorem}
\textbf{Proof.} The proof of this theorem requires only minor modifications to the one of Theorem 1.2 in \cite{[Zhang-JDE2026]}, and we will provide a brief outline for completeness.

Suppose that the initial data $\rho_0\in C^\infty_c(\mathbb{R}^n)$ is radial and non-decreasing with $\rho_0(0)<0$. Moreover, suppose that $\rho_0$ is supported in a ball $B_L(0)$ for some $L>0$.
Assume for contradiction that the corresponding unique local-in-time solution $\rho$ to \eqref{M-KS}  exists for all time.
Then it is not difficult to see that $\rho$ is radial non-decreasing on $\mathbb{R}^n$ for all time $t>0$, and that for any $t>0$ the support of $\rho(\cdot,t)$ lies inside $B_L(0)$.

We proceed to denote
\begin{eqnarray*}
J(t):=\int\limits_{B_L(0)}\frac{\rho(x,t)-\rho(0,t)}{|x|^{n+\delta}}dx,
\end{eqnarray*}
where $\delta\in(0,1)$ is arbitrarily fixed.
By H\"{o}lder's inequality, we have that,  for all $t>0$,
\begin{align*}
|J(t)|
\leq \|\nabla\rho\|_{L^\infty}\int\limits_{|x|\leq L}\frac{dx}{|x|^{n-1+\delta}}
=\frac{\omega_{n-1}L^{1-\delta}}{1-\delta}\|\nabla\rho\|_{L^\infty}<+\infty.
\end{align*}
Next we prove that $J(t)$ will blow up in finite time.
To this end, by Lemma \ref{radial-property-preserved}, we know that the velocity at the origin is $0$,
and then
\begin{eqnarray}\label{evolution-at-origin}
\rho(0,t)=\rho_0(0),{\rm  for~ any}~ t>0.
\end{eqnarray}
Then, by  \eqref{M-KS} and  Proposition \ref{n-nonlinear-inequality-monotone}, we compute
\begin{align*}
\frac{d}{dt}J(t)
&=-g\int\limits_{B_L(0)}\frac{\mathcal{R}_a \rho(x,t)\cdot\nabla \rho(x,t)}{|x|^{n+\delta}}dx
=-g\int\limits_{\mathbb{R}^n}\frac{\mathcal{R}_a \rho(x,t)\cdot\nabla \rho(x,t)}{|x|^{n+\delta}}dx\\
&\geq gC_{n,\delta}
\int\limits_{\mathbb{R}^n}\frac{(\rho(x,t)-\rho(0,t))^2}{|x|^{n+1+\delta}}\Big[1-\frac{|x|^{n+1}}{(|x|^2+a^2)^{\frac{n+1}{2}}}\Big]dx\\
&\geq gC_{n,\delta}
\int\limits_{B_L(0)}\frac{(\rho(x,t)-\rho(0,t))^2}{|x|^{n+1+\delta}}\Big[1-\frac{|x|^{n+1}}{(|x|^2+a^2)^{\frac{n+1}{2}}}\Big]dx\\
&\geq gC_{n,\delta}
\Big[1-\frac{L^{n+1}}{(L^2+a^2)^{\frac{n+1}{2}}}\Big]\int\limits_{B_L(0)}\frac{(\rho(x,t)-\rho(0,t))^2}{|x|^{n+1+\delta}}dx\\
&\geq gC_{n,\delta}
\Big[1-\frac{L^{n+1}}{(L^2+a^2)^{\frac{n+1}{2}}}\Big]\Big(\int\limits_{B_L(0)}\frac{dx}{|x|^{n-1+\delta}}\Big)^{-1}(J(t))^2\\
&=\frac{g(1-\delta)C_{n,\delta}}{\omega_{n-1}L^{1-\delta}}
\Big[1-\frac{L^{n+1}}{(L^2+a^2)^{\frac{n+1}{2}}}\Big](J(t))^2.
\end{align*}
Since
\begin{eqnarray*}
J(0)=\int\limits_{B_L(0)}\frac{\rho_0(x)-\rho_0(0)}{|x|^{n+\delta}}dx>0,
\end{eqnarray*}
$J(t)$ must blow up in finite time, leading to a contradiction. The proof of Theorem \ref{singularity-formation-monotone} is finished.
\hfill\hfill$\square$\vskip12pt

Proposition \ref{nonlinear-weighted-inequality-exponential} can be used to establish the blow-up of the solutions to \eqref{M-KS} for the radial initial data which is not necessarily radially monotone. The main result of this aspect is stated as
\begin{theorem}\label{the-1}
Let the initial data $\rho_0$ be a radial Schwartz function.
There exists a constant $A(n,a)>0$ depending only on $n$ and $a$ such that if
\begin{eqnarray}\label{Class-initial-data}
\int\limits_{\mathbb{R}^n}\frac{\rho_0(x)-\rho_0(0)}{|x|^n}e^{-|x|}dx\geq A(n,a)\|\rho_0\|_{L^\infty},
\end{eqnarray}
then the smooth solution $\rho$ to \eqref{M-KS} blows up in finite time.
\end{theorem}
\textbf{Proof.} We will argue by contradiction.
Assume that the solution $\rho$ to \eqref{M-KS} starting from the initial data $\rho_0\in \mathcal{S}(\mathbb{R}^n)$, the Schwartz class, satisfying \eqref{Class-initial-data}, exists for all time. For our purpose, we introduce a quantity $\widetilde{J}(t)$ given by
\begin{eqnarray}\label{weighted-integral-solution}
\widetilde{J}(t):=\int\limits_{\mathbb{R}^n}\frac{\rho(x,t)-\rho(0,t)}{|x|^{n}}e^{-|x|}dx.
\end{eqnarray}
By H\"{o}lder's inequality and the maximum principle  $\|\rho(t)\|_{L^\infty}\leq\|\rho_0\|_{L^\infty}$, we have that
\begin{eqnarray*}
\begin{split}
|\widetilde{J}(t)|
&\leq \|\nabla\rho(t)\|_{L^\infty}\int\limits_{|x|\leq1}\frac{dx}{|x|^{n-1}}
+2\|\rho_0\|_{L^\infty}\int\limits_{|x|>1}\frac{e^{-|x|}}{|x|^{n}}dx\\
&=\omega_{n-1}\|\nabla\rho\|_{L^\infty}+2\omega_{n-1}\|\rho_0\|_{L^\infty}\int\limits^\infty_1\frac{dr}{re^r}<+\infty,
\end{split}
\end{eqnarray*}
which shows that $\widetilde{J}(t)$ is finite for any $t>0$. Next we prove that $\widetilde{J}(t)$ will blow up at some finite time $T_0>0$ and then obtain a contradiction.

By \eqref{evolution-at-origin}, \eqref{M-KS}, Proposition \ref{nonlinear-weighted-inequality-exponential} and the maximum principle  $\|\rho(t)\|_{L^\infty}\leq\|\rho_0\|_{L^\infty}$ again, we can obtain that
\begin{eqnarray*}
\begin{split}
&\frac{d}{dt}\widetilde{J}(t)
=-g\int\limits_{\mathbb{R}^n} \frac{\mathcal{R}_a\rho(x,t)\cdot\nabla\rho(x,t)}{|x|^{n}}e^{-|x|}dx\\
&\geq g\frac{B(\frac12,\frac{n+1}{2})}{4\pi}\int\limits_{\mathbb{R}^n}\frac{(\rho(0,t)-\rho(x,t))^2}{|x|^{n+1}}\Big[1-\frac{|x|^{n+1}(|x|^2-na^2)}{(|x|^2+a^2)^{\frac{n+3}{2}}}\Big]dx
-gC_{n,a}\|\rho_0\|^2_{L^\infty}.
\end{split}
\end{eqnarray*}
By \eqref{weighted-integral-solution} and Cauchy-Schwarz inequality, we note that
\begin{eqnarray*}
\begin{split}
(\widetilde{J}(t))^2
&\leq\int\limits_{\mathbb{R}^n}\frac{(\rho(0,t)-\rho(x,t))^2}{|x|^{n+1}}\Big[1-\frac{|x|^{n+1}(|x|^2-na^2)}{(|x|^2+a^2)^{\frac{n+3}{2}}}\Big]dx
\int\limits_{\mathbb{R}^n}\frac{e^{-2|x|}}{|x|^{n-1}\Big(1-\frac{|x|^{n+1}(|x|^2-na^2)}{(|x|^2+a^2)^{\frac{n+3}{2}}}\Big)}dx\\
&=\omega_{n-1}c_{n,a}\int\limits_{\mathbb{R}^n}\frac{(\rho(0,t)-\rho(x,t))^2}{|x|^{n+1}}\Big[1-\frac{|x|^{n+1}(|x|^2-na^2)}{(|x|^2+a^2)^{\frac{n+3}{2}}}\Big]dx,
\end{split}
\end{eqnarray*}
where, by \eqref{small-order},
\begin{align*}
c_{n,a}
&:=\int\limits_0^\infty\frac{e^{-2r}}{1-\frac{r^{n+1}(r^2-na^2)}{(r^2+a^2)^{\frac{n+3}{2}}}}dr<\infty,
\end{align*}
This implies that
\begin{eqnarray*}
\int\limits_{\mathbb{R}^n}\frac{(\rho(0,t)-\rho(x,t))^2}{|x|^{n+1}}\Big[1-\frac{|x|^{n+1}(|x|^2-na^2)}{(|x|^2+a^2)^{\frac{n+3}{2}}}\Big]dx
\geq
\frac{(\widetilde{J}(t))^2}{\omega_{n-1}c_{n,a}}.
\end{eqnarray*}
Therefore, we derive that
\begin{eqnarray}\label{ordinary-differential-inequality}
\begin{split}
\frac{d}{dt}\widetilde{J}(t)
\geq c_1(\widetilde{J}(t))^2-c_2,~{\rm with}~c_1:=\frac{gB(\frac12,\frac{n+1}{2})}{4\pi\omega_{n-1}c_{n,a}},~c_2:=gC_{n,a}\|\rho_0\|^2_{L^\infty}.
\end{split}
\end{eqnarray}
To continue, we denote by $I(t)$ the solution of the following ordinary differential equation
\begin{align*}
\begin{cases}
\frac{d}{dt}I(t)=c_1(I(t))^2-c_2,\\
I(0)=\widetilde{J}(0),
\end{cases}
\end{align*}
which, was explicitly computed in \cite{[Jarrin-Vegara-Hermosilla]} as
\begin{align*}
I(t)=\sqrt{\frac{c_2}{c_1}}\frac{\Big(\widetilde{J}(0)+\sqrt{\frac{c_2}{c_1}}\Big)+\Big(\widetilde{J}(0)-\sqrt{\frac{c_2}{c_1}}\Big)e^{2\sqrt{c_1c_2}t}}{\Big(\widetilde{J}(0)+\sqrt{\frac{c_2}{c_1}}\Big)-\Big(\widetilde{J}(0)-\sqrt{\frac{c_2}{c_1}}\Big)e^{2\sqrt{c_1c_2}t}}.
\end{align*}
Note that $I(t)$ blows up at the finite time
\begin{eqnarray*}
T_{0}:=\frac{1}{2\sqrt{c_1c_2}}\log\frac{\sqrt{c_1}\widetilde{J}(0)+\sqrt{c_2}}{\sqrt{c_1}\widetilde{J}(0)-\sqrt{c_2}}>0,
\end{eqnarray*}
provided
\begin{eqnarray*}
\widetilde{J}(0)
=\int\limits_{\mathbb{R}^n}\frac{\rho_0(x)-\rho_0(0)}{|x|^n}e^{-|x|}dx
>\sqrt{\frac{c_2}{c_1}}=2\sqrt{\frac{\pi\omega_{n-1}c_{n,a}C_{n,a}}{B(\frac12,\frac{n+1}{2})}}\|\rho_0\|_{L^\infty}.
\end{eqnarray*}
Finally, by choosing $A(n,a)=2\sqrt{\frac{\pi\omega_{n-1}c_{n,a}C_{n,a}}{B(\frac12,\frac{n+1}{2})}}$, \eqref{Class-initial-data}, \eqref{ordinary-differential-inequality} and the comparison principle of ordinary differential equation, it follows that $J(t)\geq I(t)$, and the functional $J(t)$ also blows up at the finite time $T_{0}>0$, leading to the desired contradiction.
The proof of Theorem \ref{the-1} is then completed. \hfill\hfill$\square$\vskip12pt
\section{Appendix}
\label{appendix}
In this appendix, we will give the proofs of some lemmas and facts previously used, including Lemma \ref{Taylor-series-coefficients}, \eqref{singularity-order} and Lemma \ref{Bernstein-Lemma}.
\subsection{Proof of Lemma \ref{Taylor-series-coefficients}}
In the spirit of the proof of Lemma 3.2 in \cite{[Dong]}, we will give the proof by the induction arguments on the integer $k\geq1$. \vskip12pt
\textbf{Proof.}
Direct differentiation yields that
\begin{eqnarray*}
\frac{\partial g_\mu}{\partial \lambda}(\lambda,m,\beta)
=-(m+2\beta)\int\limits_0^\pi\frac{\sin^m\theta(\lambda-\cos\theta)}{(1-2\lambda\cos\theta+\lambda^2+\mu^2)^{\frac{m}{2}+\beta+1}}d\theta
\end{eqnarray*}
and
\begin{align*}
&\frac{1}{m+2\beta}\frac{\partial^2 g_\mu}{\partial \lambda^2}(\lambda,m,\beta)
=-\int\limits_0^\pi\frac{\sin^m\theta }{(1-2\lambda\cos\theta+\lambda^2+\mu^2)^{\frac{m}{2}+\beta+1}}d\theta\\
&\ \ \ \ \ \ \ \ \ \ \ \ \ \ \ \ \ \ \ \ \ \ \ \ \ \ \ \ \ \ \
+(m+2\beta+2)\int\limits_0^\pi\frac{\sin^m\theta(\lambda-\cos\theta)^2 }{(1-2\lambda\cos\theta+\lambda^2+\mu^2)^{\frac{m}{2}+\beta+2}}d\theta,
\end{align*}
which along with the identity
\begin{eqnarray*}
(\lambda-\cos\theta)^2 =(1-2\lambda\cos\theta+\lambda^2+\mu^2)-\sin^2\theta-\mu^2,
\end{eqnarray*}
yields that
\begin{eqnarray}\label{second-order}
\begin{split}
&\frac{1}{m+2\beta}\frac{\partial^2 g_\mu}{\partial \lambda^2}(\lambda,m,\beta)
=(m+2\beta+1)g_\mu(\lambda,m,\beta+1)\\
& \ \ \
 -(m+2\beta+2)g_\mu(\lambda,m+2,\beta+1)-(m+2\beta+2)\mu^2g_\mu(\lambda,m,\beta+2).
\end{split}
\end{eqnarray}
Note that
\begin{eqnarray*}
g_\mu(0,m,\beta)
=\frac{B\Big(\frac12,\frac{m+1}{2}\Big)}{(1+\mu^2)^{\frac{m}{2}+\beta}}>0.
\end{eqnarray*}
It follows from \eqref{second-order} and the identity $B\Big(\frac12,\frac{m+3}{2}\Big)=\frac{m+1}{m+2}B\Big(\frac12,\frac{m+1}{2}\Big)$ that
\begin{align*}
\frac{\partial^{2} g_\mu}{\partial \lambda^{2}}(0,m,\beta)
&=(m+2\beta)(m+2\beta+1)\frac{B(\frac12,\frac{m+1}{2})}{(1+\mu^2)^{\frac{m}{2}+\beta+1}}
-(m+2\beta)(m+2\beta+2)\frac{B(\frac12,\frac{m+3}{2})}{(1+\mu^2)^{\frac{m}{2}+\beta+2}}\\
&\ \ \ \ \
-(m+2\beta)(m+2\beta+2)\mu^2\frac{B(\frac12,\frac{m+1}{2})}{(1+\mu^2)^{\frac{m}{2}+\beta+2}}\\
&=\frac{B(\frac12,\frac{m+1}{2})}{(1+\mu^2)^{\frac{m}{2}+\beta+2}}\frac{m+2\beta}{m+2}\times
\Big[(m+2\beta+1)(m+2)(1+\mu^2)\\
&\ \ \ \ \ \ \ \ \ \ \ \ \ \ \ \ \ \ \ \ \ \ \ \ \ \ \ \ \ \ \ \ \ \ \ \ \ \ \ \ \ \ \ \ \
-(m+2\beta+2)(m+1)-(m+2\beta+2)(m+2)\mu^2\Big]\\
&=B\Big(\frac12,\frac{m+1}{2}\Big)2\beta\frac{m+2\beta}{m+2}\cdot\frac{1-\frac{m+2}{2\beta}\mu^2}{(1+\mu^2)^{\frac{m}{2}+\beta+2}},
\end{align*}
which shows that \eqref{claim} holds for $k=1$. Now we assume that \eqref{claim} holds for some integer $k\geq1$, for any $m\geq2$ and $\beta>0$.
By \eqref{second-order}, we have that
\begin{eqnarray}\label{recursive-formula}
\begin{split}
&\frac{1}{m+2\beta}\frac{\partial^{2k+2}g_\mu }{\partial \lambda^{2k+2}}(0,m,\beta)=(m+2\beta+1)\frac{\partial^{2k}g_\mu }{\partial \lambda^{2k}}(0,m,\beta+1)\\
&\ \ \
 -(m+2\beta+2)\frac{\partial^{2k}g_\mu }{\partial \lambda^{2k}}(0,m+2,\beta+1)-(m+2\beta+2)\mu^2\frac{\partial^{2k}g_\mu}{\partial \lambda^{2k}}(0,m,\beta+2).
\end{split}
\end{eqnarray}
Then, by \eqref{recursive-formula} and \eqref{claim}, we derive that
\begin{align*}
&\frac{\partial^{2(k+1)}g_\mu }{\partial \lambda^{2(k+1)}}(0,m,\beta)\\
&=(m+2\beta)(m+2\beta+1)B\Big(\frac12,\frac{m+1}{2}\Big)(2k-1)!!
\displaystyle\prod^{k-1}_{j=0}\frac{(2\beta+2+2j)(m+2\beta+2+2j)}{m+2j+2}\\
&\ \ \ \ \ \ \ \ \ \ \ \  \ \ \ \ \ \ \ \ \ \
\times
\frac{1+\displaystyle\sum^{k}_{j=1}(-1)^j\binom{k}{j}\Big(\displaystyle\prod^{j-1}_{l=0}\frac{m+2(k-l)}{2\beta+2+2l}\Big)\mu^{2j}}{(1+\mu^2)^{\frac{m}{2}+\beta+2k+1}}\\
&\ \ \
 -(m+2\beta)(m+2\beta+2)\frac{m+1}{m+2}B\Big(\frac12,\frac{m+1}{2}\Big)(2k-1)!!\displaystyle\prod^{k-1}_{j=0}\frac{(2\beta+2+2j)(m+2\beta+4+2j)}{m+2j+4}\\
&\ \ \ \ \ \ \ \ \ \ \ \  \ \ \ \ \ \ \ \ \ \
\times
\frac{1+\displaystyle\sum^{k}_{j=1}(-1)^j\binom{k}{j}\Big(\displaystyle\prod^{j-1}_{l=0}\frac{m+2+2(k-l)}{2\beta+2+2l}\Big)\mu^{2j}}{(1+\mu^2)^{\frac{m}{2}+\beta+2k+2}}\\
&\ \ \
 -\mu^2(m+2\beta)(m+2\beta+2)B\Big(\frac12,\frac{m+1}{2}\Big)(2k-1)!!\displaystyle\prod^{k-1}_{j=0}\frac{(2\beta+4+2j)(m+2\beta+4+2j)}{m+2j+2}\\
&\ \ \ \ \ \ \ \ \ \ \ \  \ \ \ \ \ \ \ \ \ \
\times
\frac{1+\displaystyle\sum^{k}_{j=1}(-1)^j\binom{k}{j}\Big(\displaystyle\prod^{j-1}_{l=0}\frac{m+2(k-l)}{2\beta+4+2l}\Big)\mu^{2j}}{(1+\mu^2)^{\frac{m}{2}+\beta+2k+2}}\\
&=B\Big(\frac12,\frac{m+1}{2}\Big)(2k+1)!!\Big(\displaystyle\prod^{k}_{j=0}\frac{(2\beta+2j)(m+2\beta+2j)}{m+2j+2}\Big)(1+\mu^2)^{-(\frac{m}{2}+\beta+2k+2)}\\
&\
\times\Big\{
\frac{(m+2\beta+1)(m+2k+2)}{2\beta(2k+1)}
\cdot(1+\mu^2)\Big[1+\sum^{k}_{j=1}(-1)^j\binom{k}{j}\Big(\displaystyle\prod^{j-1}_{l=0}\frac{m+2(k-l)}{2\beta+2+2l}\Big)\mu^{2j}\Big]\\
&\ \ \
 -\frac{(m+1)(m+2\beta+2k+2)}{2\beta(2k+1)}
\Big[1+\sum^{k}_{j=1}(-1)^j\binom{k}{j}\Big(\displaystyle\prod^{j-1}_{l=0}\frac{m+2+2(k-l)}{2\beta+2+2l}\Big)\mu^{2j}\Big]\\
&\
 -\frac{(m+2k+2)(2\beta+2k+2)(m+2\beta+2k+2)}{2\beta(2\beta+2)(2k+1)}
\mu^2\Big[1+\sum^{k}_{j=1}(-1)^j\binom{k}{j}\Big(\displaystyle\prod^{j-1}_{l=0}\frac{m+2(k-l)}{2\beta+4+2l}\Big)\mu^{2j}\Big]
\Big\}\\
&=B\Big(\frac12,\frac{m+1}{2}\Big)(2k+1)!!\Big(\displaystyle\prod^{k}_{j=0}\frac{(2\beta+2j)(m+2\beta+2j)}{m+2j+2}\Big)(1+\mu^2)^{-(\frac{m}{2}+\beta+2k+2)}\\
&\
\times\Big\{
\frac{(m+2\beta+1)(m+2k+2)}{2\beta(2k+1)}+\frac{(m+2\beta+1)(m+2k+2)}{2\beta(2k+1)}\mu^2\\
&\ \ \ \ \ \ \ \ \ \ \
+\frac{(m+2\beta+1)(m+2k+2)}{2\beta(2k+1)}\sum^{k}_{j=1}(-1)^j\binom{k}{j}\Big(\displaystyle\prod^{j-1}_{l=0}\frac{m+2(k-l)}{2\beta+2+2l}\Big)\mu^{2j}\\
&\ \ \ \ \ \ \ \ \ \ \
+\frac{(m+2\beta+1)(m+2k+2)}{2\beta(2k+1)}
\sum^{k}_{j=1}(-1)^j\binom{k}{j}\Big(\displaystyle\prod^{j-1}_{l=0}\frac{m+2(k-l)}{2\beta+2+2l}\Big)\mu^{2j+2}\\
&\
-\frac{(m+1)(m+2\beta+2k+2)}{2\beta(2k+1)} -\frac{(m+2k+2)(2\beta+2k+2)(m+2\beta+2k+2)}{2\beta(2\beta+2)(2k+1)}\mu^2\\
&\
-\frac{(m+1)(m+2\beta+2k+2)}{2\beta(2k+1)}
\sum^{k}_{j=1}(-1)^j\binom{k}{j}\Big(\displaystyle\prod^{j-1}_{l=0}\frac{m+2+2(k-l)}{2\beta+2+2l}\Big)\mu^{2j}\\
&\
-\frac{(m+2k+2)(2\beta+2k+2)(m+2\beta+2k+2)}{2\beta(2\beta+2)(2k+1)}
\sum^{k}_{j=1}(-1)^j\binom{k}{j}\Big(\displaystyle\prod^{j-1}_{l=0}\frac{m+2(k-l)}{2\beta+4+2l}\Big)\mu^{2j+2}
\Big\}\\
&=B\Big(\frac12,\frac{m+1}{2}\Big)(2k+1)!!\Big(\displaystyle\prod^{k}_{j=0}\frac{(2\beta+2j)(m+2\beta+2j)}{m+2j+2}\Big)(1+\mu^2)^{-(\frac{m}{2}+\beta+2k+2)}\\
&\ \ \
\times\Big\{\frac{(m+2\beta+1)(m+2k+2)}{2\beta(2k+1)}+\frac{(m+2\beta+1)(m+2k+2)}{2\beta(2k+1)}\mu^2\\
&\ \ \ \ \ \ \ \ \ \ \
+\frac{(m+2\beta+1)(m+2k+2)}{2\beta(2k+1)}\sum^{k}_{j=1}(-1)^j\binom{k}{j}\Big(\displaystyle\prod^{j-1}_{l=0}\frac{m+2(k-l)}{2\beta+2+2l}\Big)\mu^{2j}\\
&\ \ \ \ \ \ \ \ \ \ \
+\frac{(m+2\beta+1)(m+2k+2)}{2\beta(2k+1)}
\sum^{k+1}_{j=2}(-1)^{j-1}\binom{k}{j-1}\Big(\displaystyle\prod^{j-2}_{l=0}\frac{m+2(k-l)}{2\beta+2+2l}\Big)\mu^{2j}\\
&\ \ \ \ \ \ \ \
-\frac{(m+1)(m+2\beta+2k+2)}{2\beta(2k+1)}-\frac{(m+2k+2)(2\beta+2k+2)(m+2\beta+2k+2)}{2\beta(2\beta+2)(2k+1)}\mu^2\\
&\ \ \ \ \ \ \ \
-\frac{(m+1)(m+2\beta+2k+2)}{2\beta(2k+1)}
\sum^{k}_{j=1}(-1)^j\binom{k}{j}\Big(\displaystyle\prod^{j-1}_{l=0}\frac{m+2+2(k-l)}{2\beta+2+2l}\Big)\mu^{2j}\\
&\
-\frac{(m+2k+2)(2\beta+2k+2)(m+2\beta+2k+2)}{2\beta(2\beta+2)(2k+1)}
\sum^{k+1}_{j=2}(-1)^{j-1}\binom{k}{j-1}\Big(\displaystyle\prod^{j-2}_{l=0}\frac{m+2(k-l)}{2\beta+4+2l}\Big)\mu^{2j}\Big\}\\
&=B\Big(\frac12,\frac{m+1}{2}\Big)(2k+1)!!\Big(\displaystyle\prod^{k}_{j=0}\frac{(2\beta+2j)(m+2\beta+2j)}{m+2j+2}\Big)\frac{\displaystyle\sum^{k+1}_{j=0}\mathcal{I}_j\mu^{2j}}{(1+\mu^2)^{\frac{m}{2}+\beta+2k+2}},
\end{align*}
where
\begin{align*}
 \mathcal{I}_0:=\frac{(m+2\beta+1)(m+2k+2)}{2\beta(2k+1)}-\frac{(m+1)(m+2\beta+2k+2)}{2\beta(2k+1)},
\end{align*}
\begin{align*}
 \mathcal{I}_1&:=\frac{(m+2\beta+1)(m+2k+2)}{2\beta(2k+1)}-\frac{(m+2\beta+1)(m+2k+2)}{2\beta(2k+1)}\binom{k}{1}\frac{m+2k}{2\beta+2}\\
&\
+\frac{(m+1)(m+2\beta+2k+2)}{2\beta(2k+1)}\binom{k}{1}\frac{m+2+2k}{2\beta+2}
 -\frac{(m+2k+2)(2\beta+2k+2)(m+2\beta+2k+2)}{2\beta(2\beta+2)(2k+1)},
\end{align*}
for any $2\leq j\leq k$,
\begin{align*}
\mathcal{I}_j&:=\frac{(m+2\beta+1)(m+2k+2)}{2\beta(2k+1)}(-1)^j\binom{k}{j}
\displaystyle\prod^{j-1}_{l=0}\frac{m+2(k-l)}{2\beta+2+2l}\\
&\
+\frac{(m+2\beta+1)(m+2k+2)}{2\beta(2k+1)}
(-1)^{j-1}\binom{k}{j-1}
\displaystyle\prod^{j-2}_{l=0}\frac{m+2(k-l)}{2\beta+2+2l}\\
&\
-\frac{(m+1)(m+2\beta+2k+2)}{2\beta(2k+1)}(-1)^j\binom{k}{j}
\displaystyle\prod^{j-1}_{l=0}\frac{m+2+2(k-l)}{2\beta+2+2l}\\
&\
-\frac{(m+2k+2)(2\beta+2k+2)(m+2\beta+2k+2)}{2\beta(2\beta+2)(2k+1)}(-1)^{j-1}\binom{k}{j-1}
\displaystyle\prod^{j-2}_{l=0}\frac{m+2(k-l)}{2\beta+4+2l}
\end{align*}
and
\begin{align*}
\mathcal{I}_{k+1}&:=\frac{(m+2\beta+1)(m+2k+2)}{2\beta(2k+1)}(-1)^{k}\displaystyle\prod^{k-1}_{l=0}\frac{m+2(k-l)}{2\beta+2+2l}\\
&\ \ \
-\frac{(m+2k+2)(2\beta+2k+2)(m+2\beta+2k+2)}{2\beta(2\beta+2)(2k+1)}(-1)^{k}\displaystyle\prod^{k-1}_{l=0}\frac{m+2(k-l)}{2\beta+4+2l}.
\end{align*}
By straightforward computations, we have that
\begin{align*}
\mathcal{I}_0
&=\frac{(m+2\beta+1)((m+1)+(2k+1))-(m+1)((m+2\beta+1)+(2k+1))}{2\beta(2k+1)}\\
&=\frac{(m+2\beta+1)(2k+1)-(m+1)(2k+1)}{2\beta(2k+1)}=1,
\end{align*}
\begin{align*}
\mathcal{I}_1&=\frac{m+2k+2}{2\beta(2\beta+2)(2k+1)}
\Big[(m+2\beta+1)(2\beta+2)-(2\beta+2k+2)(m+2\beta+2k+2)\\
&\ \ \ \ \ \ \ \ \ \ \ \ \ \ \ \ \ \ \ \ \ \ \ \ \ \ \ \ \ \ \ \ \ \ \ \
+k(m+1)(m+2\beta+2k+2)-k(m+2k)(m+2\beta+1)\Big]\\
&=\frac{m+2k+2}{2\beta(2\beta+2)(2k+1)}\Big[-(2\beta+2)(2k+1)-2k(m+2\beta+1)-2k(2k+1)\\
&\ \ \ \ \ \ \ \ \ \ \ \ \ \ \ \ \ \ \ \ \ \ \ \ \ \ \ \ \ \ \ \ \ \ \ \
+k(m+1)(2k+1)-k(2k-1)(m+2\beta+1)\Big]\\
&=-\frac{m+2k+2}{2\beta(2\beta+2)(2k+1)}\cdot(k+1)(2\beta+2)(2k+1)\\
&=-\binom{k+1}{1}\frac{m+2k+2}{2\beta}
\end{align*}
and
\begin{align*}
\mathcal{I}_{k+1}
&=(-1)^{k}\frac{m+2k+2}{2\beta(2k+1)}\displaystyle\prod^{k}_{l=0}\frac{m+2(k+1-l)}{2\beta+2l}
\Big[(m+2\beta+1)\cdot\frac{2\beta}{m+2k+2}\\
&\ \ \ \ \ \ \ \ \
-\frac{(2\beta+2k+2)(m+2\beta+2k+2)}{2\beta+2}\cdot\frac{2\beta(2\beta+2)}{(2\beta+2k+2)(m+2k+2)} \Big]\\
&=\frac{(-1)^{k}}{2k+1}\displaystyle\prod^{k}_{l=0}\frac{m+2(k+1-l)}{2\beta+2l}
\Big[(m+2\beta+1)-(m+2\beta+2k+2) \Big]\\
&=(-1)^{k+1}\displaystyle\prod^{k}_{l=0}\frac{m+2(k+1-l)}{2\beta+2l}.
\end{align*}
Moreover, for any integer $2\leq j\leq k$,
\begin{align*}
\mathcal{I}_j
&=(-1)^j\displaystyle\prod^{j-1}_{l=0}\frac{m+2(k+1-l)}{2\beta+2l}
\Big[\frac{(m+2\beta+1)(m+2k+2)}{2\beta(2k+1)}\cdot\frac{2\beta(m+2k-2j+2)}{(2\beta+2j)(m+2k+2)}\binom{k}{j}\\
&\
-\frac{(m+2\beta+1)(m+2k+2)}{2\beta(2k+1)}\cdot\frac{2\beta}{m+2k+2}\binom{k}{j-1}
-\frac{(m+1)(m+2\beta+2k+2)}{2\beta(2k+1)}\cdot\frac{2\beta}{2\beta+2j}\binom{k}{j}\\
&\ \ \ \ \ \
+\frac{(m+2k+2)(2\beta+2k+2)(m+2\beta+2k+2)}{2\beta(2\beta+2)(2k+1)}
\cdot\frac{2\beta(2\beta+2)}{(2\beta+2j)(m+2k+2)}\binom{k}{j-1}\Big]\\
&=(-1)^j\displaystyle\prod^{j-1}_{l=0}\frac{m+2(k+1-l)}{2\beta+2l}
\Big[\frac{(m+2\beta+1)(m+2k-2j+2)}{(2k+1)(2\beta+2j)}\binom{k}{j}
-\frac{m+2\beta+1}{2k+1}\binom{k}{j-1}\\
&\ \ \ \ \ \ \ \
-\frac{(m+1)(m+2\beta+2k+2)}{(2k+1)(2\beta+2j)}\binom{k}{j}
+\frac{(2\beta+2k+2)(m+2\beta+2k+2)}{(2k+1)(2\beta+2j)}
\binom{k}{j-1}\Big]\\
&=(-1)^j\binom{k+1}{j}\displaystyle\prod^{j-1}_{l=0}\frac{m+2(k+1-l)}{2\beta+2l}\\
&\ \ \ \ \ \
\times
\Big\{\Big[\frac{(m+2\beta+1)(m+2k-2j+2)}{(2k+1)(2\beta+2j)}
-\frac{(m+1)(m+2\beta+2k+2)}{(2k+1)(2\beta+2j)}\Big]\Big(1-\frac{j}{k+1}\Big)\\
&\ \ \ \ \ \ \ \
+\Big[\frac{(2\beta+2k+2)(m+2\beta+2k+2)}{(2k+1)(2\beta+2j)}-\frac{m+2\beta+1}{2k+1}
\Big]\frac{j}{k+1}\Big\}\\
&=(-1)^j\binom{k+1}{j}\displaystyle\prod^{j-1}_{l=0}\frac{m+2(k+1-l)}{2\beta+2l}
\Big\{\frac{(1-\frac{j}{k+1})(2\beta(2k+1)-2(m+2\beta+1)j)}{(2k+1)(2\beta+2j)}\\
&\ \ \ \ \ \ \ \
+\frac{2(m+2\beta+2k+2)j+\frac{j}{k+1}(2\beta(2k+1)-2(m+2\beta+1)j)}{(2k+1)(2\beta+2j)}\Big\}\\
&=(-1)^j\binom{k+1}{j}\displaystyle\prod^{j-1}_{l=0}\frac{m+2(k+1-l)}{2\beta+2l}.
\end{align*}
The proof of Lemma \ref{Taylor-series-coefficients} is then complete.   \hfill\hfill$\square$\vskip12pt
\subsection{Logarithmic asymptotics of the singular integral $G(\lambda)$}
In this subsection, we prove the two limits in \eqref{singularity-order}.

\textbf{Proof of \eqref{singularity-order}.} It is sufficient to consider $\lambda>0$. The substitution $t=\sin\frac{\theta}{2}$, together with
\begin{align*}
1-2\lambda\cos\theta+\lambda^2=(\lambda-1)^2+4\lambda\sin^2\frac{\theta}{2},
\end{align*}
gives that
\begin{align}\label{G-representation}
G(\lambda)
=\int\limits^1_0\frac{t^n(1-t^2)^{\frac{n-1}{2}}}{\Big(t^2+\frac{(\lambda-1)^2}{4\lambda}\Big)^{\frac{n+1}{2}}}dt
=\lambda^{-\frac{n+1}{2}}H\Big(\frac{|\lambda-1|}{2\sqrt{\lambda}}\Big),
\end{align}
where the function $H(\cdot)$ is given by
\begin{align*}
H(s):=\int\limits^1_0\frac{t^n(1-t^2)^{\frac{n-1}{2}}}{(t^2+s^2)^{\frac{n+1}{2}}}dt.
\end{align*}
For every $s>0$, the denominator of the integrand in the definition of $H(\cdot)$ is locally bounded away from zero.  Thus, by differentiation under the integral sign and a change of variables, we obtain that
\begin{align*}
sH'(s)
&=-(n+1)s^2\int\limits^1_0\frac{t^n(1-t^2)^{\frac{n-1}{2}}}{(t^2+s^2)^{\frac{n+3}{2}}}dt=-(n+1)\int\limits^{\frac{1}{s}}_0\frac{x^n(1-s^2x^2)^{\frac{n-1}{2}}}{(1+x^2)^{\frac{n+3}{2}}}dx\\
&=-(n+1)\int\limits^\infty_0\frac{x^n(1-s^2x^2)^{\frac{n-1}{2}}}{(1+x^2)^{\frac{n+3}{2}}}\chi_{(0,\frac{1}{s})}(x) dx,
\end{align*}
where the indicator function $\chi$ is defined by
\begin{eqnarray*}
\chi_{(0,\frac{1}{s})}(x)
=
\begin{cases}
1,
&  \mbox{if $x\in(0,\frac{1}{s}),$ }\\
0,
& \mbox{if $x\in[\frac{1}{s},\infty).$ } \\
\end{cases}
\end{eqnarray*}
As $s\downarrow0$, the last integrand converges pointwise to $\frac{x^n}{(x^2+1)^{\frac{n+3}{2}}}$ and is bounded by the same integrable function. Then the Dominated Convergence Theorem yields that
\begin{align}\label{H-derivative-order}
\lim_{s\downarrow0}sH'(s)
=-(n+1)\int\limits^\infty_0\frac{x^n}{(1+x^2)^{\frac{n+3}{2}}}dx
=-\frac{x^{n+1}}{(1+x^2)^{\frac{n+1}{2}}}\Big|^\infty_0=1.
\end{align}
Consequently, by l'H\^{o}pital's rule, we obtain that
\begin{align*}
\lim_{s\downarrow0}\frac{H(s)}{\ln s}=\lim_{s\downarrow0}sH'(s)=-1,
\end{align*}
which implies that
\begin{align*}
\lim_{\lambda\rightarrow1}\frac{G(\lambda)}{\ln|\lambda-1|}
=\lim_{\lambda\rightarrow1}\lambda^{-\frac{n+1}{2}}
\cdot\frac{H\Big(\frac{|\lambda-1|}{2\sqrt{\lambda}}\Big)}{\ln\Big(\frac{|\lambda-1|}{2\sqrt{\lambda}}\Big)}
\cdot\Big(1-\frac{\ln2+\frac12\ln\lambda}{\ln|\lambda-1|}\Big)
=-1.
\end{align*}
Differentiating \eqref{G-representation}, we can obtain that
\begin{align*}
G'(\lambda)
&=-\frac{n+1}{2}\lambda^{-\frac{n+3}{2}}H\Big(\frac{|\lambda-1|}{2\sqrt{\lambda}}\Big)
+\lambda^{-\frac{n+1}{2}}H'\Big(\frac{|\lambda-1|}{2\sqrt{\lambda}}\Big)\frac{d}{d\lambda}\Big(\frac{|\lambda-1|}{2\sqrt{\lambda}}\Big)\\
&=-\frac{n+1}{2}\lambda^{-\frac{n+3}{2}}H\Big(\frac{|\lambda-1|}{2\sqrt{\lambda}}\Big)
+\lambda^{-\frac{n+1}{2}}H'\Big(\frac{|\lambda-1|}{2\sqrt{\lambda}}\Big)\cdot\frac{|\lambda-1|}{2\sqrt{\lambda}}\Big(\frac{1}{\lambda-1}-\frac{1}{2\lambda}\Big)\\
&=\lambda^{-\frac{n+1}{2}}\Big[-\frac{n+1}{2}\lambda^{-1}H\Big(\frac{|\lambda-1|}{2\sqrt{\lambda}}\Big)
+\frac{|\lambda-1|}{2\sqrt{\lambda}}H'\Big(\frac{|\lambda-1|}{2\sqrt{\lambda}}\Big)\frac{\lambda+1}{2\lambda(\lambda-1)}\Big],
\end{align*}
which together with \eqref{H-derivative-order} implies that
\begin{align*}
\lim_{\lambda\rightarrow1}(\lambda-1)G'(\lambda)
&=\lim_{\lambda\rightarrow1}\lambda^{-\frac{n+1}{2}}\Big[-\frac{n+1}{2}\frac{(\lambda-1)}{\lambda}H\Big(\frac{|\lambda-1|}{2\sqrt{\lambda}}\Big)
+\frac{|\lambda-1|}{2\sqrt{\lambda}}H'\Big(\frac{|\lambda-1|}{2\sqrt{\lambda}}\Big)\frac{\lambda+1}{2\lambda}\Big]\\
&=-\frac{n+1}{2}\lim_{\lambda\rightarrow1}(\lambda-1)H\Big(\frac{|\lambda-1|}{2\sqrt{\lambda}}\Big)
+\lim_{\lambda\rightarrow1}\frac{|\lambda-1|}{2\sqrt{\lambda}}H'\Big(\frac{|\lambda-1|}{2\sqrt{\lambda}}\Big)\\
&=\frac{n+1}{2}\lim_{\lambda\rightarrow1}(\lambda-1)\ln\frac{|\lambda-1|}{2\sqrt{\lambda}}
+(-1)\\
&=-1.
\end{align*}
We then finish the proof of \eqref{singularity-order}.   \hfill\hfill$\square$\vskip12pt
\subsection{Proof of Lemma \ref{Bernstein-Lemma}}
\textbf{Proof.} Denote $K:=k+\frac{\alpha+\beta+1}{2}$ and make a change of variables
\begin{align*}
x:=\cos\theta,~\theta\in(0,\pi).
\end{align*}
Note that
\begin{align*}
1-x=2\sin^2\frac{\theta}{2},~~~~~1+x=2\cos^2\frac{\theta}{2},
\end{align*}
it is enough to show that, for every integer $k\geq1$ and any $\theta\in(0,\pi)$,
\begin{align}\label{equivalent-form}
\Big(\sin\frac{\theta}{2}\Big)^{\alpha+\frac12}\Big(\cos\frac{\theta}{2}\Big)^{\beta+\frac12}|P^{(\alpha,\beta)}_k(\cos\theta)|\leq \frac{C_{\alpha,\beta}}{\sqrt{k}}.
\end{align}
We first consider the case $\theta\in(0,\frac{\pi}{2}]$. According to the classical Hilb formula for the Jacobi polynomials (see Theorem 8.21.12 in \cite{[Szego]}), we have that
\begin{align*}
\Big(\sin\frac{\theta}{2}\Big)^{\alpha}\Big(\cos\frac{\theta}{2}\Big)^{\beta}|P^{(\alpha,\beta)}_k(\cos\theta)| =\frac{\Gamma(k+\alpha+1)}{K^\alpha k!}\Big(\frac{\theta}{\sin\theta}\Big)^{\frac12}J_\alpha(K\theta)+R_k(\theta),
\end{align*}
where $J_\alpha$ is the Bessel function of the first kind and uniformly for $\theta\in(0,\frac{\pi}{2}]$,
\begin{align}\label{remainder-estimate}
|R_k(\theta)|\leq C_{\alpha,\beta}\theta^{\frac12}k^{-\frac{3}{2}}.
\end{align}
It follows that
\begin{align}\label{Hilb-estimate}
\Big(\sin\frac{\theta}{2}\Big)^{\alpha+\frac12}\Big(\cos\frac{\theta}{2}\Big)^{\beta+\frac12}|P^{(\alpha,\beta)}_k(\cos\theta)| =\frac{\Gamma(k+\alpha+1)}{\sqrt{2}K^\alpha k!}\theta^{\frac12}J_\alpha(K\theta)+\Big(\frac12\sin\theta\Big)^{\frac12}R_k(\theta).
\end{align}
By the standard Gamma-ratio estimate, we have
\begin{align*}
\lim_{k\rightarrow\infty}\frac{\Gamma(k+\alpha+1)}{k^\alpha k!}=1,
\end{align*}
which shows that
that, for any integer $k\geq1$,
\begin{align}\label{constant-estimate}
\frac{\Gamma(k+\alpha+1)}{\sqrt{2}K^\alpha k!}\leq C_{\alpha,\beta}.
\end{align}
For fixed $\alpha\geq0$, the Bessel function satisfies
\begin{eqnarray}\label{Bessel-estimate}
|J_\alpha(t)|
\leq C_\alpha
\begin{cases}
t^\alpha,
&  \mbox{ $0\leq t<1,$ }\\
t^{-\frac12},
& \mbox{ $t\geq1.$ } \\
\end{cases}
\end{eqnarray}
For $0<\theta<\frac{1}{K}$, that is, $0<K\theta<1$. By the first part of \eqref{Bessel-estimate}, we have that
\begin{align*}
|\theta^{\frac12}J_\alpha(K\theta)|
\leq C_\alpha\theta^{\frac12}(K\theta)^\alpha=C_\alpha K^\alpha\theta^{\alpha+\frac12}
\leq C_\alpha K^\alpha(K^{-1})^{\alpha+\frac12}=\frac{C_\alpha}{\sqrt{K}}
\leq \frac{C_\alpha}{\sqrt{k}}.
\end{align*}
For $\frac{1}{K}\leq\theta\leq\frac{\pi}{2}$, here $K\theta\geq1$. By the second part of \eqref{Bessel-estimate}, we derive that
\begin{align*}
|\theta^{\frac12}J_\alpha(K\theta)|
&\leq C_\alpha\theta^{\frac12}(K\theta)^{-\frac12}
=\frac{C_\alpha}{\sqrt{K}}\leq \frac{C_\alpha}{\sqrt{k}}.
\end{align*}
Thus, we obtain that, for any $\theta\in(0,\frac{\pi}{2}]$,
\begin{align}\label{leading-term-estimate}
|\theta^{\frac12}J_\alpha(K\theta)|
\leq  \frac{C_\alpha}{\sqrt{k}}.
\end{align}
For any $\theta\in(0,\frac{\pi}{2}]$, by virtue of \eqref{remainder-estimate}, we derive that
\begin{align*}
\Big|\Big(\frac12\sin\theta\Big)^{\frac12}R_k(\theta)\Big|
\leq C_{\alpha,\beta}\theta^{\frac12}k^{-\frac{3}{2}}
\leq\frac{C_{\alpha,\beta}}{\sqrt{k}},
\end{align*}
which along with \eqref{Hilb-estimate}, \eqref{constant-estimate} and \eqref{leading-term-estimate} gives \eqref{equivalent-form} with $\theta\in(0,\frac{\pi}{2}]$. The Hilb expansion is ordinarily stated for all sufficiently large $k$. The finitely many remaining values of $k$ cause no difficulty: for each such $k$, the function on the left side of \eqref{equivalent-form} extends continuously to $[0,\frac{\pi}{2}]$, and its supremum can be absorbed into the constant $C_{\alpha,\beta}$.

It remains to treat the case $\frac{\pi}{2}\leq\theta<\pi$. Put a change of variables $\varphi:=\pi-\theta\in(0,\frac{\pi}{2}]$. The symmetry identity for Jacobi polynomials
\begin{align*}
P^{(\alpha,\beta)}_k(-t)=(-1)^kP^{(\beta,\alpha)}_k(t)
\end{align*}
and the relations
\begin{align*}
\sin\frac{\theta}{2}=\cos\frac{\varphi}{2},~\cos\frac{\theta}{2}=\sin\frac{\varphi}{2}
\end{align*}
show that
\begin{align*}
\Big(\sin\frac{\theta}{2}\Big)^{\alpha+\frac12}\Big(\cos\frac{\theta}{2}\Big)^{\beta+\frac12}|P^{(\alpha,\beta)}_k(\cos\theta)|
=\Big(\sin\frac{\varphi}{2}\Big)^{\beta+\frac12}\Big(\cos\frac{\varphi}{2}\Big)^{\alpha+\frac12}|P^{(\beta,\alpha)}_k(\cos\varphi)|.
\end{align*}
The estimate already proved, now with $\alpha$ and $\beta$ interchanged, bounds the last expression by $\frac{C_{\alpha,\beta}}{\sqrt{k}}$. Hence the inequality \eqref{equivalent-form} holds for all $0<\theta<\pi$.
The proof of Lemma \ref{Bernstein-Lemma} is finished.
\hfill\hfill$\square$\vskip12pt
\subsection*{Acknowledgment}
\addcontentsline{toc}{section}{Acknowledgments} \hskip\parindent
W. Zhang was partially supported by the National Natural Science Foundation of China (Grant Nos. 12501301, 12671256).

\end{document}